\documentclass[reqno, 11pt, twoside]{amsart}

\usepackage[margin=3.5cm]{geometry}

\usepackage{amsmath,amssymb,amsthm,bm,colonequals,graphicx,mathrsfs,mathtools,microtype,stmaryrd, amsfonts, dsfont,physics, comment, multicol}
\usepackage[shortlabels]{enumitem}
\usepackage[colorinlistoftodos]{todonotes}

\numberwithin{equation}{section}

\theoremstyle{plain}

\newtheorem{theorem}{Theorem}[section] 

\newtheorem{lemma}[theorem]{Lemma} 

\newtheorem{proposition}[theorem]{Proposition} 

\newtheorem{proposition-definition}[theorem]{Proposition-Definition} 

\newtheorem{corollary}[theorem]{Corollary} 

\newtheorem{conjecture}[theorem]{Conjecture} 

\theoremstyle{definition}

\newtheorem{definition}[theorem]{Definition} 

\newtheorem{hypothesis}[theorem]{Hypothesis} 

\newtheorem{example}[theorem]{Example}

\theoremstyle{remark}

\renewcommand{\abs}[1]{\lvert #1 \rvert} 
\newcommand{\card}[1]{\#{ #1 }}
\renewcommand{\norm}[1]{\lvert #1 \rvert} 
\newcommand{\floor}[1]{\lfloor #1 \rfloor}
\newcommand{\ceil}[1]{\lceil #1 \rceil}

\newcommand{\belongs}{\subseteq}

\newcommand{\eps}{\epsilon}
\newcommand{\EE}{\mathbb{E}}
\newcommand{\defeq}{\colonequals}

\newcommand{\maps}{\colon}
\newcommand{\map}{\operatorname}
\newcommand{\mscr}{\mathscr}
\newcommand{\mcal}{\mathcal}
\newcommand{\set}[1]{\{#1\}}
\newcommand{\ts}{\tau}

\newcommand{\rad}{\operatorname{rad}}

\newcommand{\NN}{\mathbb{N}}
\newcommand{\ZZ}{\mathbb{Z}}
\newcommand{\QQ}{\mathbb{Q}}
\newcommand{\CC}{\mathbb{C}}

\newcommand{\FF}{\mathbb{F}}

\newcommand{\PP}{\mathbb{P}}

\DeclareMathOperator{\Spec}{Spec}

\DeclareMathOperator{\supp}{supp}

\DeclareMathOperator{\Hom}{Hom}

\DeclareMathOperator{\Gal}{Gal}

\DeclareMathOperator{\GL}{GL}

\DeclareMathOperator{\ord}{ord}
\DeclareMathOperator{\cha}{char}
\DeclareMathOperator{\disc}{disc}

\DeclareMathOperator{\rk}{rank}

\DeclareMathOperator{\sq}{sq}
\def\OK{\mathcal{O}}
\def\Omon{\OK^+}

\def\RR{\mathbb{R}}
\def\TT{\mathbb{T}}

\newcommand{\RcG}{\mathcal{R}_{\bm{c}}^{\map{G}}} 
\newcommand{\RcB}{\mathcal{R}_{\bm{c}}^{\map{B}}} 

\renewcommand{\hat}{\widehat}
\newcommand{\cc}{{\bm{c}}}
\DeclareMathOperator{\cub}{cub}
\DeclareMathOperator{\Char}{char}
\renewcommand{\leq}{\leqslant}
\renewcommand{\le}{\leqslant}
\renewcommand{\geq}{\geqslant}
\renewcommand{\ge}{\geqslant}
\newcommand{\ve}{\varepsilon}
\renewcommand{\eps}{\varepsilon}

\usepackage[pagebackref=true]{hyperref}
\usepackage[alphabetic,initials]{amsrefs}

\title{Sums of three cubes over a function field}

\author{Tim Browning, Jakob Glas, Victor Y. Wang}

\address{IST Austria\\
Am Campus 1\\
3400 Klosterneuburg\\
Austria}
\email{tdb@ist.ac.at, jakob.glas@ist.ac.at, victor.wang@ist.ac.at}


\date{\today}

\begin{document}

\begin{abstract}
We use a function field version of the circle method to prove that a positive proportion of elements in $\FF_q[t]$ are representable as a sum of three cubes
of minimal degree
from $\FF_q[t]$, assuming a
suitable  form of the Ratios Conjecture and 
that $\Char(\FF_q)>3$.
The analogue of this conjecture
for quadratic Dirichlet $L$-functions
is known
for large fixed $q$,
via recent developments in homological stability.
\end{abstract}

\maketitle

\thispagestyle{empty}
\setcounter{tocdepth}{1}
{\small
\begin{multicols}{2}
\tableofcontents
\end{multicols}
}

\section{Introduction}

A folklore  conjecture in number theory states that any $k\in \ZZ$, not congruent to $\pm4$ modulo $9$, 
admits a representation 
as a sum of three cubes of integers.  The state of the art around this difficult problem is summarised 
in recent work by Wang \cite{wang2023ratios}, where a modern form of the circle method is developed to prove the conjecture for a positive proportion of integers $k$. Wang's work is highly conditional, depending on an automorphy conjecture and GRH for certain Hasse--Weil $L$-functions, as well as the Ratios Conjecture from random matrix theory and a Square-Free Sieve Conjecture.

In this paper we investigate the analogous problem over the function field $K=\FF_q(t)$,
with the aim of removing as many of the hypotheses as possible from \cite{wang2023ratios}. 
Let 
$\OK=\FF_q[t]$ denote the ring of integers of $K$ and let $k\in \OK$.
In stark contrast to the case of $\ZZ$,
when $\cha(\FF_q)\ne 3$ and $q\notin \{2,4,7,13,16\}$,
Serre and Vaserstein \cite{vaserstein1991sums}*{Lemma~1}
have explicitly constructed \emph{linear functions} $a_ik+b_i$, with $a_i\in \FF_q^\times$ and $b_i\in \FF_q$, such that
\begin{equation*}
\sum_{1\le i\le 3} (a_ik+b_i)^3 = k.
\end{equation*}
Vaserstein also gives more complicated polynomial constructions if $q\in \{7,13\}$.
In each case, the degrees of $x,y,z\in \OK$ solving $x^3+y^3+z^3 = k$
are $\ge \deg{k}$, whereas one might hope for solutions of degree
$\sim \tfrac13 \deg{k}$.
Also, \cite{vaserstein1991sums} only constructs a bounded finite number of solutions for each $k$,
and thus it remains open to produce an infinitude of solutions, or even a number of solutions tending to infinity with $\deg{k}$.

According to \cite{gallardo2013constants},
when $q=16$ there is no known  polynomial construction, 
but a direct adaptation of \cite{colliot2012groupe}*{Lemme~4.5} shows that there are no local obstructions
to the representation of  $k$ as a sum of three cubes from $\OK$ and so it is natural to conjecture that most, or perhaps all, $k\in \OK$ admit such a representation.
When $q\in \{2,4\}$, there are local obstructions at places of residue cardinality $4$, because all cubes in $\FF_4$ lie in the proper subfield $\FF_2$;
but most, or perhaps all, $k\in \OK$ may still satisfy the Hasse principle,
with a positive proportion then being representable.
If $\cha(\FF_q)=3$, then $\set{x^3+y^3+z^3: x,y,z\in \OK}$ has density $0$ in $\OK$, however, because $x^3+y^3+z^3 = (x+y+z)^3$.

To each smooth hyperplane section $c_1x_1+\dots+c_6x_6=0$ of the hypersurface 
$$x_1^3+\dots+x_6^3 = 0$$ 
in $\PP^5_K$, for $\bm{c}=(c_1,\dots,c_6)\in \OK^6$, 
we shall see in \S~\ref{SEC:ratios-plus} that we can associate a Hasse--Weil $L$-function $L(s,\bm{c})$ over $K$.
Thanks to Grothendieck  and Deligne \cites{weil2, bourbaki}, 
we know that these  $L$-functions are actually rational functions of $q^{-s}$ 
that satisfy the Grand Riemann hypothesis (GRH).
Our main result relies on mean-value statistics of the ratio $1/L(s,\bm{c})$ over boxes of vectors $\bm{c}\in \OK^6$;
we need an asymptotic formula for
\begin{equation*}
\sum_{\substack{\bm{c}\in \OK^6 \\
\deg(c_1),\dots,\deg(c_6)\le Z}}\,
\frac{1}{L(s_1,\bm{c}) L(s_2,\bm{c})},
\end{equation*}
as $Z\to \infty$.
This is given by Conjecture~\ref{CNJ:(R2o)}~(R2),
a special case of the Ratios Conjecture.
The following is our main result.

\begin{theorem}
\label{THM:pos-density}
Suppose $\cha(\FF_q) > 3$ and 
assume the Ratios Conjecture~\ref{CNJ:(R2o)}
for $L(s,\bm{c})$, as $\bm{c}\in \OK^6$ varies.
Then each of the following sets has positive lower density in $\OK$.
\begin{enumerate}
\item $\set{x^3+y^3+z^3: x,y,z\in \OK\textnormal{ monic}}$.

\item $\set{k: x^3+y^3+z^3=k\textnormal{ is soluble in $\OK$ with $\max\{\deg{x},\deg{y},\deg{z}\} \le \ceil{\frac{\deg{k}}{3}}$}}$.
\end{enumerate}
\end{theorem}

The set in (2) is known in Waring's problem
as the set of \emph{strict sums of three cubes},
in the sense of Carlitz \cite{gallardo2008strict}*{p.~2964}.
We have already seen that the conclusions of the theorem are false if 
$\cha(\FF_q)=3$, where the sets in (1) and (2) have density $0$.
If $\cha(\FF_q)=2$, then most of the proof of Theorem~\ref{THM:pos-density} still goes through,
but some of the ingredients would need to be modified. We shall comment further on this in 
\S~\ref{SEC:final-remarks}.

The Moments and Ratios Conjectures were originally stated for automorphic $L$-functions over $\QQ$ by Conrey, Farmer, et al.~\cites{conrey2005integral,conrey2008autocorrelation}.
Their natural extension to $\FF_q(t)$  has been explained by Andrade and Keating \cite{andrade2014conjectures}. The precise form of the Ratios Conjecture we need will be presented in \S~\ref{SEC:ratios-plus}. 
It is plausible that the homological stability framework of 
Bergstr{\"o}m,  Diaconu, Petersen and Westerland 
\cite{bergstrom2023hyperelliptic}, together with  Miller, Patzt, Petersen and Randal-Williams
\cite{MPPRW},
could eventually resolve
the particular form of the Ratios Conjecture needed in the present paper, for 
sufficiently large values of $q$.
We defer the details of this to \S~\ref{SEC:detailed-ratio-reductions},
where we shall define a \emph{$q$-restricted} form of the Ratios Conjecture, 
together with a strengthening of 
Theorem \ref{THM:pos-density}, in 
the shape of Theorem \ref{THM:first-goal}.
Still working under the Ratios Conjecture, it is likely that one can show that the set
\begin{equation*}
\set{k: x^3+y^3+z^3=k\textnormal{ is soluble in $\OK$ with $\max\{\deg{x},\deg{y},\deg{z}\} \le \tfrac{\deg{k}}{3} + A$}}
\end{equation*}
has lower density approaching $1$ as $A\to \infty$, by adapting the proof of \cite{wang2023ratios}*{Theorem 1.6}.
Moreover, one can likely show that this set has upper density $< 1$ for any fixed $A\in \RR$,
by adapting local density arguments of Diaconu from \cite{diaconu2019admissible}*{\S~1}.
We defer both tasks to  future work, in order to keep the present paper as 
clean as possible.

The Ratios Conjecture has already found application in the theory of {\em rational points} on elliptic curves over $K$.  Thus, through the agency of Katz--Sarnak \cite{katz1999zeroes}, and Conrey et al.~\cites{conrey2008autocorrelation,conrey2007applications},  it has been shown that
Goldfeld's Conjecture 
follows from  the Ratios Conjecture for a suitable family of elliptic curve $L$-functions over $K$.
This conjecture states that 
there is a density $\frac12$ 
of  curves with prescribed analytic rank $r\in \set{0,1}$, and 
therefore implies the Birch and Swinnerton-Dyer Conjecture (BSD) for a density $1$ of  curves.
Whereas $L$-function ratios are directly related to analytic ranks and BSD,
their connection to {\em integral points} and sums of three cubes in Theorem \ref{THM:pos-density} is much less direct.

\medskip
	 
Our proof of 
Theorem \ref{THM:pos-density} uses the second moment method and requires a detailed analysis of   the counting function 
\begin{equation}\label{eq:mint-tea}
N(P)=\#\left\{ \bm{x}\in \OK^6:  x_1^3+\cdots+x_6^3=0, ~|x_1|=\cdots=|x_6|= |P|\right\},
\end{equation}
for given   $P\in \OK$, 
where $|\cdot|$ denote the usual absolute value on $K$.
We shall be interested in the size of $N(P)$ as $|P|\to \infty$. 
The Batyrev--Manin conjecture suggests that an asymptotic formula of the shape 
$N(P)\sim c|P|^3$ should hold, for a suitable constant $c>0$. The best progress towards this is to be found in 
recent work of Glas and Hochfilzer 
\cite{glas2022question}, who use 
a function field version of the circle method to prove that 
\begin{equation}\label{eq:eps}
N(P)=O_\ve( |P|^{3+\ve}),
\end{equation}
for any $\ve>0$, assuming only that $\cha(\FF_q)\neq 3$.
 It is worth emphasising that this bound is completely unconditional, unlike the parallel picture over $\ZZ$, where the pioneering work of Hooley \cites{hooley1986Lfunctions,hooley_greaves_harman_huxley_1997} and Heath-Brown \cite{heath1998circle} is  conditional on GRH.
Assuming that $\cha(\FF_q)>3$, our  proof of Theorem \ref{THM:pos-density} entirely rests on our ability to remove the $\ve$ from the upper bound \eqref{eq:eps}. Ultimately we shall only manage to do so at the cost of assuming the Ratios Conjecture for the Hasse--Weil $L$-functions $L(s,\bm{c})$, as $\bm{c}\in \OK^6$ varies.

Our argument will broadly follow the path laid down by Wang \cite{wang2023ratios} over $\ZZ$, but with a variety  of differences and refinements, some of which we have chosen to highlight here. 
\begin{enumerate}
\item 
The $L$-functions $L(s,\bm c)$ carry less information over $\FF_q(t)$ than over $\QQ$, since  $q$-adic sums of local coefficients are collapsed into single global coefficients. Nonetheless, these are still the only $L$-functions that we need in our analysis,
provided we choose our weight functions carefully, as suggested in the next item of this list.

\item
In order to avoid introducing character twists as in \cite{browning2015rational}, 
which would enlarge and complicate the form of the Ratios Conjecture needed, we restrict the class of weight functions that we use. This allows us to cleanly factor out the relevant oscillatory integral using symmetry ideas of Glas--Hochfilzer  \cite{glas2022question}.

\item
The precise automorphy hypotheses and Square-Free Sieve Conjecture 
required in  \cite{wang2023ratios} 
 become unconditional over $\FF_q(t)$.
 For the former, we use Poincar\'{e} duality to carefully produce poles of exterior square $L$-functions.
 For the latter, we use an argument of Poonen 
 \cite{poonen2003squarefree} on square-free values of  multivariable polynomials.

\item
In 
Theorem \ref{THM:weird-eke} we prove a simple version of the 
Ekedahl sieve in positive characteristic for square-free moduli, which is sufficient for the applications in this paper.
(A general prime moduli version  for arbitrary global fields has been worked out in forthcoming work of 
Bhargava--Shankar--Wang  \cite{bhargava2023coregular}*{Theorem 16}.)

\item
The bias in exponential sums that was discovered in
\cite{wang2023special}*{Lemma 7.7}
was 
only worked out for diagonal cubic forms in six variables; in Proposition 
\ref{Prop: LinearSpaceBias} we extend the argument to handle arbitrary non-singular senary  cubic forms.
This requires new geometric insight on quadric bundles,
via work of Beauville \cite{BeauvilleQuadricBundles}.
\end{enumerate}

\subsection*{Acknowledgements}

Thanks are due to Dan Petersen and Peter Sarnak for discussions on homological stability,
and to Trevor Wooley for providing the reference \cite{vaserstein1991sums}.
While working on this paper the first two authors were supported by FWF grant P 36278 and the third  author was supported by the European Union's Horizon 2020 research and innovation programme under the Marie Sk\l{}odowska-Curie Grant Agreement No.~101034413.

\section{Background}

In this section we collect some basic facts about function fields.
Let $K_\infty=\FF_q((t^{-1}))$ be the field of Laurent series in $t^{-1}$.
Let $\Omon$ be the set of monic polynomials $r\in \OK$; this is analogous to the set of positive integers in $\ZZ$.
For $M\in\RR$, we shall write $\widehat{M}\coloneqq q^M$. Any $\alpha \in K_\infty\setminus\{0\}$ can be written uniquely as 
\begin{equation}\label{Eq: LaurentSeriesRepresentation}
\alpha = \sum_{i\leq M}a_it^i,\quad a_M\neq 0,
\end{equation}
for some $M\in \ZZ$ and $a_i\in\FF_q$.
Define $|\alpha|\coloneqq \widehat{M}$;
then $|\cdot|$ naturally extends to $K_\infty$ the absolute value on $K$ induced by $t^{-1}$.
Moreover, $K_\infty$ is the completion of $K$ with respect to this absolute value.
The analogue of the unit interval in $K_\infty$ is given by 
\[
\TT\coloneqq \{\alpha\in K_\infty\colon |\alpha|<1\}.
\]
Since $K_\infty$ is a local field, it can be endowed with a unique Haar measure $\dd \alpha $ such that $\int_{\TT}\dd\alpha=1$.
We shall extend the absolute value to $K^n_\infty$ by $|\bm{\alpha}|\defeq \max_{1\leq i\leq n}|\alpha_i|$ and the Haar measure by $\dd \bm{\alpha}\defeq \dd \alpha_1 \cdots \dd\alpha_n$, for $\bm{\alpha}=(\alpha_1,\dots, \alpha_n)\in K^n_\infty$.
Finally, for each prime $\varpi\in \Omon$, we have an  associated absolute value $|\cdot|_\varpi$  on $K$, given by $|a|_\varpi=q^{-v_\varpi(a)\deg\varpi}$. We denote by $K_\varpi$ the completion of $K$ with respect to this absolute value and we write $\OK_\varpi\coloneqq \{\alpha\in K_\varpi \colon |\alpha|_\varpi \le 1\}.$

\subsection*{Farey dissection}

Dirichlet's approximation theorem holds over $K$.
That is, for any $\alpha \in \TT$ and $Q \in \mathbb{N}\defeq \ZZ_{\ge 1}$,
there exist $a\in \OK$ and $r\in \Omon$
with $\gcd(a,r)=1$ and $|a|<|r|\leq \widehat{Q}$
such that $|r \alpha-a|<\widehat{Q}^{-1}$.
By the ultrametric property, this is enough to obtain  an analogue of a \emph{Farey dissection} of the unit interval
\begin{equation}\label{Eq: DirichletDissection}
\TT=\bigsqcup_{\substack{r\in \Omon \\ |r|\leq \widehat{Q}}}\,
\bigsqcup_{\substack{|a|<|r|\\\gcd(a,r)=1}}\,
\{\alpha\in\TT\colon |r \alpha-a|<\widehat{Q}^{-1}\},
\end{equation}
for any $Q\geq 1$.

\subsection*{Characters}

For $\alpha \in K_\infty$ given by~\eqref{Eq: LaurentSeriesRepresentation}, we define
\[
\psi\colon K_\infty\to\CC^\times, \quad \psi(\alpha)=e\left(\frac{\Tr_{\FF_q/\FF_p}(a_{-1})}{p}\right)
\quad\textnormal{(with $p=\cha(\FF_q)$)}
\]
and set $\psi(0)=1$, where as usual $e(x)\defeq \exp(2\pi i x)$ for $x \in \mathbb{R}$.
It is easy to see that $\psi$ is a non-trivial additive character of $K_\infty$,
and that for $x\in K_\infty$ and $N\in\ZZ_{\geq 0}$, we have
\begin{equation}\label{Eq: OrthogonalityOfCharacters}
    \int_{|\alpha|<\widehat{N}^{-1}}\psi(\alpha x)\dd\alpha=\begin{cases}\widehat{N}^{-1} &\text{if }|x|<\widehat{N},\\ 0 &\text{otherwise.}\end{cases}
\end{equation}
Note that if $x \in \OK$, then \eqref{Eq: OrthogonalityOfCharacters} implies
\[
  \int_{\mathbb{T}}\psi(\alpha x)\dd\alpha=\begin{cases} 1 &\text{if }x=0,\\ 0
  &\text{otherwise.}\end{cases}
\]

In addition, we will make frequent use of the following formula for exponential sums. If $r,a \in \OK$ are such that $r \neq 0$, then
\[
\frac{1}{\lvert r \rvert} \sum_{\lvert x \rvert < \lvert r \rvert} \psi \left( \frac{ax}{r} \right) = \begin{cases}
1 \quad &\text{if $r \mid a$,} \\
0 &\text{otherwise.}
\end{cases}
\]

\subsection*{Poisson summation}

We call a function $w \colon K_\infty^n \rightarrow \mathbb{C}$ \emph{smooth} if it is locally constant. Denote by $S(K_\infty^n)$ the space of all smooth functions $w \colon K_\infty^n \rightarrow \mathbb{C}$ with compact support. For such functions the Poisson summation formula \cite{browning2015rational}*{Lemma 2.1} holds. 
\begin{lemma}\label{lem.poisson summation}
    Let $f \in K_\infty[x_1, \hdots, x_n]$ and let $w \in S(K_\infty^n)$. Then  we have
$$
        \sum_{\bm{z} \in \OK^n} w(\bm{z}) \psi(f(\bm{z})) = \sum_{\bm{c} \in \OK^n} \int_{K_\infty^n} w(\bm{u}) \psi(f(\bm{u}) + \bm{c} \cdot \bm{u}) \dd \bm{u}.
$$
\end{lemma}

\subsection*{Delta method}

Given $F\in\OK[x_1,\dots,x_n]$ and a weight function $w\in S(K^n_\infty)$, we are interested in the counting function 
\begin{equation}\label{eq:green-tea}
    N_F(w,P)=\sum_{\substack{\bm{x}\in \OK^n\\ F(\bm{x})=0}}w\left(\frac{\bm{x}}{P}\right).
\end{equation}
For a parameter $Q\geq 1$ to be specified later, we deduce from~\eqref{Eq: DirichletDissection} and ~\eqref{Eq: OrthogonalityOfCharacters} that
\[
N_F(w,P)=\sum_{\substack{r\in \Omon\\ |r|\leq \widehat{Q}}}\,
\sideset{}{'}\sum_{|a|<|r|}\int_{|\theta|<|r|^{-1}\widehat{Q}^{-1}}
S(a/r+\theta)\dd \theta,
\]
where $\sum'_{|a|<|r|} $ means that we sum over $a\in\OK$ with $\gcd(a,r)=1$ only,
and where
\[
S(\alpha)=\sum_{\substack{\bm{x}\in\OK^n}}\psi(\alpha F(\bm{x}))w(\bm{x}/P)
\]
for $\alpha \in \TT$. As explained in~\cite{browning2015rational}*{\S~4}, we can use Poisson summation in Lemma~\ref{lem.poisson summation} 
to evaluate $S(\theta+a/r)$, giving 
\begin{equation}\label{Eq: DeltaMethod'}
    N_F(w,P)=|P|^n\sum_{\substack{r \in \Omon\\ |r|\leq \widehat{Q}}}|r|^{-n}\int_{|\theta|<|r|^{-1}\widehat{Q}^{-1}}\sum_{\bm{c}\in \OK^n}S_r(\bm{c})I_r(\theta,\bm{c})\dd\theta,
\end{equation}
where 
\begin{equation}\label{Eq: Definition S_r(c)}
    S_r(\bm{c})\coloneqq \sideset{}{'}\sum_{|a|<|r|}\sum_{|\bm{x}|<|r|}\psi\left(\frac{aF(\bm{x})-\bm{c}\cdot \bm{x}}{r}\right)
\end{equation}
and 
\begin{equation}\label{Eq: Definition I_r(theta,c)}
    I_r(\theta,\bm{c})\coloneqq \int_{K_\infty^n}w(\bm{x})\psi\left(\theta P^3F(\bm{x})+\frac{P \bm{c}\cdot \bm{x}}{r}\right)\dd \bm{x}.
\end{equation}
Moreover, we will also consider averages of $I_r(\theta,\bm{c})$ of the form
\begin{equation}\label{eq:2-1 to brighton}
    I_r(\bm{c})\coloneqq \int_{|\theta|<|r|^{-1}\widehat{Q}^{-1}}I_r(\theta,\bm{c})\dd\theta,
\end{equation}
in which case we may write
\begin{equation}\label{Eq: DeltaMethod}
    N_F(w,P)=|P|^n\sum_{\substack{r\in\Omon\\ |r|\leq \widehat{Q}}}|r|^{-n}\sum_{\bm{c}\in \OK^n}S_r(\bm{c})I_r(\bm{c}).
\end{equation}
Ultimately   $Q$ will be chosen so that $|P|^{3/2}\asymp \widehat{Q}$.

The expression~\eqref{Eq: DeltaMethod'} is the starting point for our work and from now on we will mostly be concerned with estimating the integrals $I_r(\theta,\bm{c})$ and the sums $S_r(\bm{c})$. To keep notation simple, it is convenient to introduce the normalised sum 
\[
S_r^\natural(\bm{c})\coloneqq |r|^{-(n+1)/2}S_r(\bm{c}).
\]
If $r=r_1r_2$, where both $r_1,r_2\in\OK$ are monic with $\gcd(r_1,r_2)=1$, the Chinese remainder theorem readily implies that 
\begin{equation}\label{Eq: Sr(c)multiplicative}
    S_r(\bm{c})=S_{r_1}(\bm{c})S_{r_2}(\bm{c}).
\end{equation}
Therefore the prime factorisation of $r$ will play an important role in our analysis.
The letter $\varpi$  will generally denote a prime in $\Omon$.

\subsection*{Notation}
Whenever the letter $F$ appears in some context from now on, we take
$$
F = x_1^3+\dots+x_n^3,
$$
typically with $n=6$. However, we will often work more generally to clarify the nature of arguments.
Let
\begin{equation}\label{EQN:6-variable-Fermat-dual-form}
F^\ast(\bm{c}) = \prod (c_1^{3/2} \pm c_2^{3/2} \pm \dots \pm c_n^{3/2})\in \OK[\bm{c}]
\end{equation}
be the (primitive) dual form associated to $F$; it has degree $3\cdot 2^{n-2}$.
Let $V\subset \PP^{n-1}$ be the hypersurface defined by $F=0$. The dual form $F^*$ defines the dual variety $V^*\subset \PP^{n-1}$ of $V$, which parameterises hyperplanes that intersect $V$ tangentially. Alternatively, $V^*$ may be described as the closure of the image under the Gauss map $V\to \PP^{n-1}$ given by $\bm{x}\mapsto \nabla F(\bm{x})$. From this description it is clear that 
\begin{equation}\label{Eq: FdividesDualFormGrad}
    F(\bm{x}) \mid F^*(\nabla F(\bm{x})), 
\end{equation}
as an identity in $K[x_1,\dots, x_n]$. 
Let
\begin{equation}
\label{EQN:define-S_0,S_1-for-Delta-vanishing-and-non-zero-loci}
    \mathcal{S}_0 = \{\bm{c}\in \OK^n: F^\ast(\bm{c})=0\},
    \quad
    \mathcal{S}_1 = \{\bm{c}\in \OK^n: F^\ast(\bm{c})\neq 0\}.
\end{equation}
For each $\bm{c}\in \mathcal{S}_1$, let
\begin{equation}
\label{EQN:define-moduli-sets-N^c,N_c}
\RcG = \{r\in \Omon: \gcd(r,F^\ast(\bm{c})) = 1\},
\quad
\RcB = \{r\in \Omon: \varpi\mid r\Rightarrow \varpi\mid F^\ast(\bm{c})\}.
\end{equation}
Given $r\in \Omon$ and non-zero $B\in \OK$, 
we will sometimes write $r\mid B^\infty$ to mean that 
 $\varpi\mid r\Rightarrow \varpi\mid B$.
 This is equivalent to $\rad(r)\mid B$, where $\rad(r)$ is the radical of $r$.

To connect $S_r(\bm{c})$ to Hasse--Weil $L$-functions, we define further quantities.
Given a finite field $k$,
let $\mcal{V}(k)$ be the set of $k$-points on the variety $F(\bm{x})=0$ in $\PP^{5}_{k}$,
let $\mcal{V}_{\bm{c}}(k)$ be the set of $k$-points on the variety $F(\bm{x})=\bm{c}\cdot\bm{x}=0$ in $\PP^{5}_{k}$,
and let
\begin{equation*}
    E_F(k) \defeq \card{\mcal{V}(k)} - \card{\PP^{4}(k)},
    \quad E_{\bm{c}}(k) \defeq \card{\mcal{V}_{\bm{c}}(k)} - \card{\PP^{3}(k)},
\end{equation*}
where $\card{\PP^d(k)} = ((\# k)^{d+1}-1)/((\# k)-1)$.
Next,  let
\begin{equation}
\label{EQN:define-normalized-point-count-errors-E_F,E_c}
    E^\natural_F(k)\defeq (\# k)^{-2} E_F(k),
    \quad E^\natural_{\bm{c}}(k)\defeq (\# k)^{-3/2} E_{\bm{c}}(k).
\end{equation}
For primes $\varpi\nmid\bm{c}$, it is known that
\begin{equation}
\label{EQN:rewrite-S_c(p)-via-E_c}
    S^\natural_\varpi(\bm{c})
    = E^\natural_{\bm{c}}(\OK/\varpi\OK)
    - \abs{\varpi}^{-1/2}E^\natural_F(\OK/\varpi\OK);
\end{equation}
see \cite{hooley2014octonary}*{Lemma~7} for a general reference over prime fields, which carries over directly to prime-power fields.
Moreover, if $\varpi\nmid F^\ast(\bm{c})$ then \cite{browning2015rational}*{Lemma 5.2} gives
\begin{equation}
\label{EQN:smooth-vanishing-S_c}
    S_{\varpi^l}(\bm{c})=0
    \quad (l\ge 2).
\end{equation}

\subsection*{Technical results}

We proceed to record some useful results that will often be appealed to during the course of our argument. 

\begin{lemma}
\label{LEM:count-B-R_c-infty-divisors}
Let $R\ge 0$ and $B\in \OK \setminus \set{0}$.
Then $$\card{\{r\in \Omon: r\mid B^\infty,\; \abs{r}=\hat R\}} \ll_\eps (\hat R \abs{B})^\eps.$$
In particular, if $\bm{c}\in \mathcal{S}_1$, then
$\card{\{r\in \RcB: \abs{r}=\hat R\}}\ll_\eps \norm{\bm{c}}^\eps \hat R^\eps$.
\end{lemma}

\begin{proof}
$\sum_{r\mid B^\infty} \bm{1}_{\abs{r}=\hat R}
\le \sum_{r\mid B^\infty} (\hat R/\abs{r})^\eps
= \hat R^\eps \prod_{\varpi\mid B} (1 - \abs{\varpi}^{-\eps})^{-1}
\ll_\eps \hat R^\eps \abs{B}^\eps$.
\end{proof}

\begin{lemma}
\label{LEM:N_c-small-divisor-moment-bound}
If $Z,R\in \RR$ and $A,\eps>0$, then we have
\begin{equation}
\label{INEQ:N_c-small-divisor-moment-bound}
\sum_{\bm{c}\in \mathcal{S}_1:\, \norm{\bm{c}}\le \hat Z}
\card{\set{r\in \RcB: \abs{r}\le \hat R}}^A
\ll_{A, \eps} \hat Z^n \hat R^\eps.
\end{equation}
\end{lemma}

\begin{proof}
If $\hat Z<\hat R^A$, use Lemma~\ref{LEM:count-B-R_c-infty-divisors}.
Now suppose $\hat Z\ge \hat R^A$.
By H\"{o}lder's inequality, we may assume $A\in \NN$.
Then $\Sigma$, the left-hand side of \eqref{INEQ:N_c-small-divisor-moment-bound}, equals
\begin{equation*}
    \sum_{\bm{c}\in \mathcal{S}_1:\, \norm{\bm{c}}\le \hat Z}
    \sum_{\substack{u_1,\dots,u_A\in \RcB\\ \abs{u_i}\le \hat R}} 1
    = \sum_{\substack{u_1,\dots,u_A\in \Omon\\ \abs{u_i}\le \hat R}}\sum_{\bm{c}\in \mathcal{S}_1:\, \norm{\bm{c}}\le \hat Z}
    \bm{1}_{\rad(u_1\cdots u_A)\mid F^\ast(\bm{c})}.
\end{equation*}
Here $\abs{u_1\cdots u_A}\le \hat R^A\le \hat Z$, so by Lang--Weil and the Chinese remainder theorem,
\begin{equation*}
\begin{split}
\Sigma &\ll_\eps \sum_{u_1,\dots,u_A\in \Omon:\, \abs{u_i}\le \hat R} \frac{\hat Z^n}{\abs{\rad(u_1\cdots u_A)}^{1-\eps}} \\
&\le \sum_{r\in \Omon:\, \abs{r}\le \hat R^A} \frac{\hat Z^n}{\abs{r}^{1-\eps}} \sum_{u_1,\dots,u_A\in \Omon:\, \abs{u_i}\le \hat R,\; u_i\mid r^\infty} 1.
\end{split}
\end{equation*}
 Lemma~\ref{LEM:count-B-R_c-infty-divisors} now yields $\Sigma \ll_{A,\eps} \sum_{r\in \Omon:\, \abs{r}\le \hat R^A} \hat Z^n \abs{r}^{\eps-1} (\hat R \abs{r})^\eps \ll_\eps \hat Z^n \hat R^{(2A+1)\eps}$.
\end{proof}

At several points later, we shall require nontrivial bounds on certain level sets.

\begin{lemma}
Let $B,\lambda\in \RR$ and $z\in \overline{K}$.
Uniformly over places $v$ of $K$, we have
\begin{equation}
\label{INEQ:interval-estimate}
    \#\set{x\in \OK: \abs{x} \le \hat B,\; \abs{x-z}_v \le \hat \lambda}
    \ll_K 1 + \hat B^{\bm{1}_{v\ne \infty}} \hat \lambda.
\end{equation}
\end{lemma}

\begin{proof}
This is trivial if $v=\infty$.
If $v=\varpi\ne \infty$, we write $\Omega_z$ for the set appearing on the left hand side, which we may assume is non-empty. Fix an element $x_0\in \Omega_z$ and 
note that $\abs{x-z}_v \le \hat \lambda
\Leftrightarrow \abs{x-x_0}_v \le \hat \lambda
\Leftrightarrow \varpi^l \mid x-x_0$.
Hence we can 
write $x = x_0+y\varpi^l$ with $l \defeq -\floor{\lambda/\deg{\varpi}} \in \ZZ$ and $y\in \OK$,
to get
\begin{equation*}
\#\Omega_z
\le \#\set{y\in \OK: \abs{\varpi}^l\, \abs{y-z_0} \le \hat B},
\end{equation*}
where $z_0 \defeq -x_0/\varpi^l\in K$.
By the ``$v=\infty$ case'' of \eqref{INEQ:interval-estimate},
the desired result follows, since $\hat B/\abs{\varpi}^l = \hat B/q^{l\deg{\varpi}} = \hat B q^{(\deg{\varpi}) \floor{\lambda/\deg{\varpi}}} \le \hat B \hat \lambda$.
\end{proof}

\begin{lemma}
    Suppose $f\in K[x]$ has leading term $ax^d$ with $a\neq 0$ and $d\geq 1$.
    Let $B,\lambda\in \RR$.
    Then, uniformly over places $v$ of $K$,  we have 
    \begin{equation}
    \label{INEQ:integral-points-level-estimate}
        \#\set{x\in \OK: \abs{x} \le \hat B,\; \abs{f(x)}_v \le \hat \lambda}
        \ll_K d\, (1 + \hat B^{\bm{1}_{v\ne \infty}}\, (\hat \lambda/\abs{a}_v)^{1/d}).
    \end{equation}
\end{lemma}

\begin{proof}
    We use the idea behind Lemma 1 of
    \cite{browning2006counting}.
    By replacing $f$ with $f/a$, we may assume $f$ is monic.
    Let $z_1,\dots,z_d\in \overline{K}$ be the roots of $f$.
    Then $$
    \set{x\in \overline{K}: \abs{f(x)}_v \le \hat \lambda} \belongs
    \bigcup_{1\le i\le d} \set{x\in \overline{K}: \abs{x-z_i}_v \le \hat \lambda^{1/d}}.
    $$
    Summing \eqref{INEQ:interval-estimate} over $z=z_i$ gives the desired result.
\end{proof}

Given $f\in \OK[x]$ and $r\in \Omon$, let
$N(f;r)$ be the number of solutions $x\in \OK/r\OK$ to $f(x)\equiv 0\bmod{r}$.
Similarly, for $g\in \OK[y_1,\dots,y_n]$, define
\begin{equation}\label{eq:Ngr}
    N(g;r) \defeq \#{\set{(y_1,\dots,y_n)\in (\OK/r\OK)^n: g(y_1,\dots,y_n)\equiv 0\bmod{r}}}.
\end{equation}

\begin{lemma}
Suppose $f\in \OK[x]$ has leading term $ax^d$ with $a\neq 0$ and $d\geq 1$. Then there exists $A_d>0$ such that for any $r\in \Omon$ we have
\begin{equation}
\label{INEQ:univariate-zero-density-mod-r}
    N(f;r) \le A_d^{\omega(r)} \abs{a}^{1/d} \abs{r}^{1 - d^{-1}}.
\end{equation}
Moreover, if $r$ is square-free then
\begin{equation}
\label{INEQ:univariate-zero-density-mod-square-free}
    N(f;r) \le d^{\omega(r)} \abs{\gcd(r,a)}
    \ll_{a,d,\eps} \abs{r}^\eps.
\end{equation}
\end{lemma}

\begin{proof}
By the Chinese remainder theorem it suffices to assume $r=\varpi^l$ is a prime power.
Then to get \eqref{INEQ:univariate-zero-density-mod-r}, we apply \eqref{INEQ:integral-points-level-estimate} with
$    (\hat B,\hat \lambda,v) = (\abs{r},\abs{r}^{-1},\varpi)$,
noting that 
$|a|_\varpi^{-1}=q^{\ord_\varpi(a)}\leq |a|$.
Moreover, if $l=1$, then \eqref{INEQ:univariate-zero-density-mod-square-free} holds because $N(f;\varpi) \le d\, \abs{\gcd(\varpi,a)}$.
\end{proof}

\begin{corollary}\label{cor:HUA}
Fix a non-constant polynomial $g\in \OK[y_1,\dots,y_n]$, where $n\ge 1$.
Then there exists $A_g>0$ such that for any $r\in \Omon$ we have 
$$
    N(g;r)\le A_g^{\omega(r)} \abs{r}^{n - (\deg{g})^{-1}}.
$$
\end{corollary}

\begin{proof}
This is analogous to \cite{pierce2016representations}*{Lemma 4.10} (in which the homogeneity assumed  is unimportant).
By Lemma \ref{LEM:OK-linear-change-to-near-monic}, we can make an $\OK$-linear change of variables in order to assume that $\deg_{y_1}(g) = \deg(g)$.
Then \eqref{INEQ:univariate-zero-density-mod-r} immediately implies
the desired bound. 
\end{proof}

\begin{lemma}
\label{LEM:OK-linear-change-to-near-monic}
Let $n\ge 1$.
Let $S$ be a finite subset of $\OK[y_1,\dots,y_n] \setminus \set{0}$.
Then there exists $\bm{a}=(a_2,\dots,a_n)\in \OK^{n-1}$ such that for all $g\in S$, we have
$$
    \deg_{y_1}(\bm{a}^\ast{g})
    = \deg(\bm{a}^\ast{g})
    = \deg(g),
$$
where $\bm{a}^\ast{g}\defeq g(y_1,y_2+a_2y_1,\dots,y_n+a_ny_1)\in \OK[y_1,\dots,y_n]$.
\end{lemma}

\begin{proof}
For each $g\in S$, let $h_g\in \OK[y_1,\dots,y_n]\setminus \set{0}$ be the leading homogeneous part of $g$,
so that $\deg{h_g} = \deg{g}$ and $\deg(g-h_g) < \deg{g}$.
Then for any $g\in S$ and $\bm{a}\in R^{n-1}$ (for any $K$-algebra $R$), the desired property  is equivalent to
$$
h_g(1,a_2,\dots,a_n) \ne 0.
$$
But $h_g$ is homogeneous, so this defines a \emph{non-empty} open
subscheme of $\mathbb{A}^{n-1}_K$.
Since $S$ is finite and $\OK^{n-1}$ is Zariski dense in $\mathbb{A}^{n-1}_K$, the lemma follows.
\end{proof}

Beyond Corollary \ref{cor:HUA}, we need a strong estimate for square-free polynomials.

\begin{lemma}
\label{LEM:fixed-poly-cube-free-point-count-bound}
Fix a non-zero polynomial $g\in \OK[y_1,\dots,y_n]$, where $n\ge 1$.
Assume $g$ is square-free as an element of $K[y_1,\dots,y_n]$.
Then there exists $A_g>0$ such that for all cube-free $r\in \Omon$ we have 
\begin{equation*}
    N(g;r)\le A_g^{\omega(r)} \abs{r}^{n-1}.
\end{equation*}
\end{lemma}

\begin{proof}
Assume $r=\varpi^l$.
When $l=1$, the Lang--Weil estimate for the quasi-projective variety $g=0$ implies $N(g;\varpi) \ll_g \abs{\varpi}^{n-1}=\abs{r}^{n-1}$.
When $l=2$, care is needed because $K$ is not perfect.
Nonetheless, in \cite{poonen2003squarefree}*{final paragraph of \S~7 (in ``Proof of Theorem 3.4'')},
it is shown that $N(g;\varpi^2) \ll_g \abs{\varpi}^{2n-2} = \abs{r}^{n-1}$.
This completes the proof.
\end{proof}

Finally, we need an affine dimension growth bound available from \cite{browning2015rational}.

\begin{lemma}
[\cite{browning2015rational}*{Lemma 2.8}]
\label{LEM:affine-dimension-growth}
Let $B\geq 0$.
Any affine $K$-variety $W$ (equipped with an integral model over $\OK$)
satisfies $\#\set{\bm{y}\in W(\OK): \norm{\bm{y}}\le \hat B} \ll_W {\hat B}^{\dim{W}}$.
\end{lemma}

\section{Ratios Conjecture and applications}
\label{SEC:ratios-plus}

We start generally and then specialise.
Let $K^{\textnormal{sep}}$ be the separable closure of $K$.
Fix a prime $\ell\ne \cha(\FF_q)$.
For a smooth proper $K$-variety $X$, define the $\ell$-adic cohomology
$$H^i_\ell(X) \defeq H^i(X \times_K K^{\textnormal{sep}}, \QQ_\ell).$$
This is an $\ell$-adic representation of $\Gal(K^{\textnormal{sep}}/K)$,
pure of weight $i$ by Deligne's resolution of the Weil Conjectures.
The exterior square representation $\bigwedge^2{H^i_\ell(X)}$, which will soon play an important role, is pure of weight $2i$.

In general,
for an $\ell$-adic representation $\rho\maps \Gal(K^{\textnormal{sep}}/K) \to M$, pure of weight $w\in \ZZ$, we make the following definitions.
\begin{enumerate}
\item Let $M^{I_v}$ be the inertia invariants of $M$ at $v$.
Let $\alpha^0_{M,j}(v)$ (for $1\le j\le \dim{M^{I_v}}$) be the eigenvalues of geometric Frobenius on $M^{I_v}$.
\item Let $\alpha_{M,j}(v) \defeq \alpha^0_{M,j}(v) / (\# k_v)^{w/2}$, where $k_v$ is the residue field of $\mcal{O}_v$.
\item Let $$L_v(s,M) \defeq \prod_j (1-\alpha_{M,j}(v)\abs{k_v}^{-s})^{-1}$$ be the \emph{analytically normalized} local factor at $v$.
Let $$L(s,M) \defeq \prod_{\varpi} L_\varpi(s,M)$$ (with $\varpi\ne \infty$)
and $$\Lambda(s,M) \defeq \prod_v L_v(s,M)$$ (including $v=\infty$),
so that $\Lambda(\ast) = L(\ast) L_\infty(\ast)$.
\end{enumerate}

Let $V$ and $V_{\bm{c}}$ be the $K$-varieties in $\PP^{5}_K$ defined by $F(\bm{x})=0$ and $F(\bm{x})=\bm{c}\cdot\bm{x}=0$, respectively.
Let $L(s,V) = L(s,H^4_\ell(V)/H^4_\ell(\PP^5))$,
and for $\bm{c}\in \mcal{S}_1$ let
\begin{equation*}
L(s,\bm{c}) = L(s,H^3_\ell(V_{\bm{c}})),
\quad {\textstyle L(s,\bm{c},\bigwedge^2) = L(s,\bigwedge^2{H^3_\ell(V_{\bm{c}})})}.
\end{equation*}
Let $\alpha_{V,j}(v)$, $\alpha_{\bm{c},j}(v)$, $\alpha_{\bm{c},\bigwedge^2,j}(v)$ be the corresponding normalized eigenvalues.
The local factors are independent of $\ell$, by \cite{kahn2020zeta}*{Theorem 5.46}, and the number of normalized eigenvalues in each case is $\le \binom{10}{2}=45$ by classical Betti number calculations (for smooth projective hypersurfaces over $\CC$).

For $r\in \Omon$,
define $\lambda_V(r)$, $\lambda_{\bm{c}}(r)$,
$\lambda_{\bm{c},\bigwedge^2}(r)$ to be the \emph{$r$th coefficients}
of the Euler products $L(s,V)$, $L(s,\bm{c})$, $L(s,\bm{c},\bigwedge^2)$, respectively.
More precisely, for prime $\varpi\in \Omon$ let $\lambda_\ast(\varpi^k)$ be the coefficient of $\abs{\varpi}^{-ks}$ in $L_\varpi(s,\ast)$, and extend multiplicatively to define $\lambda_\ast(r)$ for $r\in \Omon$.
(We have to carefully define $\lambda_\ast(r)$, because for any $\varpi$, $\varpi'$ the sizes $\abs{\varpi}$, $\abs{\varpi'}$ are multiplicatively dependent.)

Let $\mathsf{T} = \frac{2\pi}{\log{q}}$.
Note that $q^s$ is invariant under translation by $i\mathsf{T}$, so
\begin{equation}
\label{EQN:vertical-periodicity-of-L-functions}
L(s,\ast) = L(s+i\mathsf{T}\ZZ,\ast).
\end{equation}

Let $\zeta_K(s)
\defeq \prod_{\varpi} (1 - \abs{\varpi}^{-s})^{-1}
= \sum_{r\in \Omon} \abs{r}^{-s} = (1-q^{1-s})^{-1}$.
For later convenience we define 
\begin{equation}
\label{EQN:factor-exterior-square-L-function}
    L(s,\bm{c},2)\defeq L(s,\bm{c},{\textstyle\bigwedge^2})/\zeta_K(s),
\end{equation}
and  $L_v(s,\bm{c},2)\defeq L_v(s,\bm{c},{\textstyle\bigwedge^2})/\zeta_{K,v}(s)$.

\begin{proposition}
\label{PROP:HW2-consequences}
Let $\bm{c}\in \mcal{S}_1$.
Let $L(s)$ be one of $\zeta_K(s)$, $L(s,V)$, $L(s,\bm{c})$,
$L(s,\bm{c},\bigwedge^2)$, $L(s,\bm{c},2)$.
Let $\Lambda(s) = L(s) L_\infty(s)$ be the corresponding completed $L$-function.
\begin{enumerate}
    \item There exist unique polynomials $P_0,P_1,P_2\in 1+z\,\RR[z]$,
    where the roots of $P_i$ are complex numbers of size $q^{-i/2}$, such that
    \begin{equation*}
        \Lambda(s) = \frac{P_1(q^{-s})}{P_0(q^{-s}) P_2(q^{-s})}.
    \end{equation*}
    
    \item $\deg{P_0},\deg{P_2} \ll 1$ and $\deg{P_1} \ll 1+\log\abs{F^\ast(\bm{c})}$.
    
    \item At each place $v$, the local factor $L_v(s) = \prod_\alpha (1-\alpha q^{-s})^{-1}$ has real coefficients, and the inverse roots $\alpha$ satisfy the \emph{Ramanujan bound} $\abs{\alpha}\le 1$.
    
    \item We have $1/L(s) \ll_{\eps} \norm{\bm{c}}^\eps$ for $\Re(s)\ge \frac12+\eps$.

    \item The $\hat R^{-s}$ coefficient of $1/L(s)\in \RR[[q^{-s}]]$ is $\ll_{\eps} \norm{\bm{c}}^\eps \hat R^{1/2+\eps}$ for all $R\in \ZZ_{\ge 0}$.
\end{enumerate}
\end{proposition}

\begin{proof}
(1):
For $\zeta_K(s)$, $L(s,V)$, $L(s,\bm{c})$, $L(s,\bm{c},\bigwedge^2)$,
this follows directly from \cite{kahn2020zeta}*{proof of Theorem 5.58},
since the corresponding $\ell$-adic sheaves $\QQ_\ell$, $H^4_\ell(V)/H^4_\ell(\PP^5)$, $H^3_\ell(V_{\bm{c}})$, $\bigwedge^2{H^3_\ell(V_{\bm{c}})}$ are ``weakly polarisable'' (by Poincar\'{e} duality on $V$ and $V_{\bm{c}}$, which gives a symmetric pairing on $H^4_\ell(V)$ and a skew-symmetric pairing on $H^3_\ell(V_{\bm{c}})$, and thus a symmetric pairing on the tensor squares thereof).
By Proposition~\ref{PROP:symplectic-pairing-implies-analytic-pole} and the formula $\zeta_K(s) = 1/(1-q^{1-s})$, the result then follows for $L(s,\bm{c},2)$ by \eqref{EQN:factor-exterior-square-L-function}.

(2):
This is explained in \cite{browning2015rational}*{\S~3.4} using Swan conductors.

(3):
This follows from \cite{kahn2020zeta}*{Theorem 5.46}.

(4):
This follows from (1)--(3) and the Hadamard three circle theorem as in \cite{browning2015rational}*{proof of Lemma 8.4}.

(5):
Let $s=\frac12+\eps+i\ts$ and integrate $\hat R^s/L(s) \ll_{\eps} \norm{\bm{c}}^\eps \hat R^{1/2+\eps}$ over $\ts\in \RR/\mathsf{T}\ZZ$.
\end{proof}

\begin{proposition}
\label{PROP:symplectic-pairing-implies-analytic-pole}
Let $\bm{c}\in \mcal{S}_1$.
Then $L(s,\bm{c},\bigwedge^2)$ has a pole at $s=1$.
\end{proposition}

\begin{proof}
Let $M = H^3_\ell(V_{\bm{c}})$.
Poincar\'{e} duality gives a perfect \emph{skew-symmetric} pairing $$\psi\maps M \wedge M \to \QQ_\ell(-3),$$ where $\QQ_\ell(-3)$ denotes the Tate motive of weight $6$.
Since $\psi$ is surjective, its dual $$\psi^\vee\maps \QQ_\ell(3) \to M^\vee \wedge M^\vee$$ is injective.
But $\psi$ is perfect, so that it induces an isomorphism $$\psi'\maps M \to \Hom(M, \QQ_\ell(-3)) = M^\vee(-3).$$
After twisting $\psi^\vee$ by $\QQ_\ell(-6)$, we thus obtain an injection $\QQ_\ell(-3) \to M \wedge M$.
Thus $(M\wedge M)(3)$ (which is pure of weight $0$) has a non-zero space of $\Gal(K^{\textnormal{sep}}/K)$-invariants.
So by \cite{lyons2012erratum}*{correction to Theorem~2.1} (which builds on \cites{lafforgue2002chtoucas,lyons2009rank}), it follows that 
the $L$-function $L(s,(M\wedge M)(3)) = L(s,\bm{c},\textstyle{\bigwedge^2})$ has a pole at $s=1$.
\end{proof}

\begin{proposition}
[Kisin]
\label{PROP:ineffective-Krasner-lemma-of-Kisin}
Fix a prime $\varpi\in \Omon$ and a tuple $\bm{b}\in \OK_\varpi^6$ with $F^\ast(\bm{b})\ne 0$.
Then there exists an integer $l\ge 0$, depending only on $\varpi$ and $\bm{b}$, such that for all tuples $\bm{a}\in \OK_\varpi^6$ with $\bm{a}\equiv \bm{b}\bmod{\varpi^{1+l}}$, we have $F^\ast(\bm{a})\ne 0$ and $L_\varpi(s,\bm{a}) = L_\varpi(s,\bm{b})$.
\end{proposition}

\begin{proof}
This follows directly from \cite{kisin1999local}*{case (2) of Theorem 5.1}.
\end{proof}

\begin{lemma}
\label{LEM:density-form-of-ineffective-Krasner-lemma}
For each $\bm{b}\in \OK_\varpi^6$ with $F^\ast(\bm{b})\ne 0$, let $l(\varpi,\bm{b})$ be the smallest integer $l\ge 0$ verifying Proposition~\ref{PROP:ineffective-Krasner-lemma-of-Kisin}.
Then for any integer $l\ge 0$, the set
\begin{equation}
\label{EXPR:p-adic-set-of-b-with-l(p,b)-at-most-l}
    \set{\bm{b}\in \OK_\varpi^6: F^\ast(\bm{b})\ne 0,\; l(\varpi,\bm{b})\le l}
\end{equation}
is closed under translation by $\varpi^{l+1}\OK_\varpi^6$,
and its measure tends to $1$ as $l\to \infty$.
\end{lemma}

\begin{proof}
We first show that $l(\varpi,\bm{b})$
is locally constant.
Let $\bm{b}, \bm{c}\in \OK_\varpi^6$ with $F^\ast(\bm{b})\ne 0$ and $\bm{c}\equiv \bm{b}\bmod{\varpi^{1+l(\varpi,\bm{b})}}$.
Then $F^\ast(\bm{c})\ne 0$ and $L_\varpi(s, \bm{c}) = L_\varpi(s, \bm{b})$, and thus $l(\varpi,\bm{c})\le l(\varpi,\bm{b})$ (since $\bm{a}\equiv \bm{c}\bmod{\varpi^{1+l(\varpi,\bm{b})}} \Rightarrow \bm{a}\equiv \bm{b}\bmod{\varpi^{1+l(\varpi,\bm{b})}}$).
But then $\bm{b}\equiv \bm{c}\bmod{\varpi^{1+l(\varpi,\bm{c})}}$, so $l(\varpi,\bm{b})\le l(\varpi,\bm{c})$, whence $l(\varpi,\bm{b}) = l(\varpi,\bm{c})$.
So if $\bm{b}$ lies in \eqref{EXPR:p-adic-set-of-b-with-l(p,b)-at-most-l} for some $l\ge 0$, then \eqref{EXPR:p-adic-set-of-b-with-l(p,b)-at-most-l} is indeed closed under translation by $\varpi^{l+1}\OK_\varpi^6$.

We now turn to measures.
Let $A\in \ZZ_{\ge 0}$.
The set $S_A = \set{\bm{c}\in \OK_\varpi^6: v_\varpi(F^\ast(\bm{c}))\le A}$ is compact.
Yet the function $\bm{b}\mapsto l(\varpi,\bm{b})$ on $S_A$ is continuous.
So \eqref{EXPR:p-adic-set-of-b-with-l(p,b)-at-most-l} contains $S_A$ for all $l\gg_A 1$, say.
But by 
Lemma \ref{cor:HUA},  the measure of $S_A$ tends to $1$ as $A\to \infty$.
\end{proof}

Let $\mu_{\bm{c}}(r)$ be the $r$th coefficient of the Euler product $L(s,\bm{c})^{-1} = \prod_\varpi L_\varpi(s,\bm{c})^{-1}$. The following result collects together some facts about  averages of 
$\mu_{\bm{c}}(r)$ over vectors $\bm{c}$.

\begin{proposition}
\label{PROP:(LocAvSp)}
Let $r_1,r_2\in \Omon$.
Let $\EE_{\bm{c}\in S}[f]$ be the average of $f$ over $S$.
The limit
$$
\bar{\mu}_{F,2}(r_1,r_2)
\defeq \lim_{Z\to \infty}
\EE_{\bm{c}\in \mathcal{S}_1:\, \norm{\bm{c}}\le \hat Z}
[\mu_{\bm{c}}(r_1)\mu_{\bm{c}}(r_2)]
$$
 exists. Moreover, 
 $\bar{\mu}_{F,2} (r_1,r_2)\bar{\mu}_{F,2}(r'_1,r'_2)
= \bar{\mu}_{F,2}(r_1r'_1,r_2r'_2)
$ 
 if $\gcd(r_1r_2,r'_1r'_2)=1$.

Now let $\varpi\in \Omon$ be a prime, and let $l,l_1,l_2\ge 0$ be integers.
Then
\begin{equation}
\label{INEQ:GRC-bound-for-averages-mu-bar}
 \bar{\mu}_{F,2}(\varpi^{l_1},\varpi^{l_2}) \ll_\eps \abs{\varpi}^{(l_1+l_2)\eps}.
\end{equation}
Furthermore, 
\begin{gather}
\bar{\mu}_{F,2}(\varpi,1) = \lambda_V(\varpi)\abs{\varpi}^{-1/2}+O(\abs{\varpi}^{-1}),
\quad
\bar{\mu}_{F,2}(\varpi^2,1) = 1+O(\abs{\varpi}^{-1}), \label{EQN:LocAv1} \\
\bar{\mu}_{F,2}(\varpi^l,1) = \bar{\mu}_{F,2}(1,\varpi^l),
\quad
\bar{\mu}_{F,2}(\varpi,\varpi) = 1+O(\abs{\varpi}^{-1}). \label{EQN:LocAv2}
\end{gather}
\end{proposition}

\begin{proof}
This result is directly analogous to \cite{wang2023ratios}*{Proposition 6.1}, and  is purely local.
For the reader's convenience, we highlight the main points.
Existence and multiplicativity of the limit follows from Lemma \ref{LEM:density-form-of-ineffective-Krasner-lemma}, 
the bound $\abs{\mu_\cc(r)} \ll_\eps \abs{r}^\eps$ (which follows from 
Proposition \ref{PROP:HW2-consequences}(3)), 
and the Chinese remainder theorem.
The bound $\abs{\mu_\cc(r)} \ll_\eps \abs{r}^\eps$ implies  \eqref{INEQ:GRC-bound-for-averages-mu-bar}.
The estimates in \eqref{EQN:LocAv1} and \eqref{EQN:LocAv2} are subtler; they relate to the \emph{rank} and \emph{homogeneity type} \cite{sarnak2016families}*{\S~1} of our geometric family of $L$-functions $L(s,\bm{c})$.

Recall the definition of $E^\natural_{\bm{c}}(k)$ from \eqref{EQN:define-normalized-point-count-errors-E_F,E_c}.
The following hold uniformly over primes $\varpi\in \Omon$ (where $k_\varpi = \OK/\varpi\OK$ and $k_{\varpi,2}$ is the unique quadratic field extension of $k_\varpi$):
\begin{align}
    \EE_{\bm{c}\in k_\varpi^6}[E^\natural_{\bm{c}}(k_\varpi)
    \bm{1}_{\varpi\nmid F^\ast(\bm{c})}]
    &= \lambda_V(\varpi) \abs{\varpi}^{-1/2}
    + O(\abs{\varpi}^{-1}), \label{EQN:E_c(p)-average} \\
    \EE_{\bm{c}\in k_\varpi^6}[E^\natural_{\bm{c}}(k_{\varpi,2})
    \bm{1}_{\varpi\nmid F^\ast(\bm{c})}]
    &= 1 + O(\abs{\varpi}^{-1}), \label{EQN:E_c(p^2)-average} \\
    \EE_{\bm{c}\in k_\varpi^6}[E^\natural_{\bm{c}}(k_\varpi)^2
    \bm{1}_{\varpi\nmid F^\ast(\bm{c})}]
    &= 1 + O(\abs{\varpi}^{-1}). \label{EQN:E_c(p)^2-average}
\end{align}
This is proven for prime fields in \cite{wang2023dichotomous}*{Corollary~1.7};
the proof there carries
over directly to arbitrary finite fields of characteristic $>3$.

Finally, observe that if $\varpi\nmid F^\ast(\bm{c})$,
then $L_\varpi(s,\bm{c})^{-1} = \prod_{1\le j\le 10} (1 - \alpha_{\bm{c},j}(\varpi) \abs{\varpi}^{-s})$, so
\begin{align*}
\mu_{\bm{c}}(\varpi)
&= -\sum_{1\le j\le 10} \alpha_{\bm{c},j}(\varpi)
= E^\natural_{\bm{c}}(k_\varpi), \\
\mu_{\bm{c}}(\varpi^2)
&= \sum_{1\le i<j\le 10} \alpha_{\bm{c},i}(\varpi) \alpha_{\bm{c},j}(\varpi)
= \tfrac12(E^\natural_{\bm{c}}(k_\varpi)^2+E^\natural_{\bm{c}}(k_{\varpi,2})),
\end{align*}
because $E^\natural_{\bm{c}}(k_\varpi) = -\sum_{1\le j\le 10} \alpha_{\bm{c},j}(\varpi)$
and $E^\natural_{\bm{c}}(k_{\varpi,2}) = -\sum_{1\le j\le 10} \alpha_{\bm{c},j}(\varpi)^2$
by the Grothendieck--Lefschetz trace formula.
Thus \eqref{EQN:E_c(p)-average}--\eqref{EQN:E_c(p)^2-average} imply \eqref{EQN:LocAv1} and \eqref{EQN:LocAv2}.
\end{proof}

Informally, Proposition~\ref{PROP:(LocAvSp)} tells us
\begin{equation*}
\begin{split}
\sum_{r_1,r_2\in \Omon} \frac{\bar{\mu}_{F,2}(r_1,r_2)}{\abs{r_1}^{s_1}\abs{r_2}^{s_2}}
&\approx \prod_\varpi \Bigl[1 + \abs{\varpi}^{-s_1-s_2} + \sum_{1\le j\le 2}
(\lambda_V(\varpi)\abs{\varpi}^{-s_j-1/2} + \abs{\varpi}^{-2s_j})\Bigr] \\
&\approx \zeta_K(s_1+s_2) \prod_{1\le j\le 2}(L(s_j+\tfrac12,V) \zeta_K(2s_j)).
\end{split}
\end{equation*}
Motivated by this,
let 
\begin{equation*}
A_{F,2}(s_1,s_2)
= \frac{\zeta_K(s_1+s_2)^{-1}}{\prod_{1\le j\le 2}(\zeta_K(2s_j)L(s_j+\tfrac12,V))}
\sum_{r_1,r_2\in \Omon} \frac{\bar{\mu}_{F,2}(r_1,r_2)}{\abs{r_1}^{s_1}\abs{r_2}^{s_2}},
\end{equation*}
which by Proposition~\ref{PROP:(LocAvSp)} converges absolutely for $\Re(s_1), \Re(s_2) > \frac13$.
The expression $A_{F,2}$ appears as the ``leading constant'' in the Ratios Conjecture for
\begin{equation*}
\frac{1}{L(s_1,\bm{c})L(s_2,\bm{c})}.
\end{equation*}
Let $\sigma(Z) = \frac12+\frac1Z$.
For $\bm{c}\in \mathcal{S}_1$, let
\begin{equation*}
\Phi^{\bm{c},1}(s) = \zeta_K(2s)^{-1} L(s+\tfrac12, V)^{-1} L(s,\bm{c})^{-1}.
\end{equation*}
For convenience,
let $a_{\bm{c},1}(r)$ be the $r$th coefficient of the Euler product $\Phi^{\bm{c},1}(s)$.

The Ratios Recipe \cite{conrey2008autocorrelation}*{\S~5.1}, directly adapted to function fields as in \cite{andrade2014conjectures}, produces the following  Ratios Conjecture~(R2), even with a power-saving error term $O(\hat Z^{-\delta})$ \emph{independent of $\beta$} for $\beta\le \delta$, say.
The derivation over $K$ is exactly the same as the derivation over $\QQ$ in \cite{wang2023ratios}*{\S~6.3}.
For additional context on random matrix predictions for $L$-functions, we note that the $L$-functions $L(s,\bm{c})$ over $\bm{c}\in \mcal{S}_1$ form a \emph{geometric family} in the sense of \cite{sarnak2016families}*{pp.~534--535}.

\begin{conjecture}
[R2]
\label{CNJ:(R2o)}
There exists a constant $\beta\in [0,1]$ such that
if $$s_j=\beta+\sigma(Z)+i\ts_j,$$
then uniformly over $Z\in \NN$ and $\ts_1,\ts_2\in \RR$, we have
\begin{equation}
\label{EQN:soft-R2-goal}
\sum_{\bm{c}\in \mathcal{S}_1:\, \norm{\bm{c}}\le \hat Z}
\Phi^{\bm{c},1}(s_1) \Phi^{\bm{c},1}(s_2)
= \sum_{\bm{c}\in \mathcal{S}_1:\, \norm{\bm{c}}\le \hat Z}
(\zeta_K(s_1+s_2) + O(\hat Z^{-6\beta})) \, A_{F,2}(s_1,s_2).
\end{equation}
\end{conjecture}

(In fact one would expect (R2) to hold with $\beta=0$,
but the approach of \cites{bergstrom2023hyperelliptic,MPPRW} suggests that a small positive value of $\beta$ might be more tractable, at least for large fixed $q$.
Also, the ``slope'' $6$ in the exponent $6\beta$ is what we need for our application to Conjecture~\ref{CNJ:(R2')} below, but arbitrarily large slope should be permissible.)

For a vertically $\mathsf{T}$-periodic function $f(s)=f(\sigma+i\ts)$, define the \emph{vertical average} $$\EE_{\Re(s)=\sigma}[f(s)]
\defeq \frac1{\mathsf{T}} \int_{\RR/\mathsf{T}\ZZ} f(\sigma+i\ts)\, \mathrm{d}\ts.$$
Note that if $f$ is holomorphic on a vertical strip $\Re(s)\in I$, then $\EE_{\Re(s)=\sigma}[f(s)]$ is independent of $\sigma\in I$, by Cauchy's integral theorem over rectangles of height $\mathsf{T}$.

\begin{conjecture}
\label{CNJ:(R2')}
Let $Z,R\in \ZZ$ with $R\le 3Z$.
If $\sigma_0 > 1/2$, then
\begin{equation}
\label{INEQ:R2'-goal}
\sum_{\bm{c}\in \mathcal{S}_1:\, \norm{\bm{c}}\le \hat Z}
\left\lvert\EE_{\Re(s)=\sigma_0}[\Phi^{\bm{c},1}(s) \cdot \hat R^s]\right\rvert^2
\ll \hat Z^6 \hat R.
\end{equation}
\end{conjecture}

\newcommand{\Sigzer}{\Sigma_0} 
\newcommand{\Sigone}{\Sigma_1} 
\newcommand{\Sigtwo}{\Sigma_2} 
\newcommand{\Sigthr}{\Sigma_3} 
\newcommand{\Sigfou}{\Sigma_4} 

\begin{proof}
[Proof assuming (R2)]

By \eqref{EQN:vertical-periodicity-of-L-functions} and GRH, the left-hand side of \eqref{INEQ:R2'-goal} is independent of $\sigma_0>\frac12$.
If $\hat Z\ll 1$ then \eqref{INEQ:R2'-goal} is trivial, so suppose $Z\ge 2$.
We find, after shifting contours to $\Re(s) = \beta+\sigma(Z)\in (\tfrac12, 2]$
and expanding squares using self-duality of $\Phi^{\bm{c},1}$,
the left-hand side of \eqref{INEQ:R2'-goal} equals
\begin{equation*}
\Sigzer \defeq
\sum_{\bm{c}\in \mathcal{S}_1:\, \norm{\bm{c}}\le \hat Z}
\EE_{\Re(s_1),\Re(s_2)=\beta+\sigma(Z)}[
\Phi^{\bm{c},1}(s_1) \Phi^{\bm{c},1}(s_2) \cdot \hat R^{s_1+s_2}].
\end{equation*}
After switching the order of $\bm{c}$ and $\bm{s}$ in $\Sigzer$,
and plugging in \eqref{EQN:soft-R2-goal} for each $\bm{s}$, we get
\begin{equation}
\label{EQN:decompose-Sigma_0-in-R2'-under-R2}
    \Sigzer = \Sigtwo + O(\hat Z^{6-6\beta} \Sigone^2),
\end{equation}
where $\Sigone \defeq \EE_{\Re(s)=\beta+\sigma(Z)}[\abs{ \hat R^s}]=\hat R^{\beta+\sigma(Z)}$ and
\begin{equation*}
\Sigtwo \defeq \sum_{\bm{c}\in \mathcal{S}_1:\, \norm{\bm{c}}\le \hat Z}
\EE_{\Re(s_1),\Re(s_2)=\beta+\sigma(Z)}[
\zeta_K(s_1+s_2) A_{F,2}(s_1,s_2) \cdot \hat R^{s_1+s_2}].
\end{equation*}
Since $R\le 3Z$ and $\hat{Z}^{1/Z} = q \ll 1$, we certainly have
$\Sigone\ll \hat Z^{3\beta} \hat R^{1/2}$.

Let $\delta=\frac{1}{20}$; then $\frac12-\delta\ge \frac13+\delta$.
In $\Sigtwo$ we 
shift $\Re(s_1)$  from $\beta+\sigma(Z)$ to $\frac12$,
then shift $\Re(s_2)$  from $\beta+\sigma(Z)$ to $\frac12-\delta$. This 
yields $\Sigtwo = \Sigthr+O(\Sigfou),$ where
\begin{equation*}
\Sigthr \defeq \sum_{\bm{c}\in \mathcal{S}_1:\, \norm{\bm{c}}\le \hat Z}
\EE_{\Re(s_1)=\frac12}[
\tfrac{1}{\mathsf{T}} \tfrac{2\pi}{\log{q}} A_{F,2}(s_1, 1-s_1) \cdot \hat R]
\end{equation*}
comes from the residue $(\log{q})^{-1}$ of $\zeta_K(s_1+s_2) = (1-q^{1-s_1-s_2})^{-1}$ at $s_2=1-s_1$,
and where we bound the contribution from $\Re(s_1) = \frac12$ and $\Re(s_2) = \frac12-\delta$ by
\begin{equation*}
\Sigfou \defeq \sum_{\bm{c}\in \mathcal{S}_1:\, \norm{\bm{c}}\le \hat Z}
\hat R^{1-\delta} \ll \hat Z^6 \hat R^{1-\delta}, 
\end{equation*}
on using the bounds $A_{F,2}(s_1,s_2)\ll 1$ and $\zeta_K(s_1+s_2)\ll_\delta 1$.
Next we note that 
\begin{equation*}
\begin{split}
\Sigthr &\ll \hat Z^6 \hat R, 
\end{split}
\end{equation*}
since  $A_{F,2}(s_1, 1-s_1) \ll 1$ in $\Sigthr$.
But
$$
\Sigzer \le O(\hat Z^{6-6\beta} \Sigone^2) + \Sigthr + O(\Sigfou),
$$
by \eqref{EQN:decompose-Sigma_0-in-R2'-under-R2}.
Plugging in our estimates  for $\Sigone, \Sigthr$ and $\Sigfou$, we finally get \eqref{INEQ:R2'-goal}.
\end{proof}

\begin{conjecture}
\label{CNJ:(R2'E)}
Let $Z,R\in \ZZ$ with $R\le 3Z$.
Then
\begin{equation}
\label{INEQ:R2'E-goal}
\sum_{\bm{c}\in \mathcal{S}_1:\, \norm{\bm{c}}\le \hat Z}\,
\Bigl\lvert{
\sum_{r\in \Omon} \bm{1}_{\abs{r}=\hat R} \cdot a_{\bm{c},1}(r)
}\Bigr\rvert^2
\ll \hat Z^6 \hat R.
\end{equation}
\end{conjecture}

\begin{proof}
[Proof assuming Conjecture~\ref{CNJ:(R2')}]
Since $R\in \ZZ$, we may write
$$\bm{1}_{\abs{r}=\hat R} = \EE_{\Re(s)=\sigma_0}[(\hat R/\abs{r})^s]$$
for each $r\in \Omon$. Now apply  \eqref{INEQ:R2'-goal}.
\end{proof}

\section{Euler product factorisations}

Recall the definition 
\eqref{EQN:define-moduli-sets-N^c,N_c} of $\RcG$ and $\RcB$.
Given $\bm{c}\in \mathcal{S}_1$, consider the factorisation $\Phi = \Phi^{\map{G}} \Phi^{\map{B}}$, where
\begin{align}
    \Phi_\varpi(\bm{c},s)
    &\defeq \sum_{l\ge 0} \abs{\varpi}^{-ls} S^\natural_{\varpi^l}(\bm{c}),
\notag
 \\
    \Phi(\bm{c},s)
    &\defeq \prod_\varpi \Phi_\varpi(\bm{c},s)
    = \sum_{r\in \Omon} \abs{r}^{-s} S^\natural_r(\bm{c}),
    \label{EQN:define-Phi} \\
    \Phi^{\map{G}}(\bm{c},s)
    &\defeq \prod_{\varpi\nmid F^\ast(\bm{c})} \Phi_\varpi(\bm{c},s)
    = \sum_{r\in \RcG} S^\natural_r(\bm{c}) \abs{r}^{-s},
\notag \\
    \Phi^{\map{B}}(\bm{c},s)
    &\defeq \prod_{\varpi\mid F^\ast(\bm{c})} \Phi_\varpi(\bm{c},s)
    = \sum_{r\in \RcB} S^\natural_r(\bm{c}) \abs{r}^{-s}.
    \notag
\end{align}
For convenience, let $L^{\map{G}}(\ast) \defeq \prod_{\varpi\nmid F^\ast(\bm{c})} L_\varpi(\ast)$ and $L^{\map{B}}(\ast) \defeq \prod_{\varpi\mid F^\ast(\bm{c})} L_\varpi(\ast)$.

\begin{definition}
\label{DEFN:factor-Phi^G-into-Phi_1-Phi_2-Phi_3}
Let $\Phi^{\bm{c},1}(s)\defeq L(s,\bm{c})^{-1} L(\frac12+s,V)^{-1} \zeta_K(2s)^{-1}$ as before, and let
\begin{align*}
\Phi^{\bm{c},2}(s)
&\defeq L^{\map{G}}(2s,\bm{c},2)^{-1},
 \\
\Phi^{\bm{c},3}(s)
&\defeq \Phi^{\map{G}}(\bm{c},s)
L(s,\bm{c})
L(\tfrac12+s,V)
\zeta_K(2s) L^{\map{G}}(2s,\bm{c},2).
\end{align*}
For each $j\in \set{1,2,3}$,
let $a_{\bm{c},j}(r)$ be the $r$th coefficient of the Euler product $\Phi^{\bm{c},j}(s)$.
\end{definition}

Multiplying out the definitions of $\Phi^{\bm{c},j}(s)$ above, we find that
$$\prod_{1\le j\le 3} \Phi^{\bm{c},j}(s)
= \Phi^{\map{G}}(\bm{c},s).$$
Thus,  in order to use \eqref{INEQ:R2'E-goal} (which concerns $\Phi^{\bm{c},1}$) to bound $\Phi^{\map{G}}(\bm{c},s)$,
we will first need to control the factors $\Phi^{\bm{c},2}$ and $\Phi^{\bm{c},3}$ on average over $\bm{c}$.

\begin{proposition}
Let $Z,R\in \ZZ$ and $A, \eps \in \RR_{>0}$.
Then 
\begin{equation}
\label{INEQ:good-moduli-restricted-wedge2E-goal}
\sum_{\bm{c}\in \mathcal{S}_1:\, \norm{\bm{c}}\le \hat Z}\,
\Bigl\lvert{
\sum_{r\in \Omon:\, \abs{r}=\hat R} a_{\bm{c},2}(r)
}\Bigr\rvert^A
\ll_{A, \eps} \hat Z^6 \hat R^{A/4 + \eps}.
\end{equation}
\end{proposition}

\begin{proof}
If $\hat R\le 1$ then \eqref{INEQ:good-moduli-restricted-wedge2E-goal} is trivial, so suppose $\hat R>1$.
By H\"{o}lder's inequality, we may assume $A$ is an even integer.
Let $\Sigma(Z,R)$ denote the left-hand side of \eqref{INEQ:good-moduli-restricted-wedge2E-goal} and put $$\widetilde{\Sigma}(Z,R)\defeq \Sigma(Z,R)/\hat Z^6.$$

We claim that for $Z\ge AR$ we have
\begin{equation}
\label{EQN:large-Z-inflation-for-for-restricted-wedge2E}
    \widetilde{\Sigma}(Z,R) = \widetilde{\Sigma}(AR,R).
\end{equation}
To prove this, we 
write $\abs{\sum_{\abs{r}=\hat R} a_{\bm{c},2}(r)}^A = \sum_{\abs{r_1}=\dots=\abs{r_A}=\hat R} a_{\bm{c},2}(r_1)\cdots a_{\bm{c},2}(r_A)$, on noting that $a_{\bm{c},2}(r)\in \RR$.
If $r_1\cdots r_A = M$, then (by smooth proper base change) the function $\bm{c}\mapsto a_{\bm{c},2}(r_1)\cdots a_{\bm{c},2}(r_A)$ is constant on each fibre of the map
$$
\set{\bm{c}\in \mcal{S}_1: \norm{\bm{c}}\le \hat Z,\; (M,F^\ast(\bm{c})) = 1} \to \set{\bm{c}\in (\OK/M\OK)^6: \gcd(M,F^\ast(\bm{c})) = 1}.
$$
Each fibre has cardinality $(\hat Z/\abs{M})^6$, since $1<\abs{M}\le \hat R^A\le \hat Z$.
The claim follows.

By GRH (Proposition \ref{PROP:HW2-consequences}(5)), we have $\Sigma(Z,R) \ll_{A,\eps} \hat Z^{6+\eps} \hat R^{A/4+\eps}$ for all $\hat Z>0$.
(The exponent is $A/4$ rather than $A/2$, because the Euler product 
of $\Phi^{\bm{c},2}(s)$, as  given in 
Definition \ref{DEFN:factor-Phi^G-into-Phi_1-Phi_2-Phi_3}, 
is supported on \emph{square} moduli due to the appearance of $2s$.)
Thus 
$
\widetilde{\Sigma}(Z,R) \ll_{A,\eps} \hat R^{A/4+\eps}
$
for $Z\le AR$, and so  for all $Z$, by 
\eqref{EQN:large-Z-inflation-for-for-restricted-wedge2E}.
\end{proof}

\begin{proposition}
\label{PROP:approx-Phi-past-1/2}
For $\bm{c}\in \mcal{S}_1$, primes $\varpi\in \Omon$, and integers $l\ge 1$, we have
\begin{equation*}
a_{\bm{c},3}(\varpi)\cdot\bm{1}_{\varpi\nmid F^\ast(\bm{c})}=0,
\quad
a_{\bm{c},3}(\varpi^2)
\cdot\bm{1}_{\varpi\nmid F^\ast(\bm{c})}\ll \abs{\varpi}^{-1/2},
\quad
a_{\bm{c},3}(\varpi^l)\ll_\eps \abs{\varpi}^{l\eps}.
\end{equation*}
\end{proposition}

\begin{proof}
If $\varpi\mid F^\ast(\bm{c})$, then $\Phi^{\map{G}}_\varpi(\bm{c},s) = 1$.
If $\varpi\nmid F^\ast(\bm{c})$, then \eqref{EQN:smooth-vanishing-S_c} implies
$$\Phi^{\map{G}}_\varpi(\bm{c},s)
= 1 + S^\natural_{\bm{c}}(\varpi) \abs{\varpi}^{-s}
= 1 - \lambda_{\bm{c}}(\varpi) \abs{\varpi}^{-s} - \lambda_V(\varpi) \abs{\varpi}^{-1/2-s},$$
because $S^\natural_{\bm{c}}(\varpi)
= E^\natural_{\bm{c}}(\OK/\varpi\OK)-\abs{\varpi}^{-1/2}E^\natural_F(\OK/\varpi\OK)$ by \eqref{EQN:rewrite-S_c(p)-via-E_c}
and $$\lambda_{\bm{c}}(\varpi)
= \sum_{1\le j\le 10} \alpha_{\bm{c},j}(\varpi)
= -E^\natural_{\bm{c}}(\OK/\varpi\OK),
\quad \lambda_V(\varpi)
= \sum_{1\le j\le 22} \alpha_{V,j}(\varpi)
= E^\natural_F(\OK/\varpi\OK),$$
by the Grothendieck--Lefschetz trace formula.
In either case, $a_{\bm{c},3}(\varpi^l)\ll_\eps \abs{\varpi}^{l\eps}$ follows from the Ramanujan bound $\abs{\alpha}\le 1$ in Proposition \ref{PROP:HW2-consequences}(3).

Now suppose $\varpi\nmid F^\ast(\bm{c})$.
Multiplying $\Phi^{\map{G}}_\varpi(\bm{c},s)$ by
\begin{equation*}
\begin{split}
L_\varpi(s,\bm{c})
&= 1 + \lambda_{\bm{c}}(\varpi) \abs{\varpi}^{-s}
+ \lambda_{\bm{c}}(\varpi^2) \abs{\varpi}^{-2s} + O(\abs{\varpi}^{-3s}), \\
L_\varpi(\tfrac12+s,V)
&= 1 + \lambda_V(\varpi) \abs{\varpi}^{-1/2-s} + O(\abs{\varpi}^{-1-2s}), \\
\zeta_K(2s) L_\varpi(2s,\bm{c},2)
= L_\varpi(2s,\bm{c},\textstyle{\bigwedge^2})
&= 1 + \lambda_{\bm{c},\bigwedge^2}(\varpi) \abs{\varpi}^{-2s} + O(\abs{\varpi}^{-4s}),
\end{split}
\end{equation*}
and discarding all terms of the form $O(\abs{\varpi}^{-s} \times \abs{\varpi}^{-2s})$ or $O(\abs{\varpi}^{-s} \times \abs{\varpi}^{-1/2-s})$, 
we get $$\Phi^{\bm{c},3}_\varpi(s) = 1
+ O(\abs{\varpi}^{-1/2-2s})
+ O(\abs{\varpi}^{-3s}),$$
since $\lambda_{\bm{c}}(\varpi)^2 = \lambda_{\bm{c}}(\varpi^2) + \lambda_{\bm{c},\bigwedge^2}(\varpi)$.
This completes the proof of the proposition. 
\end{proof}

\begin{corollary}
\label{COR:average-bound-on-Phi_3}
Let $Z,R\in \ZZ$ and $A, \eps \in \RR_{>0}$.
Then
\begin{equation*}
\sum_{\bm{c}\in \mathcal{S}_1:\, \norm{\bm{c}}\le \hat Z}
\biggl(\,\sum_{r\in \Omon:\, \abs{r}\le \hat R} \abs{a_{\bm{c},3}(r)}\biggr)^{\!A}
\ll_{A,\eps} \hat Z^6 \hat R^{(1/3+\eps)A}.
\end{equation*}
\end{corollary}

\begin{proof}
By multiplicativity (of $\abs{a_{\bm{c},3}}$) and positivity, we have
\begin{equation}
\label{INEQ:positive-upper-bound-factorization-for-|a_c,3|-sum}
    \sum_{r\in \Omon:\, \abs{r}\le \hat R} \abs{a_{\bm{c},3}(r)}
    \le \sum_{d\in \RcB:\, \abs{d}\le \hat R} \abs{a_{\bm{c},3}(d)}
    \sum_{e\in \RcG:\, \abs{e}\le \hat R} \abs{a_{\bm{c},3}(e)}.
\end{equation}
But by 
Proposition \ref{PROP:approx-Phi-past-1/2},  we have $a_{\bm{c},3}(d)\ll_\eps \abs{d}^\eps$ and
\begin{equation*}
\sum_{e\in \RcG:\, \abs{e}\le \hat R} \frac{\abs{a_{\bm{c},3}(e)}}{\hat R^{1/3+\eps}}
\le \sum_{e\in \RcG} \frac{\abs{a_{\bm{c},3}(e)}}{\abs{e}^{1/3+\eps}}
\le \prod_\varpi \left(1 + \frac{O(\abs{\varpi}^{-1/2})}{\abs{\varpi}^{2/3+2\eps}}
+ \frac{O_\eps(\abs{\varpi}^\eps)}{\abs{\varpi}^{1+3\eps}}\right)
\ll_\eps 1,
\end{equation*}
So by \eqref{INEQ:positive-upper-bound-factorization-for-|a_c,3|-sum}, the quantity to be estimated is
\begin{equation*}
\ll_{A,\eps} \hat R^{(1/3+2\eps) A}
\sum_{\bm{c}\in \mathcal{S}_1:\, \norm{\bm{c}}\le \hat Z}
\card{\set{d\in \RcB: \abs{d}\le \hat R}}^A.
\end{equation*}
Now the claimed bound  follows from Lemma~\ref{LEM:N_c-small-divisor-moment-bound}.
\end{proof}

\begin{conjecture}
\label{CNJ:(R2'E')}
Let $\eps\in (0,1]$.
Let $Z,R\in \ZZ$ with $R\le 3Z$.
Then
\begin{equation}
\label{INEQ:R2'E'-goal}
\sum_{\bm{c}\in \mathcal{S}_1:\, \norm{\bm{c}}\le \hat Z}\,
\Bigl\lvert{
\sum_{r\in \RcG} \bm{1}_{\abs{r}=\hat R} \cdot S^\natural_r(\bm{c})
}\Bigr\rvert^{2-\eps}
\ll_{\eps} \hat Z^6 \hat R^{(2-\eps)/2}.
\end{equation}
\end{conjecture}

\begin{proof}
[Proof assuming Conjecture~\ref{CNJ:(R2'E)}]

Let $\bm{c}\in \mcal{S}_1$ and $R_1,R_2,R_3\in \ZZ_{\ge 0}$.
Writing $$\Sigma^{\bm{c},j}(R_j)\defeq \sum_{r_j\in \Omon:\, \abs{r_j}=\hat R_j} a_{\bm{c},j}(r_j)$$ for brevity, we have
$$
    \sum_{r\in \RcG:\, \abs{r}=\hat R} S^\natural_r(\bm{c})
    = \sum_{R_1+R_2+R_3=R}\, \prod_{1\le j\le 3} \Sigma^{\bm{c},j}(R_j),
$$
since $\Phi^{\map{G}} = \Phi^{\bm{c},1} \Phi^{\bm{c},2} \Phi^{\bm{c},3}$.
Let $\beta\defeq 2-\eps$.
Let $$W_1(R_1)\defeq \hat R_1^\eta,
\quad W_2(R_1)\defeq \hat R_1^{-(\beta-1)\eta},$$
for some $\eta>0$ to be chosen.
Let $$\mcal{I}_1\defeq \sum_{R_1+R_2+R_3=R} W_1(R_1).$$
Using H\"{o}lder's inequality in the form $1 = \frac{\beta-1}{\beta} + \frac{1}{\beta}$
over the set $\set{R_1+R_2+R_3=R}$ (and then exponentiating by $\beta$), we obtain
\begin{equation}
\label{INEQ:first-Holder-for-R2'E'}
\Bigl\lvert{
\sum_{r\in \RcG:\, \abs{r}=\hat R} S^\natural_r(\bm{c})
}\Bigr\rvert^\beta
\le \mcal{I}_1^{\beta-1}
\sum_{R_1+R_2+R_3=R} W_2(R_1)
\prod_{1\le j\le 3} \abs{\Sigma^{\bm{c},j}(R_j)}^\beta,
\end{equation}
since $\beta\ge 1$ and $W_1^{\beta-1} W_2 = 1$.
Here $$\mcal{I}_1
\le \sum_{R_2,R_3\ge 0} \frac{\hat R^\eta}{(\hat R_2\hat R_3)^\eta}
\ll_\eta \hat R^\eta.$$
Upon summing \eqref{INEQ:first-Holder-for-R2'E'} over $\norm{\bm{c}}\le \hat Z$, we thus find that the left-hand side of \eqref{INEQ:R2'E'-goal} is
\begin{equation}
\label{INEQ:simplified-Holder-and-Fubini-bound-for-R2'E'}
\ll_\eta
\sum_{R_1+R_2+R_3=R} (\hat R/\hat R_1)^{(\beta-1)\eta}
\sum_{\bm{c}\in \mathcal{S}_1:\, \norm{\bm{c}}\le \hat Z}\,
\prod_{1\le j\le 3} \abs{\Sigma^{\bm{c},j}(R_j)}^\beta.
\end{equation}

Let $(\gamma_1, \gamma_2, \gamma_3)\defeq (2, 4\beta/\eps, 4\beta/\eps)$.
Then $\sum_{1\le j\le 3} \beta/\gamma_j = 1$, so by H\"{o}lder over $\bm{c}$, we obtain
\begin{equation}
\label{INEQ:second-Holder-for-R2'E'}
\sum_{\bm{c}\in \mathcal{S}_1:\, \norm{\bm{c}}\le \hat Z}\,
\prod_{1\le j\le 3} \abs{\Sigma^{\bm{c},j}(R_j)}^\beta
\le \prod_{1\le j\le 3} \mscr{M}_j^{\beta/\gamma_j},
\end{equation}
where  $\mscr{M}_j = \sum_{\bm{c}\in \mathcal{S}_1:\, \norm{\bm{c}}\le \hat Z} \abs{\Sigma^{\bm{c},j}(R_j)}^{\gamma_j}$.
By Conjecture~\ref{CNJ:(R2'E)} (with $\ceil{\max(Z,\frac13R_1)}$ in place of $Z$ and $R_1$ in place of $R$), we have 
$$\mscr{M}_1 \ll (\hat Z+\hat R_1^{1/3})^6 \hat 
R_1.
$$
Also, by \eqref{INEQ:good-moduli-restricted-wedge2E-goal} 
and Corollary \ref{COR:average-bound-on-Phi_3},
we have
$$\mscr{M}_2 \ll_{\gamma_2} \hat Z^6 \hat R_2^{11\gamma_2/40}
\quad \text{ and }\quad 
\mscr{M}_3 \ll_{\gamma_3} \hat Z^6 \hat R_3^{11\gamma_3/30}.
$$
Therefore, if $R_1+R_2+R_3=R$, then the right-hand side of \eqref{INEQ:second-Holder-for-R2'E'} is
\begin{equation*}
\ll_\beta \hat Z^{\sum_j 6\beta/\gamma_j} \hat R_1^{\beta/2} \hat R_2^{11\beta/40} \hat R_3^{11\beta/30}
= \hat Z^6 (\hat R^{1/2} \hat R_2^{-9/40} \hat R_3^{-2/15})^\beta.
\end{equation*}
Plugging \eqref{INEQ:second-Holder-for-R2'E'} into \eqref{INEQ:simplified-Holder-and-Fubini-bound-for-R2'E'} now bounds the left-hand side of \eqref{INEQ:R2'E'-goal} by
\begin{equation}
\label{EXPR:final-upper-bound-quantity-for-R2'E'}
\ll_{\eta,\beta}
\sum_{R_1+R_2+R_3=R} \frac{(\hat R/\hat R_1)^{(\beta-1)\eta}
\hat Z^6 \hat R^{\beta/2}}{(\hat R_2^{9/40} \hat R_3^{2/15})^\beta}.
\end{equation}
Let $\eta = \min(9/40, 2/15)$.
Then $R_1+R_2+R_3=R \Rightarrow (\hat R/\hat R_1)^\eta \le \hat R_2^{9/40} \hat R_3^{2/15}$.
Since $\beta-1\ge 0$, it follows that the quantity \eqref{EXPR:final-upper-bound-quantity-for-R2'E'} is
\begin{equation*}
\ll_\beta \sum_{R_1+R_2+R_3=R} \frac{\hat Z^6 \hat R^{\beta/2}}{\hat R_2^{9/40} \hat R_3^{2/15}}
\le \sum_{R_2,R_3\ge 0} \frac{\hat Z^6 \hat R^{\beta/2}}{\hat R_2^{9/40} \hat R_3^{2/15}}
\ll \hat Z^6 \hat R^{\beta/2}.
\end{equation*}
This implies the desired inequality \eqref{INEQ:R2'E'-goal}.
\end{proof}

\section{Integral estimates}
\label{SEC:integral-stuff}

Let $\gamma \in K_\infty$, $\bm{w}\in K_\infty^n$ and $G\in K_\infty[x_1,\dots, x_n]$. For a fixed smooth weight function $w\colon K_\infty^n\to \RR$ with compact support, we are interested in integrals of the form 
\[
J_{G,w}(\gamma, \bm{w})\coloneqq \int_{K_\infty}w(\bm{x})\psi(\gamma G(\bm{x})+\bm{w}\cdot \bm{x})\dd\bm{x},
\]
together with their averages 
\[
J^\Gamma_{G,w}(\bm{w})=\int_{|\gamma|<\widehat{\Gamma}}J_{G,w}(\gamma,\bm{w})\dd\gamma ,
\]
for $\Gamma\in \ZZ$. (In fact, the dependence of the latter integral on $\Gamma$ is mild, as we will soon see.)
With this notation we have $I_r(\theta,\bm{c})=J_{F,w}(\theta P^3,P\bm{c}/r)$ in \eqref{Eq: Definition I_r(theta,c)} and 
$I_r(\bm{c})=J_{F,w}^{\Gamma}(P\bm{c}/r)$ in 
\eqref{eq:2-1 to brighton},
where $\Gamma=-\deg(r)-Q$.

When $w$ is the indicator function of $\TT^n$, we shall write $J_G(\gamma,\bm{w})$ and $J^\Gamma_G(\bm{w})$ for the sake of brevity. Moreover, we let $H_G$ be the maximum of the absolute values of the coefficients of $G$ and refer to it as the height of $G$.

This level of generality, where $G$ is allowed to vary, will be useful when changing variables to analyse $J^\Gamma_{F,w}(\bm{w})$.
Let us recall Lemmas 2.4 and 2.7 of \cite{browning2015rational}. 
\begin{lemma}\label{Le: IntVanish}
    Let $G\in K_\infty[x_1,\dots, x_n]$ and $\bm{w}\in K_\infty^n$. Suppose $|\bm{w}|\geq 1$ and $|\bm{w}|>H_G$. Then 
    \[
    \int_{\TT^n}\psi(G(\bm{x})+\bm{w}\cdot\bm{x})\dd\bm{x}=0.
    \]
\end{lemma}

\begin{lemma}\label{Le:StatPhase}
    Let 
    \[
\Omega= \{\bm{x}\in\TT^n \colon |\gamma\nabla G(\bm{x})+\bm{w}|\leq H_G\max\{1, |\gamma|^{1/2}\}\}.
     \]
     Then we have
    \[
    J_G(\gamma,\bm{w})= \int_{\Omega}\psi(\gamma G(\bm{x})+\bm{w}\cdot \bm{x})\dd\bm{x}. 
    \]
\end{lemma}

We now fix a non-singular cubic form $F\in \OK[x_1,\dots, x_n]$, which we will eventually take to be our initial cubic form $x_1^3+\cdots + x_6^3$, but is allowed to be arbitrary in this section. Before we start our  investigation, we impose the following conditions on the weight function $w$.

\begin{hypothesis}\label{hyp}
Let  $H(\bm{x})=(\frac{\partial^{i+j}F(\bm{x})}{\partial x_i\partial x_j})$ be the Hessian associated to  $F$.
\begin{enumerate}[(i)]
    \item If $\bm{x}\in\supp(w)$, then $\det H(\bm{x})\neq 0$.
    \item There exist $\bm{x_0}\in K_\infty^n$ with $\norm{\bm{x}_0} \le 1$, and an integer $L\geq 0$, such that 
    \[
    w(\bm{x})=\begin{cases}
        1 &\textnormal{if }|\bm{x}-\bm{x}_0|<\widehat{L}^{-1},\\
        0&\textnormal{else.}
    \end{cases}
    \]
\end{enumerate}
\end{hypothesis}

 If $\cha(\FF_q)\in \{2,3\}$, then $\det H(\bm{x})$ vanishes identically and so (i) is impossible in this case. Therefore, all of our results in this section have the implicit assumption that $\cha(\FF_q)>3$.
 Moreover, condition (ii) is non-restrictive, since any compactly supported smooth function $w'\colon K_\infty^n\to \RR$ can be written as 
$w'= s_1 w_1+\cdots +s_mw_m$, 
where $s_i\in \RR$ and $w_i$ are indicator functions of some compact ball in $K_\infty^n$.
Part (ii) of Hypothesis \ref{hyp} implies the crucial symmetry property
\begin{equation}
\label{EQN:weight-function-symmetry}
\textnormal{$w(\bm{x})=w((1+t^{-L}\gamma)\bm{x})$ \quad for all $\gamma\in \TT$.}
\end{equation}
In other words, $w$ is invariant under scaling by $1+t^{-L}\TT$. Throughout this section, all implied constants are allowed to depend on the choice of $w$ and in particular implicitly also on the parameter $L$ and the vector $\bm{x}_0$ in the definition of $w$.

For the purposes of Theorem \ref{THM:pos-density}, we will be focusing on  diagonal cubic forms.
In this setting we shall work with the following explicit weight function. 

\begin{definition}\label{def:our-weight}
Let $F(\bm{x})=a_1x_1^3+\cdots+a_nx_n^3$, for $a_1,\dots,a_n\in \FF_q^\times$ and let $\bm{x}_0\in (\FF_q^\times)^n$ be such that $F(\bm{x}_0)=0$. Then we define
$$
 w(\bm{x})=\begin{cases}
        1 &\text{if }|\bm{x}-\bm{x}_0|<1,\\
        0&\text{else.}
    \end{cases}
$$
\end{definition}

We check that such weight functions satisfy 
Hypothesis \ref{hyp} when $F$ is a diagonal cubic  form with coefficients in $\FF_q^\times$ and $\cha(\FF_q)>3$.
Condition  (i) can then be rephrased as $x_1\cdots x_n\neq 0$ for all $\bm{x}\in \supp(w)$, which is obvious.  Condition (ii) clearly holds with $L=0$.

Returning to the setting of general weight functions $w$ satisfying 
Hypothesis \ref{hyp}, we begin with the following estimate.
\begin{lemma}\label{Le: IntEstimateI}
For $\Gamma \in \ZZ$ and $\bm{w}\in K_\infty^n$, we have 
\[
J^\Gamma_{F,w}(\bm{w})\ll_{w,F} (1+|\bm{w}|)^{1-n/2}.
\]
Moreover, if $|\bm{w}|\gg 1$, then $J_{F,w}(\gamma, \bm{w})=0$ unless $|\gamma|\asymp |\bm{w}|$.
\end{lemma}
\begin{proof}
In this proof all implied constants are allowed to depend on $F$.
Upon making the change of variables $\bm{y}=t^L(\bm{x}-\bm{x}_0)$, it follows that
    \[
    J_{F,w}(\gamma,\bm{w})=\widehat{L}^{-n}\psi(\bm{w}\cdot\bm{x_0})J_{G}(\gamma, t^{-L}\bm{w}),
    \]
    where $G(\bm{y})=F(t^{-L}\bm{y}+\bm{x}_0)$. It is clear that $H_G\ll H_F\ll 1$. At this point we can use Lemma \ref{Le: IntVanish} to conclude that $J_{F,w}(\gamma, \bm{w})=0$ unless $|\bm{w}|\ll \max\{1, |\gamma|\}$. Furthermore, by Lemma \ref{Le:StatPhase} 
\begin{align*}
    |J_{F,w}(\gamma, \bm{w})| &\leq \widehat{L}^{-n}\text{meas}\{\bm{y}\in \TT^n\colon |\gamma \nabla G(\bm{y})+t^{-L}\bm{w}| \leq H_G\max\{1,|\gamma|^{1/2}\} \}\\
    & \leq \text{meas}\{\bm{x}\in\supp(w)\colon |\gamma \nabla F(\bm{x})+t^{-L}\bm{w}|\ll  \max\{1, |\gamma|^{1/2}\}\}.
\end{align*}
Let us denote the set whose measure we want to estimate by $\Omega$. Since $\bm{0}\not\in \supp(w)$, we must have $|\nabla F(\bm{x})|\asymp 1$ for all $\bm{x}\in \supp(w)$. If $|\bm{w}|\ll 1$, then $\Omega$ is empty unless $|\gamma|\ll1$, so that we can use the trivial estimate $|J_{F,w}(\gamma,\bm{w})|\leq 1$  to conclude that $J^\Gamma_{F,w}(\bm{w})\ll 1$ if $|\bm{w}|\ll1 $. In particular, we may assume $|\bm{w}|\gg 1$ for the rest of the proof. In this case the contribution from $|\gamma|\ll 1$ vanishes, since then $\Omega$ is empty, while for $|\gamma|\gg 1$ the defining inequality for $\Omega $ becomes
\[
|\gamma \nabla F(\bm{x})+t^{-L}\bm{w}|\ll |\gamma|^{1/2}.
\]
This can only hold if $|\gamma|\asymp |\bm{w}|$. From this we deduce the statement that $J_{F,w}(\gamma,\bm{w})=0$ unless $|\gamma|\asymp |\bm{w}|$.

It remains to give an upper bound.
This will be similar to \cite{browning2015rational}*{proof of Lemma~7.3}.
Let us now assume that $\bm{x}, \bm{x}+\bm{x}'$ are both in $\Omega$. Note that $\nabla F(\bm{x}+\bm{x}')-\nabla F (\bm{x})= H(\bm{x}+\bm{x}'/2)\bm{x}'$, so that we must then have \begin{equation}\label{Eq: hessInequ}
|H(\bm{x}+\bm{x}'/2)\bm{x}'|\ll |\gamma|^{-1/2}.
\end{equation}
As $\det H(\bm{y})\neq 0$ for all $\bm{y}\in \supp(w)$, compactness of $\supp(w)$ implies that we have $|\det H(\bm{y})|\gg 1$ for all $\bm{y}\in\supp(w)$. In addition, a straightforward computation shows that $\bm{x}+\bm{x}'/2\in \supp(w)$, so that we can multiply \eqref{Eq: hessInequ} from the left with the inverse of $H(\bm{x}+\bm{x}'/2)$, whose entries have absolute value $O(1)$, to conclude that $|\bm{x}'|\ll |\gamma|^{-1/2}\ll |\bm{w}|^{-1/2}$. Finally, from this we obtain 
\begin{align*}
J^\Gamma_{F,w}( \bm{w}) \ll \int_{|\gamma|\asymp |\bm{w}|}\text{meas}(\Omega)\dd\gamma \ll |\bm{w}|^{1-n/2}    
\end{align*}
when $|\bm{w}|\gg 1$, as desired. 
\end{proof}

The following symmetry property of $J^\Gamma_{F,w}(\bm{w})$
generalises \cite{glas2022question}*{Lemma 3.6}.

\begin{lemma}\label{LEM:integral-scale-invariance}
    Suppose $\lambda_1,\lambda_2 \in K_\infty^\times$ are such that
    $\lambda_1/\lambda_2\in 1+t^{-L}\TT$.
    Then 
    \begin{equation*}
    J^\Gamma_{F,w}(\lambda_1\bm{w})=J^\Gamma_{F,w}(\lambda_2\bm{w})
    \quad\textnormal{for any $\bm{w}\in K_\infty^n$}.
    \end{equation*}
\end{lemma}

\begin{proof}
    Let us write $\lambda=\lambda_1$ and show that for $\bm{w}$ fixed, the value of $J^\Gamma_{F,w}(\lambda\bm{w})$ only depends on the image of $\lambda$ in the quotient group $K_\infty^\times / (1+t^{-L}\TT)$. We have 
    \begin{align*}
        J^\Gamma_{F,w}(\lambda\bm{w}) &= \int_{|\gamma|<\widehat{\Gamma}}\int_{K_\infty^n}w(\bm{x})\psi(\gamma F(\bm{x})+\lambda \bm{w}\cdot\bm{x})\dd \bm{x}\dd\gamma\\
        & = |\lambda|^n \int_{K_\infty^n}w(\lambda^{-1}\bm{y})\psi(\bm{w}\cdot \bm{y})\int_{|\gamma|<\widehat{\Gamma}}\psi(\gamma \lambda^{-3}F(\bm{y}))\dd\gamma \dd\bm{y},
    \end{align*}
    where we applied the change of variables $\bm{y}=\lambda \bm{x}$. It follows from \eqref{Eq: OrthogonalityOfCharacters} that 
    \[
    \int_{|\gamma|<\widehat{\Gamma}}\psi(\gamma \lambda^{-3}F(\bm{y}))\dd\gamma = \begin{cases}
        \widehat{\Gamma} &\text{if }|F(\bm{x})|< \widehat{\Gamma}^{-1}|\lambda|^3,\\
        0&\text{else.}
    \end{cases}
    \]
    This formula clearly only depends on $|\lambda|$. 
    Moreover, 
    $w(\lambda^{-1}\bm{y})$ depends only on $\lambda\bmod{1+t^{-L}\TT}$ by \eqref{EQN:weight-function-symmetry}. Since we assume $L\geq 0$, any two elements of $K_\infty^\times$ whose image in $K_\infty^\times/ (1+t^{-L}\TT)$ coincides must have the same absolute value, from which we deduce that $J^\Gamma_{F,w}(\lambda \bm{w})$ only depends on $\lambda\bmod{1+t^{-L}\TT}$. 
\end{proof}

The above establishes most of the analogue over $K$ of \cite{wang2023ratios}*{Proposition 8.1}.
It still remains to show that $J^\Gamma_{F,w}(\bm{w})$ vanishes unless $|F^*(\bm{w})|$ is small.
Working over function fields significantly simplifies the execution of the key ideas,
because fewer dyadic ranges intervene.
We start with an auxiliary lemma.

\begin{lemma}\label{Le: InFunctDual}
    Let $R$ be a compact set such that 
$
R\subset
\{\bm{x}\in K_\infty^n\colon \det H(\bm{x})\neq 0\}$.
 Then there exists a constant $C=C(R)>0$ such that if $\bm{w}\in K_\infty^n$ and $\bm{x}\in R$ satisfy
    \[
    |\nabla F(\bm{x})-\bm{w}|<\widehat{C}^{-1},
    \]
it follows that there exists $\bm{s}
\in K_\infty^n$
with $\nabla F(\bm{s})= \bm{w}$ and 
\[
|\bm{s}-\bm{x}|< \widehat{C}|\nabla F(\bm{x})- \bm{w}|.
\]
\end{lemma}
\begin{proof} 
Let $\bm{x}\in R$. Since $\det H(\bm{x}) \neq 0$, it follows from the inverse function theorem for local fields
\cite{serre2009lie}*{Part II, \S~{III.9}, Theorem 2}
that there exist open neighbourhoods $U, U'\subset K_\infty^n$, with $\det H(\bm{y})\neq 0$ for any $\bm{y}$ in the closure of $U$, such that $\bm{x}\in U$, $\nabla F(\bm{x})\in U'$ and $\nabla F\colon U \to U'$ is an isomorphism.
In addition, after shrinking $U'$ if necessary, we may assume that it is an open ball of radius $\hat C_0^{-1}$ around $\nabla F(\bm{x})$, for some real $C_0=C_0(\bm{x})\ge 0$.
It follows that if $\bm{w} \in U'$,  there exists $\bm{s}\in U$ such that $\nabla F(\bm{s})= \bm{w}$. Therefore, we have 
\begin{align*}
    \bm{w}= \nabla F (\bm{s})= \nabla F(\bm{s}-\bm{x}) + H(\bm{x})(\bm{s}-\bm{x})+\nabla F(\bm{x}).
\end{align*}
As $\bm{x}\in R$, we have $\det H(\bm{x})\neq 0$, and hence 
\[
\bm{s}-\bm{x}= H(\bm{x})^{-1}(\bm{w}-\nabla F(\bm{x}) + \nabla F(\bm{s}-\bm{x})).
\]
Moreover, as $R$ is compact, the entries of $H(\bm{x})^{-1}$ are $O_R(1)$. It follows that there exist constants $C_1, C_2\geq 0$ depending only on  $R$ and $F$ such that 
\[
|\bm{s}-\bm{x}|\leq \widehat{C}_1|\bm{w}-\nabla F(\bm{x})| +\widehat{C}_2 |\bm{s}-\bm{x}|^2. 
\]
Upon shrinking $U$ if necessary, we may assume that $|\bm{s}-\bm{x}|< (2\widehat{C}_2)^{-1}$, so that we obtain 
\[
|\bm{s}-\bm{x}|\leq 2\widehat{C}_1|\bm{w}-\nabla F(\bm{x})|.
\]
Since $R$ is compact, we can cover it with finitely many open sets $U$ and then take $C$ to be the maximum of the constants that arise in this process.
\end{proof}

\begin{lemma}
\label{LEM:new-vanishing-for-small-dual-form}
    Assume $L=0$ in condition (ii) on the weight function $w$.
    Let $\Gamma \in \ZZ$ and $\bm{w}\in K_\infty^n$.
    Then $J^\Gamma_{F,w}(\bm{w})=0$ unless $|F^*({\bm{w}})|\ll 1+ |\bm{w}|^{\deg F^*-1}$. 
\end{lemma}
\begin{proof}
Note that if $|\bm{w}|\ll 1$, then $|F^*({\bm{w}})|\ll 1+|\bm{w}|^{\deg F^*-1}$ holds trivially, so that we may assume $|\bm{w}|\gg 1$ from now on. Similarly we can assume $1\ll |\bm{w}|\ll \widehat{\Gamma}$, since else $J^\Gamma_{F,w}(\bm{w})=0$ by Lemma \ref{Le: IntEstimateI}. Using Lemma \ref{Le: IntEstimateI} again, it follows that 
\[
J^\Gamma_{F,w}(\bm{w})=\int_{|\gamma|\asymp |\bm{w}|}
J_{F,w}(\gamma,\bm{w})
\dd\gamma. 
\]

Let $Z\in\ZZ_{>0}$ be such that $\widehat{Z}\asymp |\bm{w}|$.
We will now consider the contribution to $J^\Gamma_{F,w}(\bm{w})$ from those $\gamma$ with $|\gamma|=\widehat{Z}$.
To do so, write $\gamma = a t^Z(1+\gamma')$, where $a\in\FF_q^\times$ and $\gamma'\in\TT$.
Let $\bm{y} \defeq \bm{x}-\bm{x}_0$ and $G(\bm{y})\defeq F(\bm{y}+\bm{x}_0)$, as in the proof of Lemma \ref{Le: IntEstimateI}.
Then let
\begin{equation*}
\bm{x}'
\defeq (1+\gamma')^{-1/2} \bm{x}
= (1+\gamma')^{-1/2} (\bm{y}+\bm{x}_0).
\end{equation*}
Note that since $(1+\gamma')^{-1/2} - 1 \in \TT$ and $L=0$, property \eqref{EQN:weight-function-symmetry} of the weight function $w$ implies $w(\bm{x})=w(\bm{x}')$.
In particular, after the change of variables $\bm{x}\mapsto (1+\gamma')^{-1/2}\bm{x}$ by condition (ii) on $w$, we get
\begin{align*}
\int_{\TT^n}w(\bm{x}')&\psi(\gamma F(\bm{x}')+\bm{w}\cdot\bm{x}')\dd\bm{x}\\
&=\psi((1+\gamma')^{-1/2}\bm{x}_0\cdot\bm{w})\int_{\TT^n}\psi((1+\gamma')^{-1/2}(at^ZG(\bm{y})+\bm{y}\cdot\bm{w}))\dd\bm{y}\\
&= \psi((1+\gamma')^{-1/2}\bm{x}_0\cdot\bm{w})\int_{\Omega}\psi((1+\gamma')^{-1/2}(at^ZG(\bm{y})+\bm{y}\cdot\bm{w}))\dd\bm{y},        
\end{align*}
by Lemma \ref{Le:StatPhase}, where 
\[
\Omega \defeq \{\bm{y}\in\TT^n\colon |at^Z\nabla G(\bm{y})+\bm{w}|\ll \widehat{Z}^{1/2}\}
\]
and we used that $\gamma F(\bm{x}')=at^Z (1+\gamma')^{-1/2}F(\bm{y}+\bm{x}_0)$ since $F$ is homogeneous of degree $3$.
In summary, we have proven the identity $$J_{F,w}(\gamma,\bm{w}) = J_{F,w}((1+\gamma')^{-3/2}\gamma,(1+\gamma')^{-1/2}\bm{w}),$$ and applied stationary phase to the latter integral to get an integral over $\bm{y}\in \Omega$.

Upon expressing the integral over $\bm{y}\in \Omega$ in terms of $\bm{x} = \bm{y}+\bm{x}_0$, we find that
\[
J_{F,w}(\gamma,\bm{w})
= \int_{\TT^n}w(\bm{x}')\psi(\gamma F(\bm{x}')+\bm{w}\cdot\bm{x}')\dd\bm{x}'
= \int_{\Omega'}\psi((1+\gamma')^{-1/2}\Phi(\bm{x}))\dd\bm{x},
\]
where $\Phi(\bm{x})\defeq at^ZF(\bm{x})+\bm{w}\cdot \bm{x}$
and
\[
\Omega'\defeq \{\bm{x}\in\supp(w)\colon |a\nabla F(\bm{x})+t^{-Z}\bm{w}|\ll \widehat{Z}^{-1/2}\}.
\]
The crucial observation here is that $\Omega'$ is independent of $\gamma'$.
In particular, when considering the total contribution to $J^\Gamma_{F,w}(\bm{w})$ from the set $\set{\gamma = at^Z(1+\gamma'): \gamma'\in \TT}$, we may exchange the order of integration and see that it is given by 
\[    \int_{\Omega'}\int_{\TT}\psi((1+\gamma')^{-1/2}\Phi(\bm{x}))\dd\gamma' \dd{\bm{x}}.
\]
Here $\Omega'$, $\Phi$ depend on $Z$ and $a$, but all estimates below will be uniform over $Z$ and $a$. Let us now take the Taylor expansion 
\[
(1+\gamma')^{-1/2}=\sum_{k\geq 0}a_k \gamma'^{k},
\]
where $a_k=\binom{-1/2}{k}\in \FF_q$.
This is well defined, because $|\gamma'|<1$ and $\cha(\FF_q)\neq 2$.
Since $|\gamma'|<1$, the value of $\psi((1+\gamma')^{-1/2}\Phi(\bm{x}))$ only depends on $f(\gamma')=\sum_{k=0}^{k_0}a_k \gamma'^k$ if we choose $k_0$ sufficiently large. We may now apply Lemma \ref{Le:StatPhase} again with $n=1$ to infer that the inner integral is given by 
\[
\int_{\Theta} \psi(f(\gamma')\Phi(\bm{x}))\dd\gamma',
\]
where 
\[
\Theta = \{\gamma'\in\TT\colon |f'(\gamma')\Phi(\bm{x})|\leq \max\{1, |\Phi(\bm{x})|^{1/2}\}\}.
\]
It is easy to see that $|f'(\gamma')|=1$, so that $\Theta$ is non-empty only if $|\Phi(\bm{x})|\leq 1$, which we shall assume for the rest of the proof.

Since $\bm{x}\in \Omega'$, we have $\bm{x}\in\supp(w)$ and $|a\nabla F(\bm{x})+t^{-Z}\bm{w}|\ll \widehat{Z}^{-1/2}\ll |\bm{w}|^{-1/2}$. Since we assume $|\bm{w}|\gg 1$ for a sufficiently large implied constant, we may invoke Lemma \ref{Le: InFunctDual} to deduce the existence of $\bm{s}\in K_\infty^n$ such that $|\bm{s}-\bm{x}|\ll |a \nabla F(\bm{x})+t^{-Z}\bm{w}|\ll |\bm{w}|^{-1/2}$ and $\nabla F(\bm{s})=-a^{-1}t^{-Z}\bm{w}$. Using Taylor expansion at $\bm{s}$ it follows that 
\begin{align*}        
|aF(\bm{x})+t^{-Z}\bm{w}\cdot \bm{x}|&= |aF(\bm{s})+t^{-Z}\bm{w}\cdot\bm{s}|+O(|\bm{x}-\bm{s}|^2)\\
&=|aF(\bm{s})+t^{-Z}\bm{w}\cdot\bm{s}|+O(|\bm{w}|^{-1}).
\end{align*}
Moreover, as $a\nabla F(\bm{s})=-t^{-Z}\bm{w}$, we have $3a F(\bm{s})=-t^{-Z}\bm{w}\cdot\bm{s}$ and hence $${|a F(\bm{s})+t^{-Z}\bm{w}\cdot \bm{s}|=|F(\bm{s})|}.$$

By \eqref{Eq: FdividesDualFormGrad}, the ratio $F^*(\nabla F(\bm{x})) / F(\bm{x})$ is a homogeneous polynomial in $\OK[x_1,\dots, x_n]$ of degree $2\deg F^*-3$,
which immediately implies 
\[
|F^*(\nabla F(\bm{s}))|\ll |F(\bm{s})||\bm{s}|^{2\deg F^*-3}.
\]
In addition, as $|\bm{s}-\bm{x}|\ll |\bm{w}|^{-1/2}$, we have $|\bm{s}|\asymp 1$. 
Together with $-a^{-1}t^{-Z}\bm{w}=\nabla F(\bm{s})$, this implies 
$|F^*(t^{-Z}\bm{w})|=|F^*(\nabla F(\bm{s}))|\ll |F(\bm{s})|$
and hence
\[
|F^*(\bm{w})|\ll \widehat{Z}^{\deg F^*}(\widehat{Z}^{-1}|\Phi(\bm{x})|+|\bm{w}|^{-1})\ll |\bm{w}|^{\deg F^*-1},
\]
as desired.
\end{proof}

\section{Estimates for exponential sums}

As usual we assume that $K=\FF_q(t)$ with 
$\OK=\FF_q[t]$ and 
$\Char(\FF_q)>3$.
Throughout this section we
assume that
$$
F(\bm{x})=a_1x_1^3+\dots+a_nx_n^3,
$$
where $a_1,\dots,a_n\in \FF_q^\times$.
Let $r\in \Omon$ and let $\bm{c}\in\OK^n$. In this section we collect estimates for the exponential sums
$S_r(\bm{c})$ (defined in \eqref{Eq: Definition S_r(c)})
and the normalisation $S_r^\natural (\cc)=|r|^{-(n+1)/2}S_r(\cc)$.
As we will see, the quality of our estimates is governed by the behaviour of the dual form $F^*\in \FF_q[c_1,\dots, c_n]$,
whose formula is recorded in \eqref{EQN:6-variable-Fermat-dual-form} in the case $F=x_1^3+\dots+x_n^3$.

By \eqref{Eq: Sr(c)multiplicative}, it suffices to treat the case of prime power moduli when estimating $S_r(\bm{c})$. 
Define
\begin{equation}\label{eq:sq-cub}
    \sq(c)=\prod_{\varpi^2\mid c} \varpi^{v_\varpi(c)} \quad \text{ and }
\quad
\cub(c)=\prod_{\varpi^3\mid c} \varpi^{v_\varpi(c)}
\end{equation}
for the square-full and cube-full parts of any $c\in \OK$;
if $c=0$ we take $\sq(c)=\cub(c)=0$.
With this in mind 
we summarise the relevant facts in the following result. 

\begin{lemma}\label{lem:prime_power}
Let $\varpi^e\in \OK$ be a prime power. 
\begin{enumerate}
\item
If $e=1$ then $S_\varpi^\natural(\bm{c})\ll |\gcd(\varpi,\nabla F^*(\bm{c}))|^{1/2}\leq |\varpi|^{1/2}$.
\item
If $e=2$ then 
$S_{\varpi^2}^\natural(\bm{c})\ll |\varpi|$.
\item
 If $e\geq 3$ then 
 $$
 S_{\varpi^e}^\natural(\bm{c})\ll |\varpi^e|^{1/2} \prod_{1\leq i\leq n}
 |\gcd(\varpi^e,\sq(c_i))|^{1/4}.
 $$
\end{enumerate}
\end{lemma}

\begin{proof}
Part (1) is due to 
Hooley \cite{hooley1994nonary}*{Lemma 60}.
Part (2) is straightforward and is recorded 
in \cite{glas2022question}*{Eq.~(4.7)}.
To see part (3) we invoke 
\cite{glas2022question}*{Eq.~(4.7)} to deduce that 
$$
S_{\varpi^e}^\natural(\bm{c})\ll 
\frac{1}{|\varpi^e|^{(n+1)/2}} \cdot 
|\varpi^e| \prod_{1\leq i\leq n} |\varpi^e|^{1/2} \{\varpi^e,c_i\}^{1/4},
$$
where $\{\varpi^e,c_i\}=|\varpi|^{-1}$ if $\varpi\| c_i$, with 
$\{\varpi^e,c_i\}=|\gcd(\varpi^e,c_i)|$ otherwise. 
The statement of the lemma follows on noting that 
$\{\varpi^e,c_i\}\leq |\gcd(\varpi^e,\sq(c_i))|$.
\end{proof}

The following result  follows from combining Lemma 
\ref{lem:prime_power} with multiplicativity. 

\begin{lemma}\label{lem:lem8.2}
There exists a constant $A>0$ such that 
$$
|r|^{-1/2}|S_r^\natural(\cc)| \leq A^{\omega(r)} \prod_{1\leq i\leq n}
\abs{\gcd(\cub(r), \sq(c_i))}^{1/4}.
$$
\end{lemma}

It will also be useful to record a crude  upper bound, which will be helpful when $\bm{c}=\bm{0}$.

\begin{lemma}\label{lem:lem8.2'}
There exists a constant $A>0$ such that 
$$
|r|^{-1/2}|S_r^\natural(\cc)| \leq A^{\omega(r)} |\cub(r)|^{n/6}. 
$$
\end{lemma}

\begin{proof}
By multiplicativity, it suffices to show that 
$
S_{\varpi^e}^\natural(\cc) \ll |\varpi^e|^{1/2+\theta_e}$,
for any prime power $\varpi^e\in \mathcal{O}$,
where 
$$
\theta_e=
\begin{cases}
    0 & \text{ if $e\leq 2$,}\\
    n/6 & \text{ if $e\geq 3$.}
\end{cases}
$$
For $e\leq 2$ this follows from parts (1) and (2) of Lemma 
\ref{lem:prime_power}. For $e\geq 3$, we apply the $1$-dimensional Hua type estimate recorded in \cite{glas2022question}*{\S~4.2}.
\end{proof}

The following classical recursive structure will soon prove useful.

\begin{lemma}
\label{LEM:ad-hoc-reduce-to-primitive-c's}
If $\varpi\mid \bm{c}$ and $e\ge 4$,
then 
$$
S_{\varpi^e}(\bm{c}) = 
\begin{cases}
0 & \text{ if $\varpi^2\nmid \bm{c}$,}\\ 
\abs{\varpi}^{3+2n} S_{\varpi^{e-3}}(\bm{c}/\varpi^2)
&\text{ if $\varpi^2\mid \bm{c}$.}
\end{cases}
$$
\end{lemma}

\begin{proof}
After expressing $S_{\varpi^e}(\bm{c})$ in terms of one-dimensional sums,
this is immediate from the function field analogues of \cite{hooley1986Lfunctions}*{Eq.~(44)--(45)},
which hold since  $\cha(\FF_q)\ne 3$.
\end{proof}

Since a priori (depending on one's convention\footnote{Our definition of $F^\ast$ differs from the classical ``eliminant'' of \cite{hooley1988nonary}*{p.~37, footnote 2}.})
the dual form is only defined up to scaling, and we will soon need to understand its relation to discriminants from 
work of Bus\'e and Jouanolou
\cite{buse2014discriminant}*{Definition 4.6}, we record the following standard fact.

\begin{proposition}\label{PROP:sliced-discriminant}
Let $R=\FF_q[c_1,\dots,c_{n-1}]$.
Then $F^\ast(c_1,\dots,c_{n-1},1) \in R$
is divisible by the discriminant $\Delta\in R$
of the cubic form
$$F(y_1,\dots,y_{n-1},-c_1y_1-\dots-c_{n-1}y_{n-1})
\in R[y_1,\dots,y_{n-1}].$$
\end{proposition}

\begin{proof}
Here $\Delta$ is defined by a specialisation process explained in \cite{buse2014discriminant}*{paragraph after Eq.~(4.2.1)}.
Write $\bm{c}=(c_1,\dots,c_{n-1},1)$ and
let $\disc(F,\bm{c})$ be the discriminant associated in \cite{terakado2018determinant}*{\S~1.1} to the pair $(F(\bm{x}), \bm{c}\cdot \bm{x})$.

As  explained in \cite{wang2023dichotomous}*{Proposition~4.4}, for example, 
we have  $\disc(F,\bm{c}) / F^\ast(\bm{c}) \in \FF_q^\times$. (This is the only step where we use $\cha(\FF_q) > 3$.)
Also
\begin{equation*}
    \disc(F,\bm{c})
    = \pm \disc(F(y_1,\dots,y_{n-1},-c_1y_1-\dots-c_{n-1}y_{n-1})),
\end{equation*}
by \cite{wang2023dichotomous}*{Proposition~3.2}.
It therefore follows that $\Delta / F^\ast \in \FF_q^\times$, which suffices.
\end{proof}

The following result is the function field analogue of \cite{wang2023ratios}*{Lemma~9.1 (2),(3)}.

\begin{lemma}\label{lem:VW9.1}
Let $\varpi^e\in \OK$ be a prime power. 
\begin{enumerate}
\item
If $e=2$ and $v_\varpi(F^*(\cc))\leq 1$ then 
$S_{\varpi^2}^\natural(\bm{c})\ll 1$.
\item
If $e\geq 2+v_\varpi(F^*(\cc))$ then $S_{\varpi^e}(\bm{c})=0$.
\end{enumerate}
\end{lemma}

\begin{proof}
Let $e>1$ and assume that $\varpi\nmid \cc$. 
We may appeal to an argument of Hooley 
\cite{hooley2014octonary}*{\S~6}, which carries over to the function field setting to give
$$
S_{\varpi^e}(\bm{c})=\frac{|\varpi|}{|\varpi|-1}\left\{|\varpi|^e\nu_0(\varpi^e,\cc)-
|\varpi|^{e-1}\nu_1(\varpi^e,\cc)\right\}.
$$
Here, for $i\in \{0,1\}$,  $\nu_i(\varpi^e,\cc)$ is the number of $\bm{y}\bmod{\varpi^e}$ such that 
$F(\bm{y})\equiv 0\bmod{\varpi^e}$, with 
$\varpi\nmid \bm{y}$ and 
$\cc.\bm{y}\equiv 0\bmod{\varpi^{e-i}}$.
Since $F$ is a non-singular cubic form with coefficients in $\FF_q$, it follows that 
$\varpi\nmid \nabla F(\bm{y})$ for any $\bm{y}$ such that $\varpi\mid F(\bm{y})$ but $\varpi\nmid \bm{y}$. 
Hence
$\nu_1(\varpi^e,\cc)=|\varpi|^{n-1}\nu_0(\varpi^{e-1},\cc)$, so that
$$
S_{\varpi^e}(\bm{c})=\frac{|\varpi|^{e}}{1-|\varpi|^{-1}}\left\{\nu_0(\varpi^e,\cc)-
|\varpi|^{n-2}\nu_0(\varpi^{e-1},\cc)\right\}.
$$
Since $\varpi\nmid \cc$ we can assume without loss of generality that $\varpi\nmid c_n$. Eliminating $y_n$, we are left with the expression
$$
S_{\varpi^e}(\bm{c})=\frac{|\varpi|^{e}}{1-|\varpi|^{-1}}\left\{\nu(\varpi^e,\cc)-
|\varpi|^{n-2}\nu(\varpi^{e-1},\cc)\right\},
$$
where 
$\nu(\varpi^e,\cc)$ is the number of $\bm{y}'=(y_1,\dots,y_{n-1})\bmod{\varpi^e}$ such that 
$G(\bm{y}')\equiv 0\bmod{\varpi^e}$ and 
$\varpi\nmid \bm{y}'$, where
$
G(\bm{y}')=F(c_ny_1,\dots,c_ny_{n-1},-c_1y_1-\dots-c_{n-1}y_{n-1}).
$
Part (2) for $\varpi\nmid \bm{c}$ now follows from
taking $m=e-1$, 
$\delta\in [0, \tfrac{m-1}{2}]$ and $\ell=0$
in the multivariate form of the Hensel-type lemma
\cite{browning2013norm}*{Lemma 3.3}, 
carried over to $\OK$.
Indeed, for any $\bm{u}'$ counted by $\nu(\varpi^{e-1},\cc)$, we must have
$$v_\varpi(\nabla G(\bm{u}')) \le (e-2)/2 = (m-1)/2,$$
or else $v_\varpi(F^\ast(\bm{c})) \ge e-1$ by the $\beta = \ceil{\frac{e-1}{2}}$ case of Lemma \ref{LEM:tricky-grad-squared-divisibility}.
Using Lemma \ref{LEM:ad-hoc-reduce-to-primitive-c's} and arguing by  induction on $v_\varpi(\bm{c})$, we can  also conclude that part (2) holds if $\varpi\mid \bm{c}$.

Turning to part (1), we take $e=2$ and note that $\varpi\nmid \cc$ if $v_\varpi(F^*(\cc))\leq 1$, since $\deg F^*>1$. But then it follows that 
$$
S_{\varpi^2}(\bm{c})=\frac{|\varpi|^{2}}{1-|\varpi|^{-1}}\left\{\nu(\varpi^2,\cc)-
|\varpi|^{n-2} \nu(\varpi,\cc)\right\}.
$$
Pick $\bm{u}'$ counted by $\nu(\varpi,\cc)$.
If $\varpi\nmid \nabla G(\bm{u}')$ then 
the contribution to the right hand side vanishes, by Hensel lifting.
On the other hand, if $\varpi\mid \nabla G(\bm{u}')$ then
 Lemma \ref{LEM:tricky-grad-squared-divisibility} implies that the contribution to $\nu(\varpi^2,\cc)$ vanishes.
Hence 
$$
S_{\varpi^2}^\natural(\bm{c})\ll \frac{|\varpi|^{2}}{
|\varpi^2|^{(n+1)/2}
}\cdot |\varpi|^{n-2} \#\{\bm{u}'\bmod{\varpi}: \varpi\mid \nabla G(\bm{u}'), ~\varpi\nmid \bm{u}'\}.
$$
It follows from a result of Zak \cite{hooley1991number}*{Katz's Appendix, Theorem 2} that the hypersurface $G=0$ in $\mathbb{P}^{n-2}$
has only $O(1)$ singular points, whence the number of $\bm{u}'$ is $O(|\varpi|)$.
It follows that 
$$
S_{\varpi^2}^\natural(\bm{c})\ll \frac{|\varpi|^{2}}{
|\varpi^2|^{(n+1)/2}
}\cdot |\varpi|^{n-1}\ll 1,
$$
which is what we wanted.
\end{proof}

\begin{lemma}
\label{LEM:tricky-grad-squared-divisibility}
Suppose $\varpi\nmid \bm{u}'$ and $\varpi^\beta \mid \nabla G(\bm{u}')$.
Then $v_\varpi(F^*(\cc)) \ge \min(v_\varpi(G(\bm{u}')), 2\beta)$.
\end{lemma}

\begin{proof}
Since $\varpi\nmid \bm{u}'$, we have $\varpi\nmid u_a$, say.
Then by Proposition \ref{PROP:sliced-discriminant} and \cite{buse2014discriminant}*{Corollary 4.30; see the beginning of \S~4 for the definition of $\tilde{f}$ there},
$F^\ast(\bm{c})$ is an $\OK[\bm{u}'][1/u_a]$-linear combination of
$G$ and the pairwise products $(\partial G/\partial u_i)(\partial G/\partial u_j)$.
Since $\varpi^{2\beta}$ divides each pairwise product,
the desired divisibility of $F^\ast(\bm{c})$ follows.
\end{proof}

In addition, we have the following estimate for averages of square-full moduli, which is a consequence of (6.8) and Lemma 4.4 in \cite{glas2022question}.
\begin{proposition}\label{Prop: SquarefullAverageOld}
    Let $Y, C \geq 1$. Then 
    \[
    \sum_{\substack{|\bm{c}|=\widehat{C}\\ c_1\cdots c_n\neq 0\\ F^*(\bm{c})\neq 0}}
    \sum_{\substack{r\in \RcB\\ |r|=\widehat{Y}}} |S_r(\bm{c})|
    \ll_\ve \widehat{C}^{n+\varepsilon}\widehat{Y}^{1+n/2+\varepsilon}.
    \]
\end{proposition}

We end by recalling \cite{glas2022question}*{Lemma 4.2} (a special case of \cite{browning2015rational}*{Lemma 8.5}),
which concerns square-root cancellation of $S_r(\bm{c})$ over ``good'' moduli $r$.
\begin{proposition}\label{Prop: GRHCancellation}
    Let $Y\geq 1$ and $\bm{c}\in \mcal{S}_1$.
    Then 
    \[
    \sum_{\substack{r\in \RcG\\ |r|=\widehat{Y}}} S_r(\bm{c})
    \ll_\ve |\bm{c}|^\varepsilon \widehat{Y}^{1+n/2+\varepsilon}.
    \]
\end{proposition}

\section{Moments of bad exponential sums}

Our primary goal in this section is to prove the following result.

\begin{theorem}\label{thm:B3}
Let $0\le R\le 3Z$.
There exist $\eta,\delta>0$ such that
if $2\le A\le 2+\delta$, then
$$
\sum_{
\substack{
\bm{c}\in \mathcal{S}_1\\ \norm{\bm{c}}\le \hat Z\\ c_1\cdots c_n\ne 0}}
\biggl\lvert\,
\sum_{
\substack{
r\in \RcB\\\abs{r}=\hat R}} |r|^{-1/2} |S^\natural_r(\bm{c})|
\biggr\rvert^A \ll \hat Z^n \hat R^{-\eta}.
$$
\end{theorem}

(As we will see in \S~\ref{SEC:main-generic-endgame}, the complementary locus $\set{\bm{c}\in \mcal{S}_1: c_1\cdots c_n=0}$ has essentially already been satisfactorily treated in \cite{glas2022question}.)

Let us denote by $\Sigma_A=\Sigma_A(Z,R)$ the sum on the left hand in the statement of this result.
The trivial bounds $\abs{S_r(\bm{c})} \le \abs{r}^{1+n}$ and $\#\set{r\in \RcB: \abs{r}=\hat R} \le \hat R$ imply
\begin{equation}\label{eq:burger1}
\Sigma_A \leq \hat R^{(1+n/2)\delta} \Sigma_2 \quad 
\textnormal{ for } 2\le A\le 2+\delta.
\end{equation}
To analyse $\Sigma_2$,
it will be convenient to define
$$
\mathcal{S}_2=\left\{
\bm{c}\in \mathcal{S}_1: c_1\cdots c_n\ne 0\right\}
$$
and 
$$
T_r(\cc)= \prod_{\varpi^e\| r} \max\left\{1, \frac{|S_{\varpi^e}^\natural (\cc)|}{|\varpi^e|^{1/2}}\right\}.
$$
Following the approach in
\cite{wang2023ratios}*{\S~9}, 
we shall use different arguments according to the size of 
$T_r(\cc)$.
Let $B\geq 0$ be a parameter to be chosen in due course and 
note that 
\begin{equation}\label{eq:burger2}
\Sigma_2 \leq 2\left(\Sigma^{(1)}+\Sigma^{(2)}\right),
\end{equation}
where
\begin{align}
\Sigma^{(1)}
&=\sum_{
\substack{
\bm{c}\in \mathcal{S}_2\\ \norm{\bm{c}}\le \hat Z}}
\biggl\lvert\,
\sum_{
\substack{
r\in \RcB,~\abs{r}=\hat R\\ 
T_r(\cc)<\hat B
}} |r|^{-1/2} |S^\natural_r(\bm{c})|
\biggr\rvert^2, 
\label{eq:SIG1}
\\
\Sigma^{(2)} 
&=\sum_{
\substack{
\bm{c}\in \mathcal{S}_2\\ \norm{\bm{c}}\le \hat Z}}
\biggl\lvert\,
\sum_{
\substack{
r\in \RcB,~\abs{r}=\hat R\\ 
T_r(\cc)\geq \hat B
}} |r|^{-1/2} |S^\natural_r(\bm{c})|
\biggr\rvert^2 .\label{eq:SIG2}
\end{align}

\subsection{Ekedahl sieve}\label{sec:EKE}

A version of the geometric sieve for arbitrary global fields has been worked out in forthcoming work of 
Bhargava, Shankar and Wang \cite{bhargava2023coregular}*{Theorem 16}.
Their result sieves out prime moduli,
but we need at least a partial version of the result for square-free moduli.
This is folklore but we prove it for the reader's convenience.
A similar idea is used in \cite{bhargava2021galois}*{\S~5, Case III}.

\begin{theorem}
\label{THM:weird-eke}
Let $T,Y>0$.
Fix two polynomials $H_0,H_1\in \OK[c_1,\dots,c_n]$ that are relatively prime in $K[c_1,\dots,c_n]$.
Then 
$$
\#\left\{
\bm{c}\in \OK^n: 
\begin{array}{l}
\norm{\bm{c}}\leq \hat T, ~\text{$\exists$ square-free $r\in \Omon$ such that}\\  
\text{$r\mid \gcd(H_0(\bm{c}), H_1(\bm{c}))$ and $|r|=\hat Y$} 
\end{array}
\right\}
\ll_\ve \hat Y^\ve \left(\frac{\hat T^n}{\hat Y} +\hat T^{n-1}\right).
$$
\end{theorem}

\begin{proof}
\emph{Case 1: $Y\le T$.}
Applying the  Lang--Weil estimate for the quasi-projective variety $H_0=H_1=0$, the count on the left is seen to be 
$$
\ll \sum_{\substack{r\in \Omon\\ \abs{r}=\hat Y}} \mu_K^2(r) 
|r|^{n-2+\ve} \left(
\frac{\hat T}{\abs{r}}\right)^n
\ll_\ve \frac{\hat T^n}{\hat Y^{1-\ve}},
$$
as desired.

\emph{Case 2: $Y>T$.}
By Lemma \ref{LEM:OK-linear-change-to-near-monic},
we can make an $\OK$-linear change of variables in order to assume that $\deg_{c_1}(H_0) = \deg(H_0)$
and $\deg_{c_1}(H_1) = \deg(H_1)$.

Let $H_2\in \OK[c_2,\dots,c_n]$ be the resultant of $H_0$ and $H_1$ with respect to $c_1$. Then any square-free $r\in \Omon$ such that 
$r\mid \gcd(H_0(\bm{c}), H_1(\bm{c}))$ must also satisfy\footnote{In fact this implication is valid even without assuming $r$ square-free;
the key is that $H_2$ lies in the ideal $(H_0,H_1)$ of $\OK[\bm{c}]$.
However, square-free is important when using Lang--Weil.} 
$$
r\mid \gcd\left(H_0(\bm{c}), H_2(c_2,\dots,c_n)\right).
$$
If $c_2,\dots,c_n\in \OK$ are specified, 
with $|c_2|,\dots,|c_n|\ll \hat T$, 
then either $H_2(c_2,\dots,c_n) = 0$ or else there are only $O_\ve(\hat T^\ve)$ choices for $r\mid H_2(c_2,\dots,c_n)$, by the divisor function bound.
Note that if $r$ is also specified, then the number of choices for $c_1\in 
\OK$ with $|c_1|\ll \hat T$ and 
$r\mid H_0(\bm{c})$, is 
$$
\ll N(g;r) \left(1+\frac{\hat T}{|r|}\right),
$$
in the notation of \eqref{eq:Ngr},
where $g(u)=H_0(u,c_2,\dots,c_n)$. It follows from \eqref{INEQ:univariate-zero-density-mod-square-free}
that $N(g;r)\ll_\ve |r|^{\ve}$, uniformly over $c_2,\dots,c_n\in \OK$.

Putting everything together, and recalling that $Y>T$, 
we obtain the overall contribution
\begin{equation*}
\ll_\ve \#\set{\bm{c}\in \OK^n:\, \norm{\bm{c}}\le \hat T,\; H_2=0}
+ \hat T^{n-1}\hat Y^{\ve}.
\end{equation*}
Since $H_0$ and $H_1$ are coprime over $K$, it follows that $H_2$ is a non-zero polynomial, whence the first term 
is $O(\hat T^{n-1})$ by Lemma \ref{LEM:affine-dimension-growth}.
\end{proof}

\subsection{Treatment of \texorpdfstring{$\Sigma^{(1)}$}{Sigma(1)}}

The main task of this section is to prove the following result. 

\begin{proposition}\label{prop:B3-p1}
Let 
 $0\le R\le 3Z$ and let $B\geq 0$. 
Then 
$$
\Sigma^{(1)} \ll_\ve\hat B^2 \hat Z^{n+\ve} 
\left(
\frac{1}{\hat Z^{\eta_0}} +
\frac{1}{\hat R^{1/9\cdot 2^{n}}} \right),
$$
where
$\eta_0=\min\{1/p,1/3\cdot 2^{n+1}\}$ and  $p=\Char(\FF_q)$.
\end{proposition}

Recall the definition \eqref{eq:SIG1} of $\Sigma^{(1)}$.
Any $r\in \RcB$ admits a factorisation $r=r_1s_1s_2$, where $r_1,s_1,s_2$ 
are relatively coprime such that 
$r_1$
is square-free and  $s_1,s_2$ are both square-full, and satisfy
\begin{align*}
&\varpi^e\|s_1 \Rightarrow \max\{2,e-1\}>v_\varpi(F^*(\cc)),\\
&\varpi^e\|s_2 \Rightarrow \max\{2,e-1\}\leq v_\varpi(F^*(\cc)).
\end{align*}

It follows from parts (1)--(2)
 of  Lemma \ref{lem:VW9.1} that $S_{s_1}^\natural(\cc)\ll_\ve |s_1|^\ve$ for any $\ve>0$.
 On the other hand, from the definition of $s_2$ we have
 $$
 \prod_{\varpi^e\| s_2} \varpi^{\max\{2,e-1\}}\mid F^*(\cc).
 $$
 Moreover, it is clear that the left hand side is a square-full element of $\OK$ with absolute value
 $
 \prod_{\varpi^e\| s_2} |\varpi|^{\max\{2,e-1\}}\geq |s_2|^{2/3}.
 $
We note that    $\max\{|r_1|,|s_1|,|s_2|\}\geq |r|^{1/3}=\hat R^{1/3}.$
Let us temporarily write 
$$
U(\cc)=
\sum_{
\substack{
r=r_1s_1s_2\in \RcB,~\abs{r}=\hat R\\ 
T_r(\cc)<\hat B
}} |r|^{-1/2} |S^\natural_r(\bm{c})|,
$$
so that 
$$\Sigma^{(1)}=
\sum_{
\substack{
\bm{c}\in \mathcal{S}_2\\ \norm{\bm{c}}\le \hat Z}} |U(\cc)|^2.
$$
We will argue differently according to which of 
$|r_1|,|s_1|,|s_2|$ is the maximum in the factorisation $r=r_1s_1s_2$, exploiting the multiplicativity of the sum $S^\natural_r(\bm{c})$.

The contribution to $U(\cc)$ from the case 
$\max\{|r_1|,|s_1|,|s_2|\}=|r_1|$ is 
\begin{align*}
\leq \hat B \sum_{
\substack{
s_1,s_2\in \RcB \\ 
|s_1s_2|\leq \hat R}
} \sum_{\substack{r_1\in \RcB\\ 
\hat R^{1/3}\leq |r_1|\leq \hat R
}} 
|r_1|^{-1/2} |S^\natural_{r_1}(\bm{c})|,
\end{align*}
since 
$
|s_1s_2|^{-1/2} |S^\natural_{s_1s_2}(\bm{c})|\leq T_{s_1s_2}(\cc)\leq T_{r}(\cc)<\hat B,
$
by assumption.
It follows from Lemma~\ref{LEM:count-B-R_c-infty-divisors} that
the number of available $s_1,s_2$ is $O_\ve(\hat Z^\ve)$, on recalling that $F^*(\cc)\neq 0$ and $R\leq 3Z$. Hence the overall contribution to $\Sigma^{(1)}$ from this case is 
$$
\ll_\ve \hat B^2\hat Z^\ve \Sigma_0^{(1)},
$$
where
$$
\Sigma_0^{(1)}=
\sum_{
\substack{
\bm{c}\in \mathcal{S}_2\\ \norm{\bm{c}}\le \hat Z}} 
\bigg|
\sum_{
\substack{
r\in \RcB\\
\hat R^{1/3}\leq 
~\abs{r}\leq \hat R}} \mu^2(r) |r|^{-1/2} |S^\natural_r(\bm{c})|
\bigg|^2.
$$

Next, 
the contribution to $U(\cc)$ from the case 
$\max\{|r_1|,|s_1|,|s_2|\}=|s_1|$ is 
\begin{align*}
&\leq \hat B \sum_{
\substack{
r_1,s_2\in \RcB \\ 
|r_1s_2|\leq \hat R}
} \sum_{\substack{s_1\in \RcB\\ 
\hat R^{1/3}\leq |s_1|\leq \hat R
}} 
|s_1|^{-1/2} |S^\natural_{s_1}(\bm{c})|
\ll_\ve \frac{\hat B  \hat Z^\ve }{\hat R^{1/6}}.
\end{align*}
Finally, 
the contribution to $U(\cc)$ from the case 
$\max\{|r_1|,|s_1|,|s_2|\}=|s_2|$ is 
\begin{align*}
&\leq \hat B \sum_{
\substack{
r_1,s_1\in \RcB \\ 
|r_1s_1|\leq \hat R}
} \sum_{\substack{s_2\in \RcB\\ 
\hat R^{1/3}\leq |s_2|\leq \hat R\\
S^\natural_{s_2}(\bm{c})\neq 0
}} 1
\ll_\ve \hat B  \hat Z^\ve \mathbf{1}_\cc,
\end{align*}
where
$$
\mathbf{1}_\cc = 
\begin{cases}
1 & \text{ if there exists square-full $d\in \OK$ s.t.\ $d\mid F^*(\cc)$ and $|d|\geq \hat R^{2/9}$,}\\
0 &\text{ otherwise.}
\end{cases}
$$
Putting everything together, it now follows that 
$$
\Sigma^{(1)}\ll_\ve \hat B^2\hat Z^\ve \left(
\Sigma_0^{(1)}
+ \frac{\hat Z^n}{\hat R^{1/3}}
+\max_{M\geq \frac{2}{9}R} N(M,Z)\right),
$$
where
$$
N(M,Z)=\#\left\{
\bm{c}\in \OK^n: 
\begin{array}{l}
\norm{\bm{c}}\le \hat Z, ~c_1\dots c_n F^*(\cc)\neq 0\\
\text{$\exists$ square-full $d\in \OK$ s.t.\ $d\mid F^*(\cc)$ and $|d|= \hat M$}
\end{array}
\right\}
$$
The statement of Proposition \ref{prop:B3-p1} now follows from the conjunction of the following pair of results, both of which rely on the version of the geometric sieve for function fields from  \S~\ref{sec:EKE}.

\begin{lemma}\label{lem:S0^1}
We have 
$$
\Sigma_0^{(1)}\ll_\ve \frac{\hat Z^{n+\ve}}{\min\{\hat R^{1/6}, \hat Z\}}.
$$
\end{lemma}

\begin{lemma}\label{lem:NQR}
We have 
$$
N(M,Z)\ll_\ve \hat Z^{n+\ve}\left(
\frac{1}{\hat Z^{1/p}} +
\frac{1}{\min\{\hat Z,\hat M\}^{1/3\cdot 2^{n+1}}} \right),
$$
where $p$ is the characteristic of $\FF_q$.
\end{lemma}

\begin{proof}[Proof of Lemma \ref{lem:S0^1}]
It follows from multiplicativity and Lemma \ref{lem:prime_power}(1) that 
$$
S_r^\natural(\cc)\ll_\ve \hat Z^\ve |\gcd(r,\nabla F^*(\cc))|^{1/2}.
$$ 
We pick a parameter $0\leq Y\leq R$ and 
break the set $\mathcal{S}_2$ into two sets
\begin{align*}
\mathcal{S}_{2,1}
&=
\left\{\cc\in \mathcal{S}_2:
\text{$\exists$ square-free $d\in \OK$ s.t.\ $d\mid \nabla F^*(\cc)$ and $|d|> \hat Y$}
\right\},\\
\mathcal{S}_{2,2}
&=
\left\{\cc\in \mathcal{S}_2:
\text{any  square-free $d\in \OK$ s.t.\ $d\mid \nabla F^*(\cc)$ satisfies  $|d|\leq  \hat Y$}
\right\}.
\end{align*}

To begin with, if $\cc\in \mathcal{S}_{2,1}$ then we take 
$S_r^\natural(\cc)\ll_\ve \hat Z^\ve |r|^{1/2}$ and we note that there are $O_\ve(\hat Z^\ve)$ choices of $r$, by Lemma  \ref{LEM:count-B-R_c-infty-divisors}. On the other hand, if 
$\cc\in \mathcal{S}_{2,2}$ then we may take 
$S_r^\natural(\cc)\ll_\ve \hat Z^\ve \hat Y^{1/2}$ and we also have  $O_\ve(\hat Z^\ve)$ choices of $r$ overall.
It follows that 
$$
\Sigma_0^{(1)}\ll_\ve \hat Z^\ve\left(
\#\mathcal{S}_{2,1} +
\frac{\hat Y\hat Z^{n}}{\hat R^{1/3}}\right).
$$
Since $3\cdot 2^{n-2} F^\ast = \bm{c}\cdot \nabla{F^\ast}$, we can apply the geometric sieve, in the form of Theorem~\ref{THM:weird-eke} (with $H_0 = F^\ast$ and $H_1 = \partial{F^\ast}/\partial{c_1}$, noting that $F^\ast$ is absolutely irreducible and $\deg{H_1}<\deg{H_0}$), in order to deduce the bound 
$$
\#\mathcal{S}_{2,1} \ll_\ve \hat Z^\ve\left( \frac{\hat Z^n}{\hat Y}+\hat Z^{n-1}\right).
$$
The statement of the lemma follows on combining these and taking  $Y=\frac{1}{6}R$.
\end{proof}

\begin{proof}[Proof of Lemma \ref{lem:NQR}]
We shall reduce the problem to estimating the quantities
$$
N_1(M_1,Z)=\#\left\{
\bm{c}\in \OK^n: 
\begin{array}{l}
\norm{\bm{c}}\le \hat Z, ~c_1\dots c_n F^*(\cc)\neq 0\\
\text{$\exists$ square-full $d\in \OK$ s.t.\ $d\mid F^*(\cc)$ and $\hat M_1<|d|\leq \hat Z$}
\end{array}
\right\}
$$
and 
$$
N_2(M_2,Z)=\#\left\{
\bm{c}\in \OK^n: 
\begin{array}{l}
\norm{\bm{c}}\le \hat Z, ~c_1\dots c_n F^*(\cc)\neq 0\\
\text{$\exists$ prime  $\varpi\in \OK$ s.t.\ $\varpi^2\mid F^*(\cc)$ and $|\varpi|>\hat M_2$}
\end{array}
\right\},
$$
for suitable $M_1,M_2\geq 0$.

Let $\bm{c}\in \OK^n$ be counted by $N(M,R)$, so that 
$\norm{\bm{c}}\le \hat Z$ and $c_1\dots c_n F^*(\cc)\neq 0$, and there exists a 
square-full $d\in \OK$ such that  $d\mid F^*(\cc)$ and $|d|=\hat M$.
Any such $d$ 
 admits a factorisation 
$$
d=\varpi_1^{e_1}\dots\varpi_m^{e_m},
$$
where $e_1,\dots,e_m\geq 2$. 
Let $\eta\in (0,1]$ be a parameter, to be chosen in due course.
Suppose first that $|\varpi_i^{e_i}|\leq \hat Z^\eta$ for all $1\leq i\leq m$. Then 
we claim that 
$\cc$ is counted by $N_1(M-\eta Z ,Z)$. 
To see this, if 
$M\leq Z$ then in fact 
$\cc$ is counted by $N_1(M,Z)$.  If $M>Z$ then we consider the square-full integer 
$d_1=d/\varpi_1^{e_1}$, with absolute value $|d_1|=\hat M/|\varpi_1^{e_1}|\geq \hat M/\hat Z^\eta$. 
If $\hat M/|\varpi_1^{e_1}|\leq \hat Z$ then 
$\cc$ is counted by $N_1(M-\eta Z ,Z)$. On the other hand, if 
 $\hat M/|\varpi_1^{e_1}|> \hat Z$ then we repeat argument for 
 $d_2=d_1/\varpi_2^{e_2}$. The claim follows on continuing in this way, since the process clearly terminates with a square-full divisor in the desired interval. 
 
 We now suppose, by symmetry,  that 
$|\varpi_1^{e_1}|> \hat Z^\eta$.  
If $|\varpi_1|\leq \hat Z^{\eta/2}$ then 
we
let  $e_1'\leq e_1$  be the largest integer such that $|\varpi_1^{e_1'}|\leq \hat Z^\eta$. It then follows that  
$|\varpi_1^{e_1'+1}|> \hat Z^\eta$, whence
$$
|\varpi_1^{e_1'}|> \frac{\hat Z^\eta}{|\varpi|} >\hat Z^{\eta/2}.
$$
Moreover, we must have $e_1'\geq 2$ since 
 $|\varpi_1|\leq \hat Z^{\eta/2}$. Thus, since $\eta\leq 1$, we see that $\varpi_1^{e_1'}$ is a square-full 
 divisor of $F^*(\cc)$ whose norm lies in the 
 interval $(\hat Z^{\eta/2},\hat Z]$.
 In this case, therefore, we deduce that 
$\cc$ is counted by $N_1(\frac{\eta}{2}Z ,Z)$.
Finally, we must attend to the case that 
$|\varpi_1^{e_1}|> \hat Z^\eta$ and 
 $|\varpi_1|> \hat Z^{\eta/2}$. But then it is immediate that 
$\cc$ is counted by $N_2(\frac{\eta}{2}Z ,Z)$. 

In summary, we may conclude that 
\begin{equation}\label{eq:N=N1+N2}
N(M,Z)\leq N_1(M-\eta Z, Z)+
N_1(\tfrac{\eta}{2}Z,Z)+N_2(\tfrac{\eta}{2}Z ,Z),
\end{equation}
for any $\eta\in (0,1]$.
We claim that 
\begin{equation}\label{eq:claimN1}
N_1(M_1 ,Z)\ll_\ve
\frac{\hat Z^{n+\ve}}{\hat M_1^{1/2^{n+1}}}
\end{equation}
and 
\begin{equation}\label{eq:claimN2}
N_2(M_2 ,Z)\ll_\ve \hat Z^\ve\left(\frac{\hat Z^{n}}{\hat M_2}+\hat Z^{n-1/p}\right),
\end{equation}
for any $M_1,M_2>0$,
where $p$ is the characteristic of $\FF_q$.
Applying these in \eqref{eq:N=N1+N2}, 
we deduce that 
$$
N(M,Z)\ll_\ve \hat Z^{n+\ve}\left(
\frac{1}{\hat Z^{1/p}} +
\frac{1}{\hat Z^{\eta/2^{n+2}}} +
\frac{\hat Z^{\eta/2^{n+1}}}{\hat M^{1/2^{n+1}}} \right).
$$
The statement of Lemma \ref{lem:NQR} easily follows on taking 
$$
\eta=\frac{2}{3}\frac{\min\{M,Z\}}{Z}.
$$

\subsubsection*{Treatment of \texorpdfstring{$N_1(M_1,Z)$}{N1(M1,Z)}}

We will only make $N_1(M_1,Z)$ bigger by summing over all possible square-full $d$, 
and then breaking into residue classes modulo $d$. This gives
$$
N_1(M_1,Z)\leq \sum_{\substack{
\text{$d\in \OK$ square-full}\\ \hat M_1<|d|\leq \hat Z}}
N(F^*;d) \left(\frac{\hat Z}{|d|}\right)^n,
$$
in the notation of \eqref{eq:Ngr}.  
Let $\delta=2^{-n}.$
Since $N(F^*;d)$ is multiplicative in $d$, 
it follows from Rankin's trick that 
$$
N_1(M_1,Z)\leq \frac{\hat Z^n}{\hat M_1^{\delta/2-\ve}}
\prod_{\varpi}\left(1+\sum_{e\geq 2} 
\frac{N(F^\ast;\varpi^e)}{|\varpi^e|^{n-\delta/2+\ve}}\right)
$$
When $e=2$, we have $N(F^\ast;\varpi^2) = O(|\varpi|^{2n-2})$ by Lemma \ref{LEM:fixed-poly-cube-free-point-count-bound}.
It follows that
$$
N(F^\ast;\varpi^e)\leq |\varpi|^{(e-2)n} N(F^\ast;\varpi^2)=O(|\varpi|^{en-2}).
$$
Once combined with Corollary \ref{cor:HUA}, we conclude that 
$$
N_1(M_1,Z)\leq \frac{\hat Z^n}{\hat M_1^{\delta/2-\ve}}
\prod_{\varpi}\left(1+
O\left(
\sum_{e\geq 2} 
\frac{|\varpi|^{en-\max\{2,\delta e\}}}{|\varpi^e|^{n-\delta/2+\ve}}\right)\right),
$$
since $F^*$ 
is an absolutely irreducible form of degree $3\times 2^{n-2}\leq 2^n$.
But clearly
\begin{align*}
\sum_{e\geq 2} 
\frac{|\varpi|^{en-\max\{2,\delta e\}}}{|\varpi^e|^{n-\delta/2+\ve}}
&\leq 
\sum_{2\leq e\leq 2^{n+1}} 
\frac{|\varpi|^{e(1/2^{n+1}-\ve)}}{|\varpi|^{2}}+
\sum_{e\geq  2^{n+1}} 
\frac{1}{|\varpi|^{e(1/2^{n+1}+\ve)}}
\ll \frac{1}{|\varpi|^{1+2^{n+1}\ve}},
\end{align*}
since $\delta=2^{-n}$. Thus the product over $\varpi$ converges absolutely and the bound \eqref{eq:claimN1} follows on noting that $\hat M_1^\ve\leq \hat Z^\ve$.

\subsubsection*{Treatment of \texorpdfstring{$N_2(M_2,Z)$}{N2(M2,Z)}}

We follow the argument in Poonen \cite{poonen2003squarefree}*{\S~7}, noting that $F^*$ is a square-free polynomial defined over $\FF_q$.
Poonen defines the polynomial
$G\in \OK[(y_{i,j})_{1\leq i\leq n, 0\leq j\leq p-1}]$ via
$$
G=F^*\left(
\sum_{0\leq j\leq p-1}t^j y_{1,j}^p, \dots, 
\sum_{0\leq j\leq p-1}t^j y_{n,j}^p
\right),
$$
where $p$ is the characteristic of $\FF_q$. It follows from $n$ applications of 
\cite{poonen2003squarefree}*{Lemma 7.2} that $G$ is square-free as an element of $K[y_{i,j}]$. 
Let $Z_j=(Z-j)/p$. Then, as $y_{i,j}$ runs over elements of $\OK$ with 
$|z_{i,j}|\leq \hat Z_j$, then so the $n$-tuples
$$
\left(
\sum_{0\leq j\leq p-1}t^j y_{1,j}^p, \dots, 
\sum_{0\leq j\leq p-1}t^j y_{n,j}^p
\right)
$$
run over elements  $\mathbf{z}\in \OK^n$ with $\norm{\mathbf{z}}\leq \hat Z$, with each element appearing exactly once.   
 Hence it follows that 
$$
N_2(M_2 ,Z)
=\#\left\{
\bm{y}\in \OK^{np}: 
\begin{array}{l}
|y_{i,j}|\le \hat Z_j \text{ for $1\leq i\leq n$ and $0\leq j\leq p-1$}\\
\text{$\exists$   $\varpi\in \OK$ s.t.\ $\varpi^2\mid G(\bm{y})$ and $|\varpi|>\hat M_2$}
\end{array}
\right\}.
$$
Poonen's crucial observation is that we have 
$\varpi\mid G(\bm{y})$ and $\varpi\mid G'(\bm{y})$
whenever  $\varpi^2\mid G(\bm{y})$, where $G'=\frac{\partial G}{\partial t}$.
Since $G$ and $G'$ are coprime as elements of $K[y_{i,j}]$, 
by \cite{poonen2003squarefree}*{Lemma 7.3}, 
we can therefore apply 
Theorem \ref{THM:weird-eke} to deduce that 
\begin{align*}
N_2(M_2 ,Z)
&\ll_\ve \hat Z_0^\ve\left(\frac{\hat Z_0^{pn}}{\hat M_2}+\hat Z_0^{pn-1}
\right)
\ll_\ve \hat Z^\ve\left(\frac{\hat Z^{n}}{\hat M_2}+\hat Z^{n-1/p}
\right),
\end{align*}
since $Z_j\leq Z_0\leq Z/p$ for all $0\leq j\leq p-1$.
This therefore establishes \eqref{eq:claimN2}.
\end{proof}

\subsection{Treatment of \texorpdfstring{$\Sigma^{(2)}$}{Sigma(2)}}

We now turn to the estimation of 
$\Sigma^{(2)}$, as defined in \eqref{eq:SIG2}, with the main goal being to prove the following bound.

\begin{proposition}\label{prop:B9}
Let 
 $0\le R\le 3Z$ and let $B\geq 0$. 
Then 
$$
\Sigma^{(2)} \ll_\ve  \hat Z^{n+\ve} 
\left(
\frac{1}{\hat Z^{1/2}} +
\frac{1}{\hat B^{2/n}} \right).
$$
\end{proposition}

\begin{proof}
Let $r$ and $\cc$ be counted by $\Sigma^{(2)}$. 
Then it follows from 
Lemma \ref{lem:lem8.2} that
\begin{align*}
\hat B\leq T_r(\cc)
\le \max_{s\mid r}
\frac{|S_s^\natural(\cc)|}{|s|^{1/2}} 
&\ll_\ve \max_{s\mid r}
|s|^\ve  \prod_{1\leq i\leq n}
\left|\gcd\left(\cub(s)
, \sq(c_i)\right)\right|^{1/4}\\
&\ll_\ve \hat R^\ve  \max_{1\leq i\leq n}
\left|\gcd\left(\cub(r)
, \sq(c_i)\right)\right|^{n/4},
\end{align*}
since $\cub(s)\mid \cub(r)$.
On relabelling $c_1,\dots,c_n$, we may assume that the maximum occurs at $i=n$ on the right hand side. According to 
part (2) of Lemma \ref{lem:VW9.1}, moreover, 
we will have 
$S^\natural_r(\bm{c})=0$ unless $e\leq 1+v_\varpi(F^*(\cc))$ 
for any prime power $\varpi^e\|r$.
If we let $d=\gcd(\cub(r)
, \sq(c_n))$, then $d$ is a square-full element of $\OK$ with 
$$
|d|\gg_\ve \frac{\hat B^{4/n}}{\hat R^\ve}.
$$
Moreover, $|d|\leq \hat Z$ since $d\mid c_n$ and $c_n\neq 0$. Finally, we note that 
$d\mid F^*(\cc)^2$ since $d\mid \cub(r)$ and $e\leq 2(e-1)$ for $e\geq 2$.

Let $D\in \RR$ be such that 
$\hat D=\hat B^{4/n}/\hat R^\ve$. 
Then  we have proved that 
$$
\Sigma^{(2)}
\ll_\ve \hat Z^\ve \sum_{\substack{d\in \OK \text{ square-full}\\
\hat D \ll_\ve |d|\leq \hat Z}
} |d|^{1/2} 
\sum_{
\substack{
\bm{c}\in \mathcal{S}_2\\ \norm{\bm{c}}\le \hat Z\\
d\mid \gcd(c_n, F^*(\cc)^2)
}} 
\left|\sq(c_1)\cdots \sq(c_{n-1})\right|^{1/2},
$$
since 
Lemma \ref{LEM:count-B-R_c-infty-divisors} implies that 
$\#\set{r\in \RcB: \abs{r}=\hat R}\ll_\ve \hat Z^\ve$.
In dealing with this sum we shall make frequent use of the bound
$$
\sum_{\substack{c\in \OK\\
|c|\leq \hat Z
}} |\sq(c)|^{1/2}\ll  \hat Z \log \hat Z,
$$
which readily follows on making the  factorisation $c=c_0\sq(c)$, where $c_0$ is the square-free part of $c$, 
 and noting that there are 
$O(\hat C^{1/2})$ square-full elements of $\OK$ with absolute value $\hat C$.

We proceed by sorting the vectors $\cc$ into two contributions.  Let 
$\cc'=(c_1,\dots,c_{n-1})$ and let 
$G(\cc')=
F^*(c_1,\dots,c_{n-1},0)^2$.
We first deal with the 
contribution from those $\cc$ for which $G(\cc')\neq 0$. 
In this case we invert the order of summation and observe that there are $O_\ve(\hat Z^\ve)$ choices of $d$ for which $d\mid G(\cc')$, by the standard estimate for the divisor function. 
Hence the contribution to $\Sigma^{(2)}$ from this case is 
\begin{align*}
&\ll_\ve 
\hat Z^\ve 
\hspace{-0.3cm}
\sum_{
\substack{|c_1|,\dots,|c_{n-1}|\le \hat Z\\
G(\cc')\neq 0}}
\hspace{-0.3cm}
\left|\sq(c_1)\cdots \sq(c_{n-1})\right|^{1/2}
\hspace{-0.2cm}
\sum_{
\substack{d\mid G(\cc')\\
\hat D\ll_\ve |d|\leq \hat Z}} |d|^{1/2}
\#\left\{c_n\in \OK: |c_n|\leq \hat Z, d\mid c_n\right\}\\
&\ll_\ve 
\frac{\hat Z^{1+2\ve}}{\hat D^{1/2}}
\sum_{
\substack{|c_1|,\dots,|c_{n-1}|\le \hat Z}}
\left|\sq(c_1)\cdots \sq(c_{n-1})\right|^{1/2}\\
&\ll_\ve 
 \frac{\hat Z^{n+3\ve}}{\hat D^{1/2}} .
\end{align*}
This is satisfactory for the proposition on observing that 
$\hat D\gg \hat B^{4/n}/\hat Z^\ve$. 

In order to deal with the remaining contribution from 
$\cc$ for which $G(\cc')=0$, we note that for fixed $c_1,\dots,c_{n-2}$, there are only finitely many choices of $c_{n-1}$ for which 
$G(\cc')=0$. This is due to the fact the coefficient of $c_{n-1}^{\deg G}$ appears with a non-zero coefficient, as observed by Heath-Brown \cite{heath1998circle}*{Eq.~(4.2)}.
Hence we obtain the overall contribution
\begin{align*}
&\ll_\ve 
\hat Z^\ve 
\sum_{\substack{d\in \OK \text{ square-full}\\
\hat D \ll_\ve |d|\leq \hat Z}
}
\sum_{
\substack{|c_{n}|\le \hat Z\\ d\mid c_n}}
|d|^{1/2}
\hat Z^{1/2}
\prod_{1\leq i\leq n-2}
\sum_{|c_i|\le \hat Z}
\left|\sq(c_i)\right|^{1/2}\\
&\ll_\ve 
\hat Z^{n-1/2+2\ve }
\sum_{\substack{d\in \OK \text{ square-full}\\
|d|\leq \hat Z}
}
\frac{1}{
|d|^{1/2}}\\
&\ll_\ve 
\hat Z^{n-1/2+3\ve },
\end{align*}
which thereby completes the proof.
\end{proof}

\subsection{Final touches}

We combine Propositions \ref{prop:B3-p1} and  \ref{prop:B9} 
in \eqref{eq:burger2} to deduce that 
\begin{align*}
\Sigma^{(2)} 
&\ll_\ve  \hat Z^{n+\ve} 
\left(
\frac{\hat B^2}{\hat Z^{\eta_0}} +
\frac{\hat B^2}{\hat R^{1/9\cdot 2^{n}}} 
+\frac{1}{\hat Z^{1/2}} +
\frac{1}{\hat B^{2/n}} \right)\\
&\ll_\ve  \hat Z^{n+\ve} 
\left(
\frac{\hat B^2}{\hat R^{\min\{\eta_0/3,1/9\cdot 2^{n}\}}}
+\frac{1}{\hat R^{1/6}} +
\frac{1}{\hat B^{2/n}} \right),
\end{align*}
since $Z\geq R/3$.  The middle term is dominated by the first term. Moreover, we 
 clearly have $\min\{\eta_0/3,1/9\cdot 2^{n}\}=\eta_0/3$.
 Balancing for  $B$, 
we deduce that 
\begin{equation}\label{eq:burger3}
\Sigma_2\ll_\ve \hat Z^{n+\ve}\hat R^{-\eta_1},
\end{equation}
where
$
\eta_1=
\eta_0/(3n+3).
$

We may now deduce the statement of Theorem \ref{thm:B3}, recalling the notation 
$\Sigma_A=\Sigma_A(Z,R)$ for any $2\leq A\leq 2+\delta$ and any 
 $0\le R\le 3Z$.
 We claim that 
$$
\Sigma_2(Z,R)\ll \hat Z^{n}\hat R^{-\eta_1/2},
$$
which will clearly suffice for the theorem, in the light of \eqref{eq:burger1}.
 If $Z\leq 2R$ then the statement follows from \eqref{eq:burger3}, on choosing 
 $\ve\leq \eta_1/4$.   Suppose next that 
 $Z>2R$ and open up the square to get
 $$
 \Sigma_2(Z,R)\ll \sum_{\substack{\abs{r_1}=\hat R\\ \abs{r_2}=\hat R}}
 \sum_{\substack{\bm{d} \bmod{r_1r_2}\\ \varpi\mid r_1r_2\Rightarrow \varpi\mid F^*(\bm{d})}}
|r_1r_2|^{-1/2}|S^\natural_{r_1}(\bm{d})S^\natural_{r_2}(\bm{d})|
 \#\{\cc : \cc\equiv\bm{d} \bmod{r_1r_2}\},
  $$
on noting that $S^\natural_{r_1}(\bm{c})S^\natural_{r_2}(\bm{c})$ only depends on the value of $\cc$ modulo $r_1r_2$. 
The inner cardinality is $O(\hat Z^n/\hat R^{2n})$, whence
 \begin{align*}
 \Sigma_2(Z,R)
 &\ll \frac{\hat Z^n}{\hat R^{2n}}\sum_{\substack{\abs{r_1}=\hat R\\ \abs{r_2}=\hat R}}
 \sum_{\substack{\bm{d} \bmod{r_1r_2}\\ \varpi\mid r_1r_2\Rightarrow \varpi\mid F^*(\bm{d})}}
|r_1r_2|^{-1/2}|S^\natural_{r_1}(\bm{d})S^\natural_{r_2}(\bm{d})|\\
 &\ll \frac{\hat Z^n}{\hat R^{2n}} \Sigma_2(n + 2R,R)
  \end{align*}
  by Lemma \ref{LEM:find-small-lift-in-Zariski-open}.
 The claim now follows on appealing to \eqref{eq:burger3}.

\begin{lemma}
\label{LEM:find-small-lift-in-Zariski-open}
Suppose $\norm{\bm{d}}\le r_1r_2$.
Then there exists $\bm{c}\in \mcal{S}_2$ with $\norm{\bm{c}}\le q^n\, \abs{r_1r_2}$ and $\bm{c}\equiv \bm{d}\bmod{r_1r_2}$.
\end{lemma}

\begin{proof}
Let $H(\bm{c}) = c_1\cdots c_n F^\ast(\bm{c})$.
Let $S = \set{\bm{d} + r_1r_2\bm{k}: \norm{\bm{k}}\le q^n}$.
We want to find $\bm{c}\in S$ with $H(\bm{c})\ne 0$.
The set $S$ is a Cartesian product $\prod_{1\le j\le n} I_j$,
for some sets $I_1,\dots,I_n\belongs \OK$ of size $\abs{I_j} = q^{n+1}$.
Since $\deg{H} = n + 3\cdot 2^{n-2} < 2^n + 2^n \le \abs{I_j}$ for all $j$,
the desired $\bm{c}$ exists by the combinatorial nullstellensatz \cite{alon1999combinatorial}*{Theorem 1.2} in the field $K$.
\end{proof}

\section{Endgame outside dual variety}
\label{SEC:main-generic-endgame}

Let $w$ be as in \S~\ref{SEC:integral-stuff} 
with the properties specified in Hypothesis \ref{hyp}, but with parameter $L=0$.
Let 
\begin{equation}\label{eq:assam-tea}
E_1(P)=|P|^n\sum_{\substack{\bm{c}\in\OK^n\\F^*(\bm{c})\neq 0}}\sum_{\substack{r\in\Omon\\ |r|\leq \widehat{Q}}}|r|^{-n}S_r(\bm{c})I_r(\bm{c}),
\end{equation}
where we recall that $Q$ is chosen in such a way that $|P|^{3/2}\asymp \widehat{Q}$. By Lemma \ref{Le: IntVanish} we have $I_r(\bm{c})=0$ unless $|\bm{c}|\ll |r|\max\{1, |P|^3|r|^{-1}\widehat{Q}^{-1}\}|P|^{-1}\ll |P|^{1/2}$. In particular, the sum over $\bm{c}$ in the definition of $E_1(P)$ only runs over vectors $\bm{c}\in\OK^n$ with $|\bm{c}|\ll |P|^{1/2}$. Moreover, a simple change of variables shows that Lemma \ref{Le: IntEstimateI} implies 
\begin{equation}\label{Eq: IntEstimateEndgame}
    I_r(\bm{c})\ll |P|^{-3}\left(1+\frac{|P||\bm{c}|}{|r|}\right)^{1-n/2}.
\end{equation}
We also keep in mind the $L=0$ case of Lemma \ref{LEM:integral-scale-invariance}, which implies the following:
\begin{equation*}
    \textnormal{If $r_1,r_2\in \Omon$ with $\abs{r_1}=\abs{r_2}$, then $I_{r_1}(\bm{c}) = I_{r_2}(\bm{c})$.}
\end{equation*}
The primary  goal of this section is to establish the following result.

\begin{proposition}\label{Prop:nonDualContribution}
Assume Conjecture \ref{CNJ:(R2'E')} holds. 
If $F=x_1^3+\dots+x_6^3$, then
$
E_1(P)\ll |P|^3.
$
\end{proposition}

We begin by dealing with the contribution from $c_1\cdots c_n=0$, which already follows from \cite{glas2022question}. Let $\mathcal{I}\subset\{1,\dots, n\}$ with $\#\mathcal{I}=t$ and define 
\[
\mathcal{R}(C)\coloneqq \{\bm{c}\in\OK^n\colon |\bm{c}|=\widehat{C},
\text{ such that }
c_i=0 \text{ if }i\not\in \mathcal{I}\text{ and }c_i\neq 0 \text{ else}\}.
\]
Let us also fix $0\leq Y\leq Q$. We will now consider the contribution to $E_1(P)$ coming from vectors $\bm{c}\in \mathcal{R}(C)$ and $r\in \Omon$ with $|r|=\widehat{Y}$. Since the contribution from $\widehat{C}\gg |P|^{1/2}$ vanishes, it is clear that there are at most $O(\log^2 |P|)$ permissible choices for pairs $(C,Y)$. It follows from the last display on page 21 of \cite{glas2022question} that 
\[
\frac{|P|^n}{\widehat{Y}^n}\sum_{\substack{\bm{c}\in \mathcal{R}(C) \\ F^*(\bm{c})\neq 0}}\sum_{|r|=\widehat{Y}}S_r(\bm{c})I_r(\bm{c})\ll \frac{|P|^{n-3/2+\varepsilon}}{\widehat{Y}^{n/2}}\left(\frac{\widehat{Y}}{|P|\widehat{C}}\right)^{n/4}\widehat{Y}^{(n-t)/6}\widehat{C}^t\min\left\{1, \frac{\widehat{Y}}{|P|\widehat{C}}\right\}^{t/4}.
\]
Note that they work with a different weight function;
nevertheless, the estimate continues to hold in our setting as we have the stronger integral estimate \eqref{Eq: IntEstimateEndgame} at our disposal. Since we are assuming that  $F^*(\bm{c})\neq 0$ and $c_1\cdots c_n=0$, we must have $1\leq t \leq n-1$. One now easily sees that the expression above is maximal at $t=1$ or $t=n-1$. The contribution from $t=n-1$ is 
\[
\ll |P|^{n/2-5/4+\varepsilon}\widehat{Y}^{-1/12}\widehat{C}^{n/2-3/4}\ll |P|^{3n/4-13/8+\varepsilon},
\]
while $t=1$ gives 
\[
|P|^{3n/4-7/4+\varepsilon}\widehat{Y}^{(1-n)/12}\widehat{C}^{(3-n)/4}\ll |P|^{3n/4-7/4+\varepsilon},
\]
where we used that $n\geq 3$ and $\widehat{C}\ll |P|^{1/2}$. We have $3n/4-7/4<n-3$ and $3n/4-13/8<n-3$ for $n\geq 6$ and since there at most $O(\log^2 |P|)$ choices for $(Y,C)$, this shows that the contribution from $c_1\cdots c_n=0$ to $E_1(P)$ is $O(|P|^{n-3-\delta})$ for some $\delta>0$. 
This is satisfactory for Proposition \ref{Prop:nonDualContribution}.

Next, we factor $r=r_1r_2$, where $r_1\in \RcG$ and $r_2\in \RcB$. For $1\leq \widehat{C}\ll |P|^{1/2}$ and $0\leq Y_1+Y_2\leq Q$, we shall write $Y=Y_1+Y_2$ and consider the sum
\[
E_1(Y_1,Y_2, C)\coloneqq \sum_{\substack{|\bm{c}|=\widehat{C}}}\,
\left\lvert\sum_{\substack{r_1\in \RcG\\|r_1|=\widehat{Y}_1}}S^\natural_{r_1}(\bm{c})\right\rvert\,
\left\lvert\sum_{\substack{r_2\in \RcB\\ \abs{r_2}=\widehat{Y}_2}}S^\natural_{r_2}(\bm{c})\right\rvert, 
\]
where the sum over $\bm{c}$ runs over $\bm{c}\in \OK^n$ with $c_1\cdots c_n\neq 0$ and $F^*(\bm{c})\neq 0$. 
We will now provide an upper bound for $E_1(Y_1,Y_2,C)$ depending on the size of $C$ and the relative size of $Y_2$ compared to $\widehat{Q}/\widehat{Y}$. For the rest of this section, let $\delta>0$ be a fixed small real number.

\subsection{The case \texorpdfstring{$\widehat{C}<|P|^{1/2-\delta}$}{hatC<|P|1/2-delta}}

In this case it suffices to use the old estimates from \cite{glas2022question} and \cite{browning2015rational}. Applying Propositions \ref{Prop: GRHCancellation} and \ref{Prop: SquarefullAverageOld} in succession gives 
\[
E_1(Y_1,Y_2,C) \ll \widehat{Y}^{1/2+\varepsilon} \widehat{C}^{n+\varepsilon}.
\]
Combining this with \eqref{Eq: IntEstimateEndgame}, shows that the contribution to $E_1(P)$ with $\widehat{C}\leq |P|^{1/2-\delta}$ is at most 
\begin{align*}
    \sum_{Y_1,Y_2, C} |P|^{n-3} \widehat{Y}^{1-n/2+\varepsilon}\widehat{C}^{n+\varepsilon}\left(1+\frac{|P|\widehat{C}}{\widehat{Y}}\right)^{1-n/2}&\ll\sum_{Y_1,Y_2,C} \abs{P}^{n/2-2}\widehat{C}^{1+n/2+\varepsilon}\widehat{Y}^\varepsilon\\
    &\ll \abs{P}^{3n/4-3/2+\varepsilon - \delta(1+n/2+\varepsilon)/2},
\end{align*}
where we used that $\widehat{C}\ll |P|^{1/2-\delta}$, $\widehat{Y}\ll |P|^{3/2}$ and that the number of available $(Y_1,Y_2,C)$ is $O(\log |P|)=O(|P|^{\varepsilon})$. Thus the contribution to $E_1(P)$ is at most $O(|P|^{n-3})$ for $n\geq 6$, which is satisfactory. 

\subsection{The case \texorpdfstring{$\widehat{C}\geq |P|^{1/2-\delta}$}{hatC>=|P|1/2-delta}}

We shall divide the complentary case into two further subcases. First, let from now on $\varepsilon>0$ be a fixed, sufficiently small real number. We then define $\alpha = (1/2-\varepsilon)/\deg F^*$ and 
\begin{equation}\label{Eq: DefiMinAlpha}
    \widehat{W}=\min\left\{\frac{|P|^{3/2}}{\widehat{Y}},|P|^{(1-\varepsilon)/2}\right\}.
\end{equation}

\subsubsection*{The case \texorpdfstring{$\widehat{Y}_2^{n+\varepsilon}\leq \widehat{W}^{\alpha/2}$}{Y2 small}}

In this regime we require the full strength of Conjecture~\ref{CNJ:(R2'E')} and of Lemma \ref{LEM:new-vanishing-for-small-dual-form}. To make use of Lemma \ref{LEM:new-vanishing-for-small-dual-form}, we need the following auxiliary result. 
\begin{lemma}\label{Le: PolyInequSaving}
Let $G\in \OK[c_1,\dots, c_n]$ be a polynomial.
Let $C\geq 0$ and $\lambda\in \RR$.
Then 
\[
\#\{\bm{c}\in \OK^n \colon |\bm{c}|=\widehat{C}, |G(\bm{c})|\leq \widehat{C}^{\deg G}\hat \lambda \}\ll \widehat{C}^n(\hat \lambda +\widehat{C}^{-1})^{1/\deg G}.  
\]
\end{lemma}
\begin{proof}
By Lemma \ref{LEM:OK-linear-change-to-near-monic}, we may assume $\deg_{c_1}(G) = \deg{G}$.
The result then follows from \eqref{INEQ:integral-points-level-estimate} (with $v=\infty$), after fixing $c_2,\dots,c_n$.
\end{proof}
Recall from Lemma \ref{LEM:new-vanishing-for-small-dual-form} that $I_r(\bm{c})=0$ unless $|F^*(P\bm{c}/r)|\ll 1 + (|P||\bm{c}|/\widehat{Y})^{\deg F*-1}$. After setting 
\[
\mathcal{S}(C)=\{\bm{c}\in \OK^n\colon |\bm{c}|=\widehat{C}, 0<|F^*(P\bm{c}/r)|\ll 1 + (|P||\bm{c}|/\widehat{Y})^{\deg F*-1}, c_1\cdots c_n\neq 0\},
\]
it then follows from H\"older's inequality with $p_1=2-\eps_0$, $p_2=(2-\eps_0)/(1-2\eps_0)$ and $p_3=(2-\eps_0)/\eps_0$ that
\[
E_1(Y_1,Y_2, C)\leq \left(\sum_{\substack{|\bm{c}|=\widehat{C}} }\left|\sum_{|r_1|=\widehat{Y}_1}S_{r_1}^\natural(\bm{c})\right|^{p_1}\right)^{1/p_1}
\hspace{-0.4cm}
\left(\#\mathcal{S}(C)\right)^{1/p_2}\left(\sum_{|\bm{c}|=\widehat{C}}\left|\sum_{|r_2|=\widehat{Y}_2}S_{r_2}^\natural(\bm{c})\right|^{p_3} \right)^{1/p_3}.
\]
We can now apply the trivial estimate $|S_{r_2}^\natural(\bm{c})|\leq |r_2|^{(n+1)/2}$ 
and invoke Lemma \ref{LEM:N_c-small-divisor-moment-bound} to estimate the sum over $r_2$.
Combining this with Conjecture \ref{CNJ:(R2'E')} with $Z=\ceil{\max(C,\frac13Y_1)}$ and $R=Y_1$ to estimate the sum over $r_1$, we get 
\begin{align*}
E_1(Y_1,Y_2,C)&\ll
\widehat{Y}_1^{1/2}\left(\widehat{C}^n+\widehat{Y}_1^{n/3}\right)^{1/(2-\eps_0)}
\#\mathcal{S}(C)^{(1-2\eps_0)/(2-\eps_0)}
\hat C^{n\eps_0/(2-\eps_0)}\widehat{Y}_2^{(n+1)/2+\varepsilon} \\
&\le \widehat{Y}^{1/2}\widehat{Y}_2^{n/2+\varepsilon}
\left(\widehat{C}^n+\widehat{Y}_1^{n/3}\right)^{1/2+\varepsilon}
\#\mathcal{S}(C)^{1/2-\varepsilon},
\end{align*}
where $\varepsilon \defeq \frac12 - (1-2\eps_0)/(2-\eps_0) = \frac34\eps_0 + O(\eps_0^2)$.
Note that Lemma \ref{Le: PolyInequSaving} implies that 
\[
\frac{\#\mathcal{S}(C)}{\hat C^n}
\ll \left(\min\left\{1, \frac{\widehat{Y}}{|P|\widehat{C}}\right\}+\widehat{C}^{-1}\right)^{1/\deg F^*}
\ll \left(\frac{\widehat{Y}}{|P|\widehat{C}}+\widehat{C}^{-1}\right)^{1/\deg F^*}.
\]
Together with the last display and \eqref{Eq: IntEstimateEndgame} this shows that the contribution to $E_1(P)$ from those $(Y_1,Y_2, C)$ under consideration is 
\begin{align*}
    \ll~& \sum_{Y_1,Y_2, C}|P|^{n-3}\widehat{Y}^{1-n/2}\widehat{Y}_2^{n/2+\varepsilon}\widehat{C}^{n(1/2-\varepsilon)}\left( \frac{\widehat{Y}}{|P|\widehat{C}}+\widehat{C}^{-1}\right)^\alpha \left(\widehat{C}^n+\widehat{Y}_1^{n/3}\right)^{1/2+\varepsilon}
    \\
&\qquad \times    \left(1+\frac{|P|\widehat{C}}{\widehat{Y}}\right)^{1-n/2}\\
    \ll~& \sum_{Y_1,Y_2, C}|P|^{n/2-2}\widehat{Y}_2^{n/2+\varepsilon}\left(\left(\frac{\widehat{Y}}{|P|\widehat{C}}\right)^\alpha +\widehat{C}^{-\alpha}\right)\left(\widehat{C}^n+\widehat{Y}_1^{n/3}\right)^{1/2+\varepsilon}\widehat{C}^{1-n\varepsilon/2}.
\end{align*}
One can check that $-\alpha +n(1/2+\varepsilon)+1-n\varepsilon/2>0$ and $-\alpha +1-n\varepsilon /2>0$, so that upon recalling $\widehat{C}\ll |P|^{1/2}$, we get that the contribution is 
\begin{align*}
    &\ll \sum_{Y_1,Y_2}|P|^{n/2-3/2-n\varepsilon/2}\widehat{Y}_2^{n/2+\varepsilon}\left(\left(\frac{\widehat{Y}}{|P|^{3/2}}\right)^\alpha+|P|^{-\alpha/2}\right)\left(|P|^{n/2}+\widehat{Y}_1^{n/3}\right)^{1/2+\varepsilon}\\
    &\ll \sum_{Y_1,Y_2}|P|^{3n/4-3/2}\widehat{Y}_2^{n/2+\varepsilon}\left(\left(\frac{\widehat{Y}}{|P|^{3/2}}\right)^\alpha +|P|^{-\alpha/2}\right).
\end{align*}
We can now use that 
\[
\widehat{Y}_2^{n/2+\varepsilon}\leq \widehat{Y}_2^{-n/2}\widehat{W}^{\alpha/2},
\]
which together with the definition of $\widehat{W}$ in \eqref{Eq: DefiMinAlpha} implies that the contribution is 
\begin{align*}
    &\ll \sum_{Y_1,Y_2}|P|^{3n/4-3/2}\widehat{Y}_2^{-n/2}\widehat{W}^{\alpha/2}\left(\left(\frac{\widehat{Y}}{|P|^{3/2}}\right)^\alpha +|P|^{-\alpha/2}\right)\\
    &\ll |P|^{3n/4-3/2}\left(|P|^{-3\alpha/4}\sum_{Y_1,Y_2}\widehat{Y}_2^{-n/2}\widehat{Y}^{\alpha/2}
    +\sum_{Y_1,Y_2}\widehat{Y}_2^{-n/2}|P|^{-\alpha(1+\varepsilon)/4}\right)\\
    &\ll |P|^{3n/4-3/2}\left(1+|P|^{\varepsilon -\alpha(1+\varepsilon)/4}\right),
\end{align*}
where we used that $\sum_{Y_1,Y_2}\widehat{Y}_2^{-n/2}\widehat{Y}^{\alpha/2}\ll |P|^{3\alpha/4}$, because $\widehat{Y}\ll |P|^{3/2}$. Therefore, as $\varepsilon(1-\alpha/4)-\alpha<0$ for $\varepsilon$ sufficiently small, the contribution from the case under consideration is again $O(|P|^{n-3})$ for $n\geq 6$.

\subsubsection*{The case \texorpdfstring{$\widehat{Y}_2^{n+\varepsilon}>\widehat{W}^{\alpha/2}$}{Y2 large}}

Applying H\"older's inequality to $E_1(Y_1,Y_2,C)$ with $p_1=2-\varepsilon$ and $p_2=(2-\varepsilon)/(1-\varepsilon)$ gives 
\[
E_1(Y_1,Y_2,C)\leq \left(\sum_{\bm{c}}\left|\sum_{r_1}S_{r_1}^\natural(\bm{c})\right|^{p_1}\right)^{1/p_1}\left(\sum_{\bm{c}}\left|\sum_{r_2}S_{r_2}^\natural(\bm{c})\right|^{p_2}\right)^{1/p_2}.
\]
Next we use Conjecture \ref{CNJ:(R2'E')} with $Z=\ceil{\max(C,\frac13Y_1)}$ and $R=Y_1$ to estimate the first sum and Theorem \ref{thm:B3} with $Z=\ceil{\max(C,\frac13Y_2)}$ and $R=Y_2$ for the second, to obtain 
\begin{equation}\label{Eq: LastCaseEndgame}  
E_1(Y_1,Y_2,C)\ll \widehat{Y}^{1/2}\widehat{Y}_2^{-\eta/p_2}\left(\widehat{C}^{n/p_1}+\widehat{Y}_1^{n/3p_1}\right)\left(\widehat{C}^{n/p_2}+\widehat{Y}_2^{n/3p_2}\right)
\end{equation}
for some $\eta>0$. If we choose $\delta$ sufficiently small, we have $\widehat{C}>|P|^{1/2-\delta}\gg |P|^{1/4}\gg \widehat{Y}^{1/6}$. As $\widehat{Y}_1\widehat{Y}_2=\widehat{Y}$, this implies that $\widehat{Y}_1^{1/3}\ll \widehat{C}$ or $\widehat{Y}_2^{1/3}\ll \widehat{C}$. Therefore, the bound in \eqref{Eq: LastCaseEndgame} gives
\begin{align*}
E_1(Y_1,Y_2,C)&\ll \widehat{Y}^{1/2}\widehat{Y}_2^{-\eta/p_2}\widehat{C}^{n/p_2}\left(\widehat{C}^{n/p_1}+\widehat{Y}_1^{n/3p_1}+\widehat{C}^{n(1/p_1-1/p_2)}\widehat{Y}_2^{n/3p_2}\right).
\end{align*}
Moreover, if the third term in the brackets dominates then we must have $\widehat{C}^3\ll \widehat{Y}_2$. Thus, on noting that $1/p_1-1/p_2>0$, we obtain
\begin{align*}
    E_1(Y_1,Y_2,C)&\ll \widehat{Y}^{1/2}\widehat{C}^{n/p_2}\widehat{Y}_2^{-\eta/p_2}\left(\widehat{C}^{n/p_1}+\widehat{Y}_1^{n/3p_1}+\widehat{Y}_2^{n/3p_1}\right)\\
&\ll  \widehat{Y}^{1/2}\widehat{C}^{n/p_2}\widehat{Y}_2^{-\eta'}\left(\widehat{C}^{n/p_1}+\widehat{Y}^{n/3p_1}\right),
\end{align*}
where we trivially estimated $\widehat{Y}_1,\widehat{Y}_2\leq \widehat{Y}$. Moreover, we have set $\eta'=\eta/p_2$, which satisfies $\eta'\geq \eta/4$ for $\varepsilon$ small enough.

Observe that $n/p_2+1-n/2>0$, so that summing over $Y_1,Y_2,C$ and using the integral estimate \eqref{Eq: IntEstimateEndgame} shows that the contribution from this case to $E_1(P)$ is 
\begin{align*}
    &\ll \sum_{Y_1,Y_2, C}|P|^{n-3}\widehat{C}^{n/p_2}\widehat{Y}^{1-n/2}\widehat{Y}_2^{-\eta'}\left(1+\frac{|P|\widehat{C}}{\widehat{Y}}\right)^{1-n/2}\left(\widehat{C}^{n/p_1}+\widehat{Y}^{n/3p_1}\right)\\
    &\ll \sum_{Y_1,Y_2}|P|^{n/2-2}\widehat{Y}_2^{-\eta'}\sum_C \widehat{C}^{1+n/p_2-n/2}\left(\widehat{C}^{n/p_1}+\widehat{Y}^{n/3p_1}\right)\\
    &\ll \sum_{Y_1,Y_2}|P|^{n/2-2}\widehat{Y}_2^{-\eta'}|P|^{(1+n/p_2-n/2)/2}\left(|P|^{n/2p_1}+\widehat{Y}^{n/3p_1}\right)\\
    &\ll \sum_{Y_1,Y_2}|P|^{3n/4-3/2}\widehat{Y}_2^{-\eta'},
\end{align*}
where we used $1/p_1+1/p_2=1$ and that $\widehat{C}\ll |P|^{1/2}$, as well as the fact that  $\widehat{Y}\ll |P|^{3/2}$. Since $3n/4-3/2\leq n-3$ for $n\geq 6$, to complete our treatment of $E_1(P)$ it will suffice to show that $\sum_{Y_1,Y_2}\widehat{Y}_2^{-\eta'}\ll 1$.
To do so, recall that $\widehat{Y}_2^{n+\varepsilon} > \widehat{W}^{\alpha/2}$.
If we set $\beta = \eta'\alpha/(4(n+\varepsilon))$, then recalling the definition of $\widehat{W}$ in \eqref{Eq: DefiMinAlpha} this implies 
\begin{align*}
    \widehat{Y}_2^{-\eta'}&\leq \widehat{Y}_2^{-\eta'/2}\widehat{W}^{-\beta}\\
    &\ll_\ve \widehat{Y}_2^{-\eta'/2}\left(\left(\frac{\widehat{Y}}{|P|^{3/2}}\right)^\beta +|P|^{\beta(\ve-1)/2}\right).
\end{align*}
Therefore, as $\sum_{Y_1,Y_2} |P|^{\beta(\ve-1)/2} \le Q^2 |P|^{\beta(\ve-1)/2}\ll 1$, we get
\begin{align*}
    \sum_{0\leq Y_2\leq Q}\sum_{0\leq Y_1 \leq Q-Y_2}\widehat{Y}_2^{-\eta'}&\ll |P|^{-3\beta/2}\sum_{0\leq Y_2\leq Q}\widehat{Y}_2^{-\eta'/2+\beta}\sum_{0\leq Y_1 \leq Q-Y_2}\widehat{Y}_1^\beta \\
    &\ll \sum_{0\leq Y_2\leq Q}\widehat{Y}_2^{-\eta'/2}\\
    &\ll 1,
\end{align*}
as $\eta'>0$. This completes our
proof of  Proposition \ref{Prop:nonDualContribution}.

\section{Contribution from the centre}

In this section we estimate the contribution from $\bm{c}=\mathbf{0}$ in 
\eqref{Eq: DeltaMethod}. 
This takes the shape
\begin{equation}\label{eq:darjeeling-tea}
M(P)
=|P|^n\sum_{\substack{r\in \Omon\\ |r|\leq \widehat{Q}}}|r|^{-n}
S_r(\mathbf 0)I_r(\mathbf 0),
\end{equation}
where $S_r(\bm{0})$ is given by \eqref{Eq: Definition S_r(c)}, 
and  
$I_r(\bm{0})$ is given by  \eqref{Eq: Definition I_r(theta,c)} and \eqref{eq:2-1 to brighton}.
The following is the main goal of this section. 

\begin{proposition}\label{prop:MP}
Let $n\geq 5$. Then 
$
M(P)\ll |P|^{n-3},
$
for any $Q\geq 0$.
\end{proposition}

To begin with, 
it follows from \eqref{Eq: IntEstimateEndgame} 
that 
$
I_r(\mathbf 0) \ll |P|^{-3},
$
if $n\geq 3$. Hence
\begin{equation}\label{eq:m-speak}
M(P)\ll |P|^{n-3} \sum_{r \in \Omon}
|r|^{-n}|S_r(\mathbf 0)|=
|P|^{n-3} \sum_{r \in \Omon}
|r|^{(1-n)/2}|S_r^\natural(\mathbf 0)|,
\end{equation}
if $n\geq 3$.
We shall require the following technical result, which 
will also be useful in the next section.

\begin{lemma}
\label{LEM:sum-S_0(n)-trivially}
If $R\ge 0$ and $\ve>0$, 
then
\begin{equation*}
\sum_{\substack{r\in \Omon\\  \abs{r}=\hat R}} \abs{r}^{1-n/2}
\sum_{d\mid r} |d|^{-1/2} \lvert{S^\natural_d(\bm{0})}\rvert
\ll_\ve \hat R^{(4-n)/3+\ve}.
\end{equation*}
\end{lemma}

The statement of Proposition \ref{prop:MP} follows on combining  this result with \eqref{eq:m-speak}, to deduce that 
\begin{align*}
\sum_{r \in \Omon}
|r|^{(1-n)/2}|S_r^\natural(\mathbf 0)| 
&\leq \sum_{R\geq 0} 
\sum_{\substack{r\in \Omon\\  \abs{r}=\hat R}} \abs{r}^{1-n/2}
\sum_{d\mid r} d^{-1/2} \lvert{S^\natural_d(\bm{0})}\rvert\\
&\ll_\ve \sum_{R\geq 0}  
 \hat R^{(4-n)/3+\ve}\\
& \ll 1,
\end{align*}
if $n\geq 5$.

\begin{proof}[Proof of Lemma \ref{LEM:sum-S_0(n)-trivially}]

Recalling the notation for $\cub(d)$ introduced in \eqref{eq:sq-cub}, 
it follows from Lemma~\ref{lem:lem8.2'} that 
$|d|^{-1/2}S_{d}^\natural(\bm{0})\ll_\ve  |\cub(d)|^{n/6+\ve}$,
for any $\ve>0$. Hence
$$
\sum_{\substack{r\in \Omon\\  \abs{r}=\hat R}} \abs{r}^{1-n/2}
\sum_{d\mid r} d^{-1/2} \lvert{S^\natural_d(\bm{0})}\rvert
\ll_\ve 
\hat R^{1-n/2+\ve}
\sum_{\substack{r\in \Omon\\  \abs{r}=\hat R}} 
|\cub(r)|^{n/6},
$$
on taking the trivial estimate for the divisor function. 
Writing $r=r_1r_2$, where $r_1$ is cube-free and $r_2$ is cube-full, we see that 
\begin{align*}
\sum_{\substack{r\in \Omon\\  \abs{r}=\hat R}} 
|\cub(r)|^{n/6}
&\leq
\sum_{\substack{r_1\in \Omon\\  \text{$r_1$ cube-free}}}
\sum_{\substack{r_2\in \Omon\\  
|r_2|\leq \hat R/|r_1|\\
\text{$r_2$ cube-full}}} 
|r_2|^{n/6}\\
&\ll
\hat R^{1/3+n/6}
\sum_{r_1\in \Omon} \frac{1}{|r_1|^{1/3+n/6}}\\
&\ll
\hat R^{1/3+n/6},
\end{align*}
since $n\geq 5$. The statement of the lemma easily follows. 
\end{proof}

\section{Contribution from dual variety}

Let $w$ be as in \S~\ref{SEC:integral-stuff}, with the properties specified in Hypothesis 
\ref{hyp}, with parameter $L=0$.
The purpose of this section is to study the quantity
\begin{equation}\label{eq:rosehip-tea}
E_2(P)=|P|^n\sum_{\substack{\bm{c}\in\OK^n\setminus\{\mathbf{0}\}\\F^*(\bm{c})= 0}}\sum_{\substack{r\in\Omon\\ |r|\leq \widehat{Q}}}|r|^{-n}S_r(\bm{c})I_r(\bm{c}),
\end{equation}
where  $Q$ is chosen so that $|P|^{3/2}\asymp \widehat{Q}$ and 
$I_r(\bm{c})=0$ unless $|\bm{c}|\ll  |P|^{1/2}$, by Lemma \ref{Le: IntVanish}. 
We shall be interested in this quantity when $n=6$ and 
$F=x_1^3+\dots+x_6^3$. Let $V\subset \PP_K^5$ denote the cubic hypersurface defined by $F$.

Let $\Upsilon$ be the set of $3$-dimensional $K$-vector spaces $L\belongs K^6$ with $F\vert_L = 0$. (Note that the projectivization of $\Upsilon$ is the $K$-point set of the Fano variety $F_2(V)$ of planes on $V$.)
Since $F$ is non-singular, it is well-known that $\Upsilon$ is finite, as explained by Starr \cite{browning2006density}*{Appendix}.
The task of estimating $E_2(P)$  is divided into two steps. 
 First, 
in Proposition \ref{Prop: LinearSpaceBias} 
we will establish a bias in the exponential sum $S_r(\bm{c})$ whenever $\bm{c}\in L^\perp$, for some $L\in \Upsilon$. Note that the projectivization of $L^\perp$ corresponds to a plane 
contained in the dual variety of $V$. Planes contained in the dual variety of $V$ correspond to planes in $V$  via bi-duality. Using Poisson summation, we shall show that the  bias of the exponential sums 
gives a contribution that exactly matches the contribution from the points of bounded height on 
$\bigcup_{L\in \Upsilon} L$.
The second task will be to show that the vectors $\bm{c}\in \OK^6$ such that 
$F^*(\bm{c})= 0$, but which do not lie on 
$L^\perp$ for any $L\in \Upsilon$, are sparse and form a negligible contribution.

 The goal of this section  is summarised in the  following result.

\begin{proposition}\label{Prop:DualContribution}
Let $\eps>0$. 
If $F=x_1^3+\dots+x_6^3$, then
$$
E_2(P)
= \sum_{L\in \Upsilon} \sum_{\bm{x}\in L\cap \OK^6} w(\bm{x}/P)
+ O_\ve(\abs{P}^{3-1/4+\eps}).
$$
\end{proposition}

In fact, much of our argument carries over to general $F$ in $n\in \{4,6\}$ variables, but a restriction to diagonal $F$ is necessary when it comes to handling square-full moduli.  
 Throughout this section we shall therefore  assume that $F=x_1^3+\dots+x_6^3$, unless otherwise indicated.
For later convenience, we note that if $\cha(\FF_q) > 3$ there 
 is a unique $L\in \Upsilon$ up to a $\FF_q$-linear automorphism of $F$, given by 
\begin{equation}\label{Eq: TheLinSpace}
L_0: \quad x_1+x_4 = x_2+x_5 = x_3+x_6 = 0.
\end{equation}
A proof of this fact is given in \cite{wang2022thesis}*{Remark 6.3.8},  based on \cites{kontogeorgis2002automorphisms,shioda1988arithmetic}.

\subsection{Geometric preliminaries}

Throughout this subsection, 
$k$ denotes an arbitrary field with $\cha(k)\neq 2$ and 
we shall adopt the notation $\PP(E)$ for the projectivization of a $k$-vector space $E$. (In particular, $\PP^n$ denotes the projective space $\PP(k^{n+1})$.)
Suppose that $V\subset \PP^5$ is a smooth cubic fourfold containing a $k$-plane $M\subset \PP^5$,  corresponding to a three-dimensional vector space $L\subset k^6$.
Projection away from $M$, defined by the natural map $\PP(k^6)\to \PP(k^6/L)$, gives a morphism $V\setminus M \to \PP(k^6/L)$. Blowing up $V$ along $M$ resolves the indeterminacy locus and gives a fibration $\widetilde{V}\to \PP(k^6/L)$ whose fibres are projective quadric surfaces. Being the blow-up of a smooth variety along a smooth subvariety, $\widetilde{V}$ is also smooth. If we denote by $\mathcal{C}\subset \PP(k^6/L)$ the zero scheme
of the determinant of the matrices that define the quadratic forms in the fibers, it follows from Proposition~1.2 and Exemple~1.4.2 of Beauville~\cite{BeauvilleQuadricBundles} that $\mathcal{C}$ is a curve of degree 6 which is either smooth or has ordinary double points.
Let $L^\perp$ be the orthogonal complement of $L$ inside $k^6$ and note that $\PP(L^\perp)$ parameterises hyperplanes in $\PP(k^6)$ containing $M$.
If $\bm{c}\in \PP(L^\perp)$,
we shall write $H_{\bm{c}}$ for the corresponding hyperplane containing $M$ in $\PP(k^6)$ and $H_{\bm{c}}'$ for the image of $H_{\bm{c}}$ in $\PP(k^6/L)$, which defines a projective line.

If we denote by $V_{\bm{c}}=H_{\bm{c}}\cap V$, then we get again a morphism $V_{\bm{c}}\setminus M\to H_{\bm{c}}'$ by projecting away from $M$. Blowing up $V_{\bm{c}}$ gives a fibration $\widetilde{V}_{\bm{c}}\to H_{\bm{c}}'$ into quadric surfaces and the locus of degenerate quadric surfaces is given by $\mathcal{C}_{\bm{c}}=H_{\bm{c}}'\cap \mathcal{C}$. 

Note that after a suitable change of variables we may assume that $L$ is given by $x_1=x_2=x_3=0$, so that $V$ is defined by a cubic form of the shape 
$$
x_1Q_1(\bm{x},\bm{y})+x_2Q_2(\bm{x},\bm{y})+x_3Q_3(\bm{x},\bm{y}),
$$ 
for some quadratic forms $Q_1, Q_2, Q_3 \in k[x_1,x_2,x_3,y_1,y_2,y_3]$. Upon defining 
\[
W\coloneqq \{(\bm{s}, \bm{y})\in \PP^2 \times \PP^2\colon s_1Q_1(\bm{0},\bm{y})+s_2Q_2(\bm{0},\bm{y})+s_3Q_3(\bm{0},\bm{y})=0\},
\]
we obtain a conic bundle $W\to \PP^2$ by projecting to the first factor. A straightforward computation with the Jacobian criterion shows that smoothness of $V$ along $L$ implies smoothness of $W$. Hence, on appealing once again to Proposition~1.2 of Beauville~\cite{BeauvilleQuadricBundles}, it follows that $f(\bm{s})=\det(s_1M_1+s_2M_2+s_3M_3)$ is non-zero as a polynomial in $\bm{s}$, where $M_i$ is the $3\times 3$ matrix underlying $Q_i$. The equation $f=0$ thus defines a curve $\mathcal{D}\subset \PP^2$.
\begin{definition}\label{Def: c.good}
    Let $\bm{c}\in L^\perp \setminus\{\bm{0}\}$. We say that $\bm{c}$ is \emph{good} if 
    \begin{enumerate}[(i)]
        \item $\mathcal{C}_{\bm{c}}$ is smooth of dimension $0$, 
        \item $H_{\bm{c}}\subset \PP^5$ contains no plane contained in $V$ other than $M$,
        \item $H'_{\bm{c}}$ is not an irreducible component of $\mathcal{D}$,
    \end{enumerate}
    and we say that $\bm{c}$ is \emph{bad} otherwise.
    \end{definition}
    Note that the definition depends on the choice of $L$. The significance of this definition is that hyperplane sections coming from good $\bm{c}$'s enjoy good geometric properties and bad vectors $\bm{c}$ only occur rarely, as we will see in the next two results.

\begin{lemma}\label{Le:Goodc.are.rare}
    Suppose $\bm{c}$ is bad. Then $\bm{c}$ lies in a subvariety of $\PP(L^\perp)$, all of whose irreducible components have dimension at most one.
\end{lemma}
\begin{proof}
Since $\mathcal{C}$ is reduced and has at worst ordinary double points, a generic hyperplane section is smooth of dimension $0$. Indeed, as long as $H'_{\bm{c}}$ does not lie on the dual variety of $\mathcal{C}$ or contain one of the singular points of $\mathcal{C}$, the intersection $H'_{\bm{c}}\cap \mathcal{C}$ will be smooth. Moreover, it will be $0$-dimensional as long as $H'_{\bm{c}}$ is not a component of $\mathcal{C}$. Each case forces $\bm{c}$ to lie on a subvariety of dimension at most one. 
    If $H_{\bm{c}}$ contains two distinct planes contained in $V$, then it must also contain the linear space they span, which has dimension at least 3. Hyperplanes in $\PP(k^6)$ containing a fixed three dimensional linear space are parameterised by a $\PP^1$ and since there are only finitely many planes contained in $V$, this shows that $\bm{c}$ lies on a union of finitely many lines. Finally, the curve $\mathcal{D}$ has at most 3 irreducible components and so there are at most 3 possibilities for $\bm{c}$ such that $H_{\bm{c}}'$ is an irreducible component of $\mathcal{D}$.
\end{proof}

\begin{lemma}\label{Le: Goodcsaregood}
    Suppose that $\bm{c}$ is good. Then every fibre of the quadric surface bundle $\widetilde{V}_{\bm{c}}\to H_{\bm{c}}'$ has at worst isolated singularities.
\end{lemma}
\begin{proof}
    For $\bm{s}\in H_{\bm{c}}'$, let $Q_{\bm{s}}$ be the fibre above $\bm{s}$ in $\widetilde{V}_{\bm{c}}$, which is a projective quadric surface and hence corresponds to some matrix $M_{\bm{s}}\in \mathrm{Mat}_{4\times 4}(k)$. The claim of the lemma is now equivalent to $M_{\bm{s}}$ having rank at least 3. Note that if $\rk M_{\bm{s}} = 1$, then $Q_{\bm{s}}$ contains a double plane. However, $Q_{\bm{s}}$ is also a fibre of the fibration $\widetilde{V}\to \PP(k^6/L)$ and is thus the residual intersection of a linear $3$-space containing $M$ with the cubic hypersurface $V$. Therefore, a result of Zak~\cite{hooley1991number}*{Katz's Appendix, Theorem 2} shows that the singular locus of $Q_{\bm{s}}$ has dimension at most 1 and so $Q_{\bm{s}}$ cannot be a double plane.
    Now assume that $\rk M_{\bm{s}}=2$. This implies that $Q_{\bm{s}}$ is a union of two distinct planes. However, this is impossible as at least one such plane would be distinct from $M$ and  contained in $H_{\bm{c}}$. This would contradict the assumption that $\bm{c}$ is good.
\end{proof}

\subsection{Exponential sums} 
In this section it is no harder to work with an arbitrary cubic form $F\in \OK[x_1,\dots, x_6]$ 
defining a smooth hypersurface  $\mathcal{V}\subset \PP^5_{\OK}$. Let $V\coloneqq \mathcal{V}\times \Spec \FF_q(t)$ and for a prime $\varpi$ of $\OK$ we let $\FF_\varpi = \OK/\varpi\OK$ and set $\mathcal{V}_\varpi \coloneqq \mathcal{V}\times \Spec \FF_\varpi$. Moreover, we fix a $K$-vector space $L\subset K^6$ such that $F_{|L}\equiv 0$ and denote by $L^\perp$ its orthogonal complement inside $K^6$. We then set $\Lambda = L\cap \OK^6$ and $\Lambda^\perp = L^\perp \cap \OK^6$. We shall  write $M$ for the subspace of $\PP^{5}$ corresponding to $L$. In addition, we let $L_\varpi = \Lambda/\varpi \Lambda$ and $L^\perp_\varpi = \Lambda^\perp/\varpi \Lambda^\perp$, which are both $\FF_\varpi$-vector spaces of dimension 3. 

Let $\bm{c}\in L^\perp\setminus\{\bm{0}\}$. Abusing the definition of the previous subsection slightly, we shall also say that $\bm{c}$ is good if the corresponding point in projective space is good. Let $g\in \OK[y_1,y_2,y_3]$ be a degree 6 form defining the curve $\mathcal{C}$ in $\PP(K^6/L)$. Note that if $\bm{c}\in L^\perp\setminus\{\bm{0}\}$ is good, then by definition $\mathcal{C}_{\bm{c}}\subset H_{\bm{c}}'$ is smooth. In particular, if we work with coordinates $(s:t)$ on $H_{\bm{c}}'$, then there exists a separable binary form $g_{\bm{c}}\in \OK[s,t]$ of degree 6 whose coefficients are degree 6 polynomials in $\bm{c}$ that define $\mathcal{C}_{\bm{c}}$. 
\begin{example}
    Suppose that $L=\{x_1=x_2=x_3=0\}$. Then any $\bm{c}\in L^\perp\setminus\{\bm{0}\}$ takes the shape $(c_1,c_2,c_3,0,0,0)$ and, without loss of generality, we may assume that $c_3\neq 0$ and $c_1,c_2,c_3\in\OK$. Rational points on $H_{\bm{c}}'$ are of the form $(c_3s:c_3t: -(c_1s+c_2t))$ with $(s:t)\in \PP^1(\QQ)$ and we may take $g_{\bm{c}}(s,t)=g(c_3s,c_3t, -(c_1s+c_2t))$.
\end{example}
Fix coordinates $z_1,z_2, z_3$ on $L^\perp$. It follows from Lemma~\ref{Le:Goodc.are.rare} that the collection of all bad $\bm{c}\in \PP(L^\perp)$ is contained in a hypersurface. In particular, there exists a non-zero form $B\in \OK[z_1,z_2,z_3]$ such that if $B(\bm{c})\neq 0$, then $\bm{c}$ is good. Moreover, as Lemma~\ref{Le:Goodc.are.rare} is independent of the base field, it follows that if $\varpi \nmid B(\bm{c})$, then the reduction of $\bm{c}$ modulo $\varpi$ is also good with respect to $L_\varpi$ and $\FF_\varpi$. 
Recall from part (1) of 
Lemma \ref{lem:prime_power} 
that $S_\varpi^\natural(\bm{c})\ll |\varpi|^{1/2}$ if $F^*(\bm{c})=0$. The following result refines this estimate and  establishes a bias in the exponential sums for primes. 

\begin{proposition}\label{Prop: LinearSpaceBias}
    Suppose that $\bm{c}$ is good and $\varpi \nmid B(\bm{c})$. Then 
$$
    S^\natural_\varpi(\bm{c})=|\varpi|^{1/2}+O(1).
$$
\end{proposition}
\begin{proof}
It follows from \eqref{EQN:rewrite-S_c(p)-via-E_c} that 
        \begin{equation}\label{eq:pure-veg}
    S_{\varpi}(\bm{c})=|\varpi|(|\varpi|\#\mathcal{V}_{\bm{c},\varpi}(\FF_\varpi)-\#\mathcal{V}_\varpi(\FF_\varpi)).
    \end{equation}
    After a suitable $\FF_\varpi$-linear change of variables, we may assume that $L_\varpi$ is given by $x_1=x_2=x_3=0$ and the cubic form $F$ takes the shape 
    $$
    F(\bm{x},\bm{y})=x_1Q_1(\bm{x},\bm{y})+x_2Q_2(\bm{x},\bm{y})+x_3Q_3(\bm{x},\bm{y}),$$ 
    for some quadratic forms $Q_1,Q_2,Q_3\in \FF_\varpi[x_1,x_2,x_3,y_1,y_2,y_3]$. In particular, we have $\bm{c}=(c_1,c_2,c_3,0,0,0)\in \FF_\varpi^6$. Moreover, as $g_{\bm{c}}\neq 0$, we must have $\bm{c}\neq 0$. Without loss of generality, assume that $c_3\neq 0$, so that $\mathcal{V}_{\bm{c},\varpi}$ is defined by the vanishing of the form $F_{\bm{c}}(s,t, \bm{y})=F(c_3s,c_3t, -(c_1s+c_2t),\bm{y})$. 
Then we may rewrite $F_{\bm{c}}$ as $sq_1(s,t,\bm{y})+tq_2(s,t,\bm{y})$, where
    \[
    q_i(s,t, \bm{y})= c_3Q_i(c_3s,c_3t, -(c_1s+c_2t),\bm{y})-c_iQ_3(c_3s,c_3t, -(c_1s+c_2t),\bm{y}).
    \]
     Any $(s,t)\in \FF_\varpi^2\setminus\{\bm{0}\}$ satisfies $\beta s -\alpha t=0$ for a unique $(\alpha:\beta)\in \PP^1(\FF_\varpi)$ and we may then write $(s:t:y_1:y_2:y_3)= (\alpha u:\beta u: y_1: y_2:y_3)$ for some $u\in\FF_\varpi\setminus\{0\}$. If $(\alpha u: \beta u: y_1: y_2: y_3)\not\in M$, then $u\neq 0$ and the equation $F_{\bm{c}}(s,t,\bm{y})=0$ becomes 
    \begin{equation}\label{Eq: QuadricBundle}
        \alpha q_1(\alpha u, \beta u, y_1,y_2,y_3)+\beta q_2(\alpha u, \beta u, y_1,y_2,y_3)=0. 
    \end{equation}
    This is the defining equation for the residual quadric surface that arises when intersecting $\mathcal{V}_{\bm{c},\varpi}$ with the hyperplane spanned by $L_\varpi$ and $(s:t)$. Upon defining the quadric surfaces 
    \[
    S_{(\alpha:\beta)}= \{ (u:y_1:y_2:y_3)\in \PP^3\colon \eqref{Eq: QuadricBundle}\text{ holds}\}
    \]
    and the conics
    \[
    C_{(\alpha:\beta)}=\{(y_1:y_2:y_3)\in \PP^2\colon \alpha q_1(0,0,y_1,y_2,y_3)+\beta q_2(0,0, y_1,y_2,y_3)=0\},
    \]
    we see that $C_{(\alpha,\beta)}$ is precisely $S_{(\alpha,\beta)}\cap L_\varpi$ and we may therefore rewrite 
    \begin{equation}\label{Eq: IncExcQuadrics}
    \#\mathcal{V}_{\bm{c},\varpi}(\FF_\varpi)= \# L_{\varpi}(\FF_\varpi)+ \sum_{(\alpha:\beta )\in \PP^1(\FF_\varpi)}\#S_{(\alpha:\beta)}(\FF_\varpi) -\sum_{(\alpha:\beta)\in\PP^1(\FF_\varpi)}\#C_{(\alpha:\beta)}(\FF_\varpi).
    \end{equation}

    Since $\varpi\nmid B(\bm{c})$, the reduction of $\bm{c}$ modulo $\varpi$ is good and hence the reduction of the binary form $g_{\bm{c}}$ modulo $\varpi$ is separable of degree 6. Moreover, by construction $S_{(\alpha:\beta)}$ will be singular precisely when $g_{\bm{c}}(\alpha,\beta)=0$, since $g_{\bm{c}}(\alpha,\beta)$ is just the determinant of the matrix underlying the quadratic form defining $S_{(\alpha:\beta)}$.
    Let $\chi\colon \FF_\varpi^\times\to \CC^\times$ be the unique character of order 2.
    It  follows from Schmidt \cite{schmidt1976}*{\S~IV.2} that    
$$
        \#S_{(\alpha:\beta)}(\FF_\varpi)= |\varpi|^2 +|\varpi|(1+\chi(g_{\bm{c}}(\alpha,\beta)))+1,
$$
provided $g_{\bm{c}}(\alpha,\beta)\neq 0$.
    Let $0\leq N_0 \leq 6$ be the number of $(\alpha:\beta)\in \PP^1(\FF_\varpi)$ such that $g_{\bm{c}}(\alpha,\beta)=0$. It then follows that
    \begin{align*}
        \sum_{\substack{(\alpha:\beta )\in \PP^1(\FF_\varpi)\\ g_{\bm{c}}(\alpha,\beta)\neq 0}}
        \hspace{-0.5cm}\#S_{(\alpha:\beta)}(\FF_\varpi) &= (|\varpi|^2+|\varpi|)(|\varpi|+1-N_0) +|\varpi|
                \hspace{-0.2cm}
        \sum_{ \substack{t\in \FF_\varpi \\ g_{\bm{c}}(t,1)\neq 0}}
                \hspace{-0.2cm}
         \chi(g_{\bm{c}}(t,1))+O(1)\\
        &= (|\varpi|^2+|\varpi|)(|\varpi|+1-N_0) + O(|\varpi|^{3/2}), 
    \end{align*}
    by the classical Weil bound for character sums, which is applicable since $g_{\bm{c}}(t,1)$ is square-free. 

    As the reduction of $\bm{c}$ modulo $\varpi$ is good, we know from Lemma~\ref{Le: Goodcsaregood} that the matrix underlying the quadratic form defining $S_{(\alpha:\beta)}$ has rank at least 3. In particular, $S_{(\alpha:\beta)}$ is irreducible and by Lang-Weil we get $\#S_{(\alpha:\beta)}(\FF_\varpi)=|\varpi|^2 +O(|\varpi|^{3/2})$. Thus 
$$
        \sum_{\substack{(\alpha:\beta )\in \PP^1(\FF_\varpi)\\ g_{\bm{c}}(\alpha,\beta)= 0}}\#S_{(\alpha:\beta)}(\FF_\varpi) = N_0 |\varpi|^2 +O(|\varpi|^{3/2}). 
$$

Finally, observe that since $\varpi\nmid B(\bm{c})$, property (iii) in Definition~\ref{Def: c.good} implies that there are at most 3 possible $(\alpha:\beta)\in \PP^1(\FF_\varpi)$ for which $C_{(\alpha:\beta)}$ is singular. If $C_{(\alpha:\beta)}$ is smooth, then it contains $|\varpi|+1$ $\FF_\varpi$-points, while if it is singular it is either a double line or a union of two lines, so that in either case contains $O(|\varpi|)$ $\FF_\varpi$-points. Therefore, 
$$
    \sum_{(\alpha:\beta)\in \PP^1(\FF_\varpi)}\#C_{(\alpha:\beta)}(\FF_\varpi)= |\varpi|^2 +O(|\varpi|). 
$$
Combining these estimates  with the fact that $\#L_\varpi(\FF_\varpi)=|\varpi|^2+|\varpi|+1$, it follows from \eqref{Eq: IncExcQuadrics} that 
\[
\#\mathcal{V}_{\bm{c},\varpi}(\FF_\varpi)= |\varpi|^3+2|\varpi|^2 +O(|\varpi|^{3/2}).
\]
Moreover, by Deligne's bounds (as recorded in Equation (3.12) and the paragraph thereafter of \cite{browning2015rational}), we have $\#\mathcal{V}(\FF_\varpi)= |\varpi|^4+|\varpi|^3 +O(|\varpi|^2)$. Thus 
\eqref{eq:pure-veg} yields
\begin{align*}
    S_{\varpi}(\bm{c})&=|\varpi|(|\varpi|\#\mathcal{V}_{\bm{c},\varpi}(\FF_\varpi)-\#\mathcal{V}_\varpi(\FF_\varpi))
    = |\varpi|^4+O(|\varpi|^{7/2}).
\end{align*}
On dividing both sides by $|\varpi|^{7/2}$, we are led to the statement of the proposition. 
\end{proof}

\newcommand{\LcG}{\mathcal{N}_{\bm{c}}^{\map{G}}} 
\newcommand{\LcB}{\mathcal{N}_{\bm{c}}^{\map{B}}} 

Motivated by the previous result, for good $\bm{c}\in L^\perp \setminus\{\bm{0}\}$ we define 
\[
\LcG\coloneqq \{r\in \OK^+\colon \gcd(r,B(\bm{c}))=1\}
\quad\text{and}\quad
\LcB\coloneqq \{r\in \OK^+\colon \rad(r)\mid B(\bm{c})\}
\]
and observe that  these sets depend implicitly on $L$. Moreover, it will be convenient to introduce the notation 
\begin{equation*}
\mcal{N}_{\ge}(l) \defeq \set{r\in \Omon: \varpi\mid r\Rightarrow v_\varpi(r)\ge l},
\end{equation*}
for any integer $l\geq 1$.
\subsection{A general upper bound} 
Having completed the treatment of individual exponential sums, we will now focus on our case of interest and assume that 
\[
F(\bm{x})=x_1^3+\dots +x_6^3
\]
for the rest of this section.
Recall the definition \eqref{EQN:define-S_0,S_1-for-Delta-vanishing-and-non-zero-loci} of $\mathcal{S}_0$. 
For each set $\mathcal{T}\belongs \mathcal{S}_0$, let
\begin{equation}
\label{EQN:sparse-absolute-contribution-f(S)-to-delta-method}
f(\mcal{T})\defeq \abs{P}^6 \sum_{\bm{c}\in \mcal{T}\setminus \set{\bm{0}}}
\sum_{
\substack{r\in \Omon\\ 
\abs{r}\le \widehat{Q}}} \abs{I_r(\bm{c})} \cdot \abs{r}^{-6}
\cdot \sum_{s\mid r} \abs{r/s}^{4} \abs{S_s(\bm{c})}.
\end{equation}
At several points later, it will be convenient to
discard $f(\mcal{T})$ for various choices of $\mcal{T}$. This is achieved through the following result (which is analogous to \cite{wang2023special}*{Lemma~5.2}).

\begin{lemma}
\label{LEM:sparse-bound}
Let $\mathcal{T}\belongs \mathcal{S}_0$ and $\delta\in \RR$.
Suppose $\set{\lambda\bm{c}: (\lambda, \bm{c})\in K^\times \times \mcal{T}} \cap \OK^n = \mcal{T}$ and $\#{\set{\bm{c}\in \mcal{T}: \norm{\bm{c}}\le \hat C}}=O(\hat C^{n-3-\delta})$ for all $C\in \RR_{\geq 0}$. 
Assume $\delta\leq \frac52$.
Then $$
f(\mcal{T})\ll_\ve \abs{P}^{3- \delta/2 + \eps},
$$
for any $\eps>0$.
\end{lemma}

\begin{proof}
The result  \cite{glas2022question}*{Lemma 6.1}  essentially gives a version of this result with $\delta<0$,
and with $s=r$ instead of $\sum_{s\mid r}$. 
The same method extends to the present setting, with minor modifications that we now describe.

First, the integral estimate \cite{glas2022question}*{Eq.~(6.9)} still holds.
(In fact, we have seen that the stronger estimate \eqref{Eq: IntEstimateEndgame} is available.)
Second, taking $n=6$, we have 
\begin{equation*}
S_s(\bm{c})
\ll \abs{s}^\ve \abs{s_1}^{4} \abs{s_2}^{14/3} \abs{s_3}^{5} \abs{\gcd(s_2,\bm{c})}^{1/2},
\end{equation*}
in the notation of \cite{glas2022question}.
Therefore,
\begin{equation*}
\sum_{s\mid r} \abs{r/s}^{4} \abs{S_s(\bm{c})}
\ll \abs{r}^{\ve} \abs{r_1}^{4} \abs{r_2}^{14/3} \abs{r_3}^{5} \abs{\gcd(r_2,\bm{c})}^{1/2},
\end{equation*}
since $s_2\mid r_2$ and $s_3\mid r_3$.
This is of the exact same strength as the bound for the $s=r$ term given in 
\cite{glas2022question}*{Lemma~6.1}.
Since $-\frac52 + \delta \le 0$, the rest of the argument of \cite{glas2022question} goes through with $-\delta$ in place of  $\eta$ to give the desired upper bound for $f(\mcal{T})$.
\end{proof}

\subsection{Factorising exponential sums}

Given the bias established in Proposition~\ref{Prop: LinearSpaceBias}, we typically expect $S^\natural_r(\bm{c})$ to be of size $|r|^{1/2}$ for $\bm{c}\in \Lambda^\perp$, at least when $r$ is square-free. To keep track of the error, we will decompose $S_r^\natural(\bm{c})$ into smaller pieces. 

For $r\in \OK$, let $\phi_K(r)=\#(\OK/r\OK)^\times$ be the Euler totient function. Consider the Dirichlet series
\begin{equation*}
\Psi(s) \defeq
\sum_{r\in \Omon} \frac{\phi_K(r) \abs{r}^{-1/2}}{\abs{r}^s}
= \prod_\varpi \frac{1 - \abs{\varpi}^{-s-1/2}}{1 - \abs{\varpi}^{-s+1/2}}
= \frac{\zeta_K(s-\frac12)}{\zeta_K(s+\frac12)}.
\end{equation*}
In the notation of \eqref{EQN:define-Phi}, it follows from  Proposition~\ref{Prop: LinearSpaceBias} that  the local factor $\Phi_\varpi(\bm{c},s)$ should  resemble the local factor $\Psi_\varpi(s)$ of $\Psi(s)$.
Let $S_{r,0}(\bm{c})$ be the $r$th coefficient of the Euler product $\Phi(\bm{c},s)/\Psi(s)$; then
\begin{equation}
\label{EQN:define-multiplicative-error-Dirichlet-series}
\sum_{r\in \Omon} \abs{r}^{-s} S_{r,0}(\bm{c})
= \Phi(\bm{c},s)/\Psi(s)
= \zeta_K(s-\tfrac12)^{-1} \cdot \zeta_K(s+\tfrac12) \cdot \Phi(\bm{c},s).
\end{equation}
Since $\mu_K(r)$ is the $r$th coefficient of $\zeta_K(s)^{-1} = \prod_\varpi (1-\abs{\varpi}^{-s})$, it follows upon expanding products that
\begin{equation}
\label{EQN:expand-multiplicative-error-as-convolution}
S_{r,0}(\bm{c})
= \sum_{\substack{d_0,d_1,d_2\in \Omon\\ d_0d_1d_2=r}} \mu_K(d_0) \abs{d_0}^{1/2}
\cdot \abs{d_1}^{-1/2}\cdot S^\natural_{d_2}(\bm{c}).
\end{equation}
Moreover,
the Euler product factorisation $\Phi = (\Phi/\Psi) \cdot \Psi$ implies
\begin{equation}
\label{EQN:decompose-S_c-as-convolution}
S^\natural_r(\bm{c}) = \sum_{\substack{r_0,r_1\in \Omon\\  r_0r_1 = r}}
S_{r_0,0}(\bm{c}) \cdot \phi_K(r_1) \abs{r_1}^{-1/2}.
\end{equation}
We will need some basic properties of $S_{r,0}(\bm{c})$ as a function of $\bm{c}\in \Lambda^\perp$ and $r\in \Omon$.

\begin{lemma}
\label{PROP:S_c,0(n)-nice-algebraic-properties}
The quantity $S_{r,0}(\bm{c})$ is multiplicative in $r$ and only depends on $\bm{c}\bmod{r}$.
Also,
\begin{equation*}
\EE_{\bm{c}\in \Lambda^\perp/r\Lambda^\perp}[S_{r,0}(\bm{c})] = 
\begin{cases}
1 & \text{ if $r=1$,}\\
0 & \text{ otherwise.}
\end{cases}
\end{equation*}
\end{lemma}

\begin{proof}
    Being the convolution of multiplicative functions, the multiplicativity is clear. Moreover, from \eqref{EQN:expand-multiplicative-error-as-convolution} it follows that $S_{r,0}(\bm{c})$ depends only on the list of values $\bm{c} \bmod d$ for divisors $d$ of $r$, which in turn implies that $S_{r,0}(\bm{c})$ only depends on $\bm{c}\bmod r$. Finally, by orthogonality of characters we have 
    \[
    \EE_{\bm{c}\in\Lambda^\perp/r\Lambda^\perp}[S^\natural_r(\bm{c})]=\phi_K(r)|r|^{-1/2},
    \]
    so that the last assertion follows from the formal identity
    \[
    \sum_{r\in\OK^+}|r|^{-s}\EE_{\bm{c}\in \Lambda^\perp/r\Lambda^\perp}[S^\natural_r(\bm{c})]= \Psi(s)\sum_{r\in\OK^+}|r|^{-s}\EE_{\bm{c}\in \Lambda^\perp/r\Lambda^\perp}[S_{r,0}(\bm{c})],
    \]
    which follows from 
\eqref{EQN:define-multiplicative-error-Dirichlet-series}.
\end{proof}

We now provide some bounds for $S_{r,0}(\bm{c})$ for $\bm{c}\in \Lambda^\perp\setminus \set{\bm{0}}$.

\begin{lemma}
\label{Le:UpperBound.S(r,0)}
    Let $r\in \OK^+$, $\bm{c}\in \Lambda^\perp\setminus \set{\bm{0}}$ and let $\ve>0$.
    Then 
$$
S_{r,0}(\bm{c})
\ll_\ve \abs{r}^{1/2+\eps}
    \max_{\substack{d\in \mcal{N}_{\ge}(3)\\ d\mid r}} \frac{\lvert{S^\natural_d(\bm{c})}\rvert}{\abs{d}^{1/2}}.
    $$
    Moreover, if $B(\bm{c})\neq 0$ and $r$ is a square-free element of $\LcG$, then
        $
S_{r,0}(\bm{c})\ll_\ve |r|^{\varepsilon}.
$
\end{lemma}

\begin{proof}
Applying the  triangle inequality in \eqref{EQN:expand-multiplicative-error-as-convolution}, followed by  Lemma~\ref{lem:prime_power} to cube-free divisors of $r$, we obtain
\begin{equation*}
    \abs{S_{r,0}(\bm{c})}
    \ll_\ve \abs{r}^{1/2+\eps}
    \cdot \max_{d\mid r} \frac{\lvert{S^\natural_d(\bm{c})}\rvert}{\abs{d}^{1/2}}
    \ll_\eps \abs{r}^{1/2+2\eps}
    \max_{\substack{d\in \mcal{N}_{\ge}(3)\\  d\mid r}} \frac{\lvert{S^\natural_d(\bm{c})}\rvert}{\abs{d}^{1/2}},
\end{equation*}
which verifies the first part. 
If $\varpi\nmid B(\bm{c})$, then \eqref{EQN:expand-multiplicative-error-as-convolution},
 Proposition~\ref{Prop: LinearSpaceBias} 
 and  Lemma \ref{PROP:S_c,0(n)-nice-algebraic-properties} together
 imply
$
S_{\varpi,0}(\bm{c})
= -\abs{\varpi}^{1/2} + \abs{\varpi}^{-1/2} + S^\natural_\varpi(\bm{c})
\ll 1$.
The second assertion now follows from multiplicativity.
\end{proof}

We shall also require the following estimate for averages of $S_{r,0}(\bm{c})$.
\begin{lemma}
\label{LEM:bound-S_c,0-average-in-terms-of-S}
Let $\ve>0$. 
Let $\bm{c}\in \Lambda^\perp\setminus \set{\bm{0}}$ be such that $B(\bm{c})\neq 0$ and $R\ge 0$.
Then
\begin{equation*}
\sum_{\substack{r\in \Omon\\\abs{r}=\hat R}} \abs{S_{r,0}(\bm{c})}
\ll_\ve (\hat R \norm{\bm{c}})^\eps \cdot \hat R \sum_{\substack{d\in \mcal{N}_{\ge}(2)\\ 
\abs{d}\le \hat R}}
\frac{\lvert{S^\natural_d(\bm{c})}\rvert}{\abs{d}}.
\end{equation*}
\end{lemma}

\begin{proof}
Any $r\in \Omon$ can be written uniquely as $r_1r_2r_3$, where $r_1,r_2,r_3\in \Omon$ are pairwise coprime with
$r_1\in \LcG$ square-free,
$r_2\in \LcB$ square-free and
$r_3$ square-full.
Upon writing $S_{r,0}(\bm{c}) = \prod_{1\le i\le 3} S_{r_i,0}(\bm{c})$, and applying Lemma~\ref{Le:UpperBound.S(r,0)} to each factor, we get after decomposing into $q$-adic intervals
\begin{equation*}
\sum_{
\substack{r\in \Omon\\  \abs{r}=\hat R}} \abs{S_{r,0}(\bm{c})}
\ll_\ve \sum_{R_1+R_2+R_3=R} \hat R^\eps
\hat R_1 \sum_{\substack{r_2\in \LcB \\ \abs{r_2}=\hat R_2}} \hat R_2^{1/2}
\sum_{\substack{r_3,d\in \mcal{N}_{\ge}(2) \\ \abs{r_3}=\hat R_3,\; d\mid r_3}} \abs{r_3/d}^{1/2} \abs{S^\natural_d(\bm{c})}.
\end{equation*}
Since $B(\bm{c})\neq 0$ and $|B(\bm{c})|\ll |\bm{c}|^{O(1)}$, 
 it follows from Lemma~\ref{LEM:count-B-R_c-infty-divisors} that the number of $r_2\in \LcB$ with $|r_2|=\widehat{R}_2$ is $O_\ve((|\bm{c}|\widehat{R}_2)^\varepsilon)$. In particular, we get after summing over the $r_i$ for each fixed $d$, that the right-hand side is
\begin{equation*}
\ll_\ve \sum_{R_1+R_2+R_3=R} (\hat R\norm{\bm{c}})^\eps
\hat R_1 \hat R_2^{1/2}
\sum_{\substack{d\in \mcal{N}_{\ge}(2)\\ \abs{d}\le \hat R_3}} (\hat R_3/\abs{d})^{1/2} \abs{d}^\eps
(\hat R_3/\abs{d})^{1/2} \abs{S^\natural_d(\bm{c})}.
\end{equation*}
Lemma~\ref{LEM:bound-S_c,0-average-in-terms-of-S} follows, since $\hat R_1 \hat R_2^{1/2} \hat R_3^{1/2} \hat R_3^{1/2} \le \hat R_1\hat R_2\hat R_3 = \hat R$.
\end{proof}

\subsection{Linear space extraction}

By a $\FF_q$-linear change of variables, when studying linear spaces $L\in \Upsilon$ we may concentrate on $L=L_0$, as given by \eqref{Eq: TheLinSpace}.
Elements of $L_0^\perp$ are of the shape $(c_1,c_2,c_3,c_1,c_2,c_3)$ for $\bm{c}=(c_1,c_2,c_3)\in \OK^3$.
Motivated by this, we define $\bm{c}^*=(c_1,c_2,c_3,c_1,c_2,c_3)$ for $\bm{c}\in \OK^3$.

Our next result will allow us to extract the main contribution from solutions on linear subspaces to our counting function.
Before stating it, we need some more notation. Given $\bm{x}\in \OK^6$, set 
\[
y_i=x_i+x_{i+3}\quad \text{and}\quad z_i = x_i-x_{i+3}
\]
for $i=1,2,3$. Moreover, we let $M\in \GL_6(\OK)$ be such that $\bm{x}=M(\bm{y},\bm{z})$. We then define the cubic form
\[
\widetilde{F}(\bm{y},\bm{z})= F(\bm{x})=\frac{1}{8}\sum_{i=1}^3\left( (y_i+z_i)^3+(y_i-z_i)^3\right)
\]
and the density 
\[
\sigma_{L_0,w}\coloneqq \int_{K_\infty^3}w(M(\bm{0},\bm{z}))\dd \bm{z}.
\]
Note that $M(\bm{0},\bm{z})=(z_1,z_2,z_3,-z_1,-z_2,-z_3)$, so that $\sigma_{L_0}$ measures the density of $K_\infty$-points on the lattice $\Lambda_0$ with respect to $w$.
\begin{lemma}\label{Le: LinSpaceExtract}
    Let $\bm{c}_0\in \OK^3$ and assume that $r=r_0r_1\in \Omon$. There exists a constant $C\in \NN$ such that if $|r_1|\geq \widehat{C}|P|$, then 
    \[
    \sum_{\bm{c}\in \OK^3}I_r(\theta, \bm{c}^*_0+r_0\bm{c}^*)=\sigma_{L_0,w}\left(\frac{|r_1|}{|P|}\right)^{3}.
    \]
    
\end{lemma}
\begin{proof}
    Upon making the unimodular change of variables $\bm{x}\mapsto (\bm{y}, \bm{z})$ described above, we get that 
    \[
    I_r(\theta, \bm{c}_0^*+r_0\bm{c}^*)= \int_{K_\infty^3}\int_{K_\infty^3}w(M(\bm{y},\bm{z}))\psi\left(\theta P^3\widetilde{F}(\bm{y},\bm{z})+\frac{P\bm{c}\cdot \bm{y}}{r_1}+\frac{P\bm{c}_0\cdot \bm{y}}{r}\right)\dd\bm{y}\dd\bm{z}.
    \]
    Applying another change of variables given by $\bm{y}'=P\bm{y}/r_1$, we obtain 
    \[
    \sum_{\bm{c}\in \OK^3}I_r(\theta, \bm{c}_0^*+r_0\bm{c}^*)=\left(\frac{|r_1|}{|P|}\right)^3\int_{K_\infty^3}\sum_{\bm{c}\in \OK^3}\widehat{g}_{\bm{z}}(\bm{c}) \dd \bm{z},
    \]
    where $\widehat{g}_{\bm{z}}$ is the Fourier transform of the function 
    \[
    g_{\bm{z}}\colon K_\infty^3\to \CC, 
    \quad
    \bm{b}\mapsto w(M(r_1\bm{b}/P,\bm{z}))\psi\left(\theta P^3\widetilde{F}(r_1\bm{b}/P,\bm{z})+\frac{\bm{c}_0\cdot \bm{b}}{r_0}\right).
    \]
    In particular, using Poisson summation, in the form Lemma \ref{lem.poisson summation}, we deduce that
    \[
    \sum_{\bm{c}\in \OK^3}
    \hspace{-0.1cm}
    I_r(\theta, \bm{c}_0^*+r_0\bm{c}^*)=\frac{|r_1|^3}{|P|^3}\sum_{\bm{b}\in\OK^3}\psi\left(\frac{\bm{c}_0\cdot \bm{b}}{r_0}\right)\int_{K_\infty^3}
        \hspace{-0.2cm}
    w(M(r_1\bm{b}/P,\bm{z}))\psi(\theta P^3 \widetilde{F}(r_1\bm{b}/P,\bm{z}))\dd\bm{z}.
    \]
    Suppose that $\bm{b}\neq \bm{0}$. Then, since we may assume $|r_1|\geq \widehat{C}|P|$ for a sufficiently large constant $C$, we must have $w(M(r_1\bm{b}/P,\bm{z}))=0$. It follows that  only the term $\bm{b}=\bm{0}$ contributes. However, we have $\widetilde{F}(\bm{0},\bm{z})=0$ identically in $\bm{z}$, so that we arrive at the statement of the lemma.
\end{proof}

\subsection{Proof of Proposition~\ref{Prop:DualContribution}} We will now combine all the results we have obtained so far to complete the proof of Proposition~\ref{Prop:DualContribution}. 
For $\mathcal{E}\subset \OK^6$, we set
\[
S(\mathcal{E}) = |P|^6\sum_{\bm{c}\in\mathcal{E}\setminus\{\bm{0}\}}
\sum_{\substack{r\in \Omon \\ |r|\leq \widehat{Q}}}|r|^{-5/2}S^\natural_r(\bm{c})I_r(\bm{c}).
\]
Then  $S(\mathcal{S}_0)= E_2(P)$, so that 
Proposition~\ref{Prop:DualContribution} can be rephrased as 
\begin{equation*}
    S(\mathcal{S}_0) = \sum_{L\in \Upsilon}\sum_{\bm{x}\in L\cap \OK^6}w(\bm{x}/P)+O_\ve(|P|^{3-1/4+\varepsilon}).
\end{equation*}
Let us define 
\[
\mathcal{E}_1=\mathcal{S}_0\setminus \bigcup_{L\in \Upsilon}\Lambda^\perp \quad \text{and}\quad \mathcal{E}_2(L) =\{\bm{c}\in \Lambda^\perp\colon B(\bm{c})=0\text{ or }c_1\cdots c_6=0\}, 
\]
where $B\in \OK[x_1,x_2,x_3]$ is the form in the statement of Proposition~\ref{Prop: LinearSpaceBias}.

\begin{lemma}\label{Le: UpperBoundexceptionalc}
    For any $C\geq 1$ and $L\in \Upsilon$, we have
    \[
    \#\{\bm{c}\in \mathcal{E}_1\cup \mathcal{E}_2(L)\colon |\bm{c}|\leq \widehat{C}\} \ll_\ve \widehat{C}^{2+\varepsilon}.
    \]
\end{lemma}
\begin{proof}
   Since $B\in \OK[x_1,x_2,x_3]$ is a non-zero form, we get from Lemma~\ref{LEM:affine-dimension-growth} that the contribution from $\bm{c}\in \mathcal{E}_2(L)$ with $B(\bm{c})=0$ is $O(\widehat{C}^2)$.
    Moreover, any $\bm{c}\in \Lambda^\perp$ takes the shape $(c_1, c_2, c_3, c_1, c_2, c_3)$ for some $(c_1,c_2,c_3)\in \OK^3$.
    From this description it immediately follows that the contribution from $\bm{c}\in \Lambda^\perp$ with $c_1\cdots c_6=0$ is $O(\widehat{C}^2)$. 
    
    For $\bm{c}\in \mathcal{E}_1$, this is already implicit in the proof of Lemma~5.1 of \cite{glas2022question} and so we shall be brief.
    Let $m_{1},\dots, m_{k}\in \OK$ with $1\leq k \leq 6$ be square-free and either monic or with leading coefficient a primitive root of $\FF_q^\times$.
    Suppose that $\bm{c}\in \OK^6$ satisfies $F^*(\bm{c})=0$.
    We then partition $\{1,\dots, 6\}$ into sets $\mcal{I}(1),\dots, \mcal{I}(k)$, where $i\in \mcal{I}(j)$ if and only if $c_i^3=m_jd_i^2$ for some $d_i\in \OK$.
    It then follows from the last display in the proof of Lemma~5.1 of \cite{glas2022question} that the contribution from all $\bm{c}\in \OK^6$ ranging over all permissible $m_1,\dots, m_k$ is $O_\ve(\widehat{C}^{2+\varepsilon})$ unless $k=3$ and $\#\mathcal{I}(1)=\#\mathcal{I}(2)=\#\mathcal{I}(3)=2$. If the latter holds, then the display after (5.1) in the proof of Lemma~5.1 of \cite{glas2022question} gives $d_{i_1}=-d_{i_2}$ for $\mathcal{I}(j)=\{i_1,i_2\}$ and $j=1,2,3$, which in turn implies $c_{i_1}^3=c^3_{i_2}$ and hence $\bm{c}$ lies in $\Lambda^\perp$ for some $L\in \Upsilon$.
\end{proof}
Note that for any distinct $L_1,L_2\in \Upsilon$, we have 
\[
\{\bm{c}\in \Lambda_1^\perp\cap \Lambda_2^\perp\colon |\bm{c}|\leq \widehat{C}\} \ll \widehat{C}^2.
\]
In particular, applying Lemma~\ref{LEM:sparse-bound} 
twice, 
with $\delta=1-\ve$,  
in conjunction with Lemma~\ref{Le: UpperBoundexceptionalc}, we obtain 
\[
S(\mathcal{S}_0) = \sum_{L\in \Upsilon}\sum_{\bm{c}\in \Lambda^\perp \setminus \mathcal{E}_2(L)}\sum_{\substack{r\in \Omon \\ |r|\leq \widehat{Q}}}|r|^{-5/2}S_r^\natural(\bm{c})I_r(\bm{c}) +O_\ve(|P|^{3-1/2+\varepsilon}).
\]

Note that since any $L_1, L_2\in \Upsilon$ are isomorphic under a $\FF_q$-linear map,
we may concentrate on the contribution from $L=L_0$.
In particular, Proposition~\ref{Prop:DualContribution} will follow once we have established 
\begin{equation}
\label{EQN-GOAL:isolate-maximal-linear-subvariety}
\abs{P}^6 \sum_{\bm{c}\in \Lambda_0^\perp\setminus \mcal{E}_2}
\sum_{\substack{r\in \Omon\\ \abs{r}\le \widehat{Q}}}
\abs{r}^{-5/2} S^\natural_r(\bm{c}) I_r(\bm{c})
=\sigma_{L_0, w} \abs{P}^{3}+ O_\ve(\abs{P}^{3-1/4+\eps}),
\end{equation}
where $\mathcal{E}_2=\mathcal{E}_2(L_0)$. Indeed, the only step that needs explaining is the equality
\[
\sum_{\bm{c}\in \Lambda_0}w(\bm{x}/P)=\sigma_{L_0,w}|P|^{3} +O_\ve(|P|^{3-1/4+\varepsilon}).
\]
However, this counting result is standard (with a much better error term) and we omit the proof.

For a real number $U>0$ and a set $\mcal{T}\belongs \OK^6$, let
\begin{equation}\label{DEFN:Sigma_K,X,T}
\Sigma_{<U}(P, \mcal{T})
\defeq \abs{P}^6 \sum_{\bm{c}\in \mcal{T}}
\sum_{\substack{r\in \Omon \\ \abs{r}\le \widehat{Q}}}
\abs{r}^{-5/2}
I_r(\bm{c})
\sum_{\substack{r_0r_1 = r \\ \abs{r_1}<U}} S_{r_0,0}(\bm{c})
\cdot \phi_K(r_1) \abs{r_1}^{-1/2}.
\end{equation}
Similarly we define $\Sigma_{\ge U}(P, \mcal{T})$ by replacing $\abs{r_1}<U$ with $\abs{r_1}\ge U$.
For a parameter $W\geq 1$, to be chosen in due course, \eqref{EQN:decompose-S_c-as-convolution} allows us to rewrite the left-hand side of \eqref{EQN-GOAL:isolate-maximal-linear-subvariety} as
\begin{equation}
    \label{EXPR:decompose-LHS-of-goal-into-3-pieces}
    \Sigma_{<\widehat{Q}/\widehat{W}}(P, \Lambda_0^\perp\setminus \mcal{E}_2)
    + \Sigma_{\ge \widehat{Q}/\widehat{W}}(P, \Lambda_0^\perp)
    - \Sigma_{\ge \widehat{Q}/\widehat{W}}(P, \mcal{E}_2).
\end{equation}
We will complete the proof of Proposition~\ref{Prop:DualContribution} with the following three results. 
\begin{lemma}
\label{LEM:handle-n_1<Y/P}
For any real $W\geq 1$, we have
\begin{equation*}
\Sigma_{<\widehat{Q}/\widehat{W}}(P, \Lambda_0^\perp\setminus \mcal{E}_2)
\ll_\ve \abs{P}^{3+\eps} \widehat{W}^{-1/2}.
\end{equation*}
\end{lemma}

\begin{proof}
Since $I_r(\bm{c})\ne 0$ only for $|\bm{c}|\ll \abs{P}^{1/2}$,
it follows that
$$
\Sigma_{<\widehat{Q}/\widehat{W}}(P, \Lambda_0^\perp\setminus \mcal{E}_2)\ll \abs{P}^6 \sum_{\substack{\bm{c}\in \Lambda_0^\perp\setminus \mcal{E}_2 \\ \norm{\bm{c}}\ll \abs{P}^{1/2}}}\,
\sum_{\substack{\abs{r_0r_1}\le \widehat{Q} \\ \abs{r_1}<\widehat{Q}/\widehat{W}}}
\abs{r_0r_1}^{-5/2}\abs{I_{r_0r_1}(\bm{c})}
\cdot \abs{r_1}^{1/2} \abs{S_{r_0,0}(\bm{c})}.
$$
We now examine an individual $\bm{c}\in \Lambda_0^\perp\setminus \mcal{E}_2$.
Since $c_1\cdots c_6\neq 0$, 
it follows from \eqref{Eq: IntEstimateEndgame} that 
\begin{equation}
\label{INEQ:Heath-Brown-integral-bound}
\abs{P}^3 I_{r}(\bm{c})
\ll \norm{P\bm{c}/r} \prod_{1\le i\le 6} \abs{Pc_i/r}^{-1/2}.
\end{equation}
Recall that any element of $\Lambda_0^\perp$ takes the shape $(c_1,c_2, c_{3},c_1,c_2, c_{3})$. Let us now consider the contribution from those $r_0$ with $|r_0|=\widehat{R}_0$. Upon inserting \eqref{INEQ:Heath-Brown-integral-bound} and then applying Lemma~\ref{LEM:bound-S_c,0-average-in-terms-of-S}, we find that the the contribution to the quantity
$\Sigma_{<\widehat{Q}/\widehat{W}}(P, \Lambda_0^\perp\setminus \mcal{E}_2)$
is
\begin{equation*}
\ll_\ve \sum_{\substack{\bm{c}\in \OK^{3} \cong \Lambda_0^\perp \\ 0<|c_i|\ll |P|^{1/2}}}
\sum_{\substack{\hat R_0\abs{r_1}\le \widehat{Q} \\ \abs{r_1}<\widehat{Q}/\widehat{W}}}
\frac{\abs{P}^{3+\eps} \abs{P}^{-2}\norm{\bm{c}}}{(\hat R_0\abs{r_1})^{1/2} \abs{c_1\cdots c_{3}}}
\cdot \abs{r_1}^{1/2}\cdot \hat R_0 \sum_{\substack{d\in \mcal{N}_{\ge}(2) \\ \abs{d}\le \hat R_0}} \frac{\lvert{S^\natural_d(\bm{c})}\rvert}{\abs{d}}.
\end{equation*}
We also have $S^\natural_{\bm{c}}(d) \ll_\ve \abs{d}^{1/2+\eps} \prod_{1\le i\le 3} \abs{\map{sq}(c_i)}^{1/2}$ by Lemma \ref{lem:lem8.2},
and $$
\sum_{\substack{d\in \mcal{N}_{\ge}(2)\\  \abs{d}\le \hat R_0}} \abs{d}^{-1/2} \ll_\ve \hat R_0^\eps,
$$
 so that the last display is
\begin{equation*}
\ll_\ve \sum_{\substack{\bm{c}\in\OK^{n/2} \\ 0<|c_i|\ll |P|^{1/2}}}\,
\sum_{\abs{r_1}<\widehat{Q}/\widehat{W}}
\abs{P}^{1+\eps} \norm{\bm{c}}
\cdot (\widehat{Q}/\abs{r_1})^{1/2}
\prod_{1\le i\le 3} \frac{\abs{\map{sq}(c_i)}^{1/2}}{\abs{c_i}}.
\end{equation*}
Since $\sum_{0<\abs{c}\le \hat C} \abs{\map{sq}(c)}^{1/2}/\abs{c}\ll_\ve \hat C^\eps$, this becomes
\begin{equation*}
\ll_\ve\sum_{\abs{r_1}<\widehat{Q}/\widehat{W}} \abs{P}^{1+\eps} \abs{P}^{1/2}
\cdot (\widehat{Q}/\abs{r_1})^{1/2}
\ll_\ve \abs{P}^{3/2+\eps} \widehat{Q}^{1/2} (\widehat{Q}/\widehat{W})^{1/2}.
\end{equation*}
Recalling that $\widehat{Q} \asymp \abs{P}^{3/2}$ and that there are $O_\ve(\widehat{Q}^\varepsilon)=O_\ve(|P|^\varepsilon)$ choices for $R_0$ completes the proof.
\end{proof}

In the light of Lemma~\ref{Le: LinSpaceExtract}, we make the choice 
\begin{equation}
\label{EQN:choice-of-P=P_X}
\widehat{W} = \widehat{C}^{-1}|P|^{1/2}. 
\end{equation}
We now turn to the range $\abs{r_1}\ge \widehat{Q}/\widehat{W}$ in \eqref{EXPR:decompose-LHS-of-goal-into-3-pieces}.

\begin{lemma}
We have
\begin{equation*}
\Sigma_{\ge \widehat{Q}/\widehat{W}}(P, \Lambda_0^\perp)
= \sigma_{L_0,w} \abs{P}^{3} (1 + O(\widehat{W}^{-1})).
\end{equation*}
\end{lemma}

\begin{proof}
By the first part of Lemma~\ref{PROP:S_c,0(n)-nice-algebraic-properties},
$S_{r_0,0}(\bm{c})$ only depends on $\bm{c}\bmod{r_0}$.
If $\abs{r_1}\ge \widehat{Q}/\widehat{W} = \widehat{C}\abs{P}$, then after splitting $\bm{c}$ into residue classes, Lemma~\ref{Le: LinSpaceExtract}  implies that
\begin{equation*}
\sum_{\bm{c}\in \Lambda_0^\perp}
S_{r_0,0}(\bm{c})
\cdot I_r(\theta, \bm{c})
= \sigma_{L_0, w}\left(\frac{|r_1|}{|P|}\right)^{3}\sum_{\bm{b}\in \Lambda_0^\perp/r_0\Lambda_0^\perp}
S^\natural_{r_0,0}(\bm{b}).
\end{equation*}
By Lemma~\ref{PROP:S_c,0(n)-nice-algebraic-properties} and the equality $\card{\Lambda_0^\perp/r_0\Lambda_0^\perp} = |r_0|^{3}$, we conclude that
\begin{equation*}
\sum_{\bm{c}\in \Lambda_0^\perp}
S^\natural_{r_0,0}(\bm{c})
\cdot I_{r}(\theta, \bm{c})
= \sigma_{L_0,w}\left(\frac{|r|}{|P|}\right)^{3} \bm{1}_{r_0=1}.
\end{equation*}
We  have therefore established the identity
\[
\Sigma_{\geq \widehat{Q}/\widehat{W}}(P,\Lambda_0^\perp) = \sigma_{L_0,w}|P|^3
\sum_{\substack{r\in \Omon \\ \widehat{Q}/\widehat{W}\leq |r|\leq \widehat{Q}}}\phi_K(r)\int_{|\theta|<|r|^{-1}\widehat{Q}^{-1}}1\dd\theta.
\]
However, by Dirichlet's approximation theorem \eqref{Eq: DirichletDissection} we have 
\begin{align*}
  \sum_{\substack{r\in \Omon \\ \widehat{Q}/\widehat{W}\leq |r|\leq \widehat{Q}}}
  \hspace{-0.2cm}
    \phi_K(r)\int_{|\theta|<|r|^{-1}\widehat{Q}^{-1}}
    \hspace{-0.3cm}
    1\dd\theta & = \sum_{\substack{r\in \Omon \\ |r|\leq \widehat{Q}}}\sideset{}{'}\sum_{|a|<|r|}\int_{|\theta|<|r|^{-1}\widehat{Q}^{-1}}1\dd\theta - \sum_{\substack{r\in \Omon \\ |r|< \widehat{Q}/\widehat{W}}}\phi_K(r)|r|^{-1}\widehat{Q}^{-1}\\
  & = 1 + O(\widehat{W}^{-1}),
\end{align*}
which completes the proof.
\end{proof}

\begin{lemma}
\label{LEM:handle-n_1>=Y/P-for-extra-c's}
We have
$$
\Sigma_{\ge \widehat{Q}/\widehat{W}}(P, \mcal{E}_2)
\ll_\ve \abs{P}^{3-2/3+\eps}.
$$
\end{lemma}

\begin{proof}
The main subtlety here is that we must treat $\bm{c}\neq \bm{0}$ and $\bm{c}=\bm{0}$ separately.
First, given $r\in \Omon$,
we apply the triangle inequality and the trivial divisor estimate to \eqref{EQN:expand-multiplicative-error-as-convolution}, giving 
\begin{equation}
\label{INEQ:loose-end-reduce-S_c,0-to-S}
\sum_{\substack{r_0r_1 = r \\ |r_1|\ge \widehat{Q}/\widehat{W}}}\,
\abs{S_{r_0,0}(\bm{c})}\, \phi_K(r_1)|r_1|^{-1/2}
\ll_\ve \bm{1}_{|r|\geq \widehat{Q}/\widehat{W}}|r|^{1/2+\eps} \sum_{d\mid r} |d|^{-1/2}\abs{S^\natural_{d}(\bm{c})}.
\end{equation}
Inserting \eqref{INEQ:loose-end-reduce-S_c,0-to-S} into the definition \eqref{DEFN:Sigma_K,X,T} of $\Sigma_{\ge \widehat{Q}/\widehat{W}}(X, \mcal{E}_2\setminus \set{\bm{0}})$  and recalling the definition \eqref{EQN:sparse-absolute-contribution-f(S)-to-delta-method} of $f$, we get
\begin{equation*}
\Sigma_{\ge \widehat{Q}/\widehat{W}}(P, \mcal{E}_2\setminus \set{\bm{0}})
\ll_\ve |P|^6\sum_{\bm{c}\in \mcal{E}_2\setminus \set{\bm{0}}}
\sum_{\substack{r\in \Omon \\ |r|\leq \widehat{Q}}} |r|^{-2+\eps} \abs{I_{r}(\bm{c})}\sum_{d\mid r} |d|^{-1/2}\abs{S^\natural_{d}(\bm{c})}
\ll_\ve \widehat{Q}^\eps f(\mcal{E}_2).
\end{equation*}
On combining Lemma~\ref{LEM:sparse-bound} with 
the choice $\delta=1-\ve$ in
Lemma~\ref{Le: UpperBoundexceptionalc}, we deduce that $f(\mathcal{E}_2)\ll_\ve |P|^{3-1/2+\varepsilon}$, which is satisfactory since $\widehat{Q}= |P|^{3/2}$. 
Similarly, \eqref{INEQ:loose-end-reduce-S_c,0-to-S} gives
\begin{equation*}
\Sigma_{\ge \widehat{Q}/\widehat{W}}(P, \set{\bm{0}})
\ll_\ve |P|^6 \sum_{\substack{r\in \Omon \\ |r|\geq \widehat{Q}/\widehat{W}}}
|r|^{-2+\eps} \abs{I_r({\bm{0})}} \sum_{d\mid r} |d|^{-1/2}\abs{S^\natural_d({\bm{0}})},
\end{equation*}
which is $O_\ve(|P|^{3+\varepsilon} (\widehat{Q}/\widehat{W})^{-2/3+\varepsilon})$ by \eqref{Eq: IntEstimateEndgame} and Lemma~\ref{LEM:sum-S_0(n)-trivially} summed over $q$-adic ranges for $|r|$. Since $\hat Q/\hat W\gg |P|$, 
we conclude that 
$|P|^{3}(\widehat{Q}/\widehat{W})^{-2/3 }  \ll |P|^{3 -2/3  }$.
\end{proof}
Recalling our choice for $\widehat{W}$ in \eqref{EQN:choice-of-P=P_X}, we see that \eqref{EQN-GOAL:isolate-maximal-linear-subvariety} follows from inserting Lemmas~\ref{LEM:handle-n_1<Y/P}--\ref{LEM:handle-n_1>=Y/P-for-extra-c's} into \eqref{EXPR:decompose-LHS-of-goal-into-3-pieces}. This finally completes the  proof of Proposition~\ref{Prop:DualContribution}.

\section{Final remarks}\label{SEC:final-remarks}

\subsection{Proof of Theorem \ref{THM:pos-density}}

Let $F(\bm x)=x_1^3+\cdots +x_6^3$ 
and let $w_{\bm{x}_0}$ be the weight function defined in 
Definition \ref{def:our-weight}, associated to an appropriate $\bm{x}_0\in (\FF_q^\times)^6$. Then  we have $$N(P)=
\sum_{\substack{\bm{x}_0
\in (\FF_q^\times)^6\\ F(\bm{x}_0)=0}
}N_F(w_{\bm{x}_0},P),
$$
for any $P\in \OK$,
in the notation of 
\eqref{eq:mint-tea} and 
\eqref{eq:green-tea}.
Let  $L(s,\bm c)$ be the Hasse--Weil $L$-function 
introduced in \S~\ref{SEC:ratios-plus}, for $\bm{c}\in \OK^6$. 
Theorem \ref{THM:pos-density} relies on the Ratios Conjecture for these $L$-functions,
as $\bm{c}$ varies. 
We have seen several forms of the Ratios Conjecture during the course of the paper. The following result is a stronger version of Theorem~\ref{THM:pos-density}, and makes clear the various interdependencies. 

\begin{theorem}
\label{THM:first-goal}
Assume $\cha(\FF_q)>3$.
Then each of the following implies the next:
\begin{enumerate}
\item The Ratios Conjecture \ref{CNJ:(R2o)} (R2), for shifted second moments of $1/L(s,\bm{c})$.
\item Conjecture \ref{CNJ:(R2')} holds.
\item Conjecture \ref{CNJ:(R2'E)} holds.
\item Conjecture \ref{CNJ:(R2'E')} holds.
\item $N(P)\ll |P|^3$ for any $P\in \OK$.
\item If $S\belongs \OK$ has positive lower density, then so does $\set{x^3+y^3+z^3: x,y,z\in S}$.
\item The set $\set{k: x^3+y^3+z^3=k\textnormal{ is soluble in $\Omon$ with $\max\{\abs{x},\abs{y},\abs{z}\} = \abs{k}^{1/3}$}}$ has positive lower density.
\end{enumerate}
\end{theorem}

\begin{proof}
We have already seen during the course of our argument
that  (1)$\Rightarrow$(2), 
(2)$\Rightarrow$(3) and 
 (3)$\Rightarrow$(4). These implications  mostly lie in the world of $L$-functions. 
 
 The implication 
 (4)$\Rightarrow$(5) requires the full force of the function field circle method. 
 Combining 
\eqref{Eq: DeltaMethod} with the notation introduced in 
\eqref{eq:assam-tea}, \eqref{eq:darjeeling-tea} and \eqref{eq:rosehip-tea}, 
we see that
$$
N(P)\leq q^6\max_{\substack{\bm{x}_0
\in (\FF_q^\times)^6\\ F(\bm{x}_0)=0}} \left(M(P)+E_1(P)+E_2(P)\right).
$$
Assuming that 
  Conjecture \ref{CNJ:(R2'E')} holds, 
 we may combine 
Propositions \ref{Prop:nonDualContribution}, 
 \ref{prop:MP} and 
 \ref{Prop:DualContribution} in order to  conclude that 
 $N(P)\ll |P|^3$ for any $P\in \OK$.

 The implication (6)$\Rightarrow$(7) follows by taking $S=\Omon$ in (6).
 Indeed, in $\Omon$ we always have $\deg(x^3+y^3+z^3) = 3\max(\deg(x),\deg(y),\deg(z))$,  since $\cha(\FF_q) > 3$.
 
Finally, we deal with the implication (5)$\Rightarrow$(6).
For any positive integer $d$, let 
$$
U_S(d)= \#\left\{ k\in \OK: |k|\leq q^{3d},  ~\exists \text{ 
$(x,y,z)\in S^3$ such that  $x^3+y^3+z^3=k$}\right\}.
$$
We wish to prove that 
$$
\liminf_{d\to \infty}
\frac{\#U_S(d)}{q^{3d}} \gg_S 1,
$$
where the implied constant is allowed to depend on $S$ and $q$.
Since $S\belongs \OK$ has positive lower density, there exists $A\in \NN$ such that
$$
\liminf_{d\to \infty}
\frac{
\#\{x\in S: |x|< q^d\}}{q^d} > \frac{1}{q^A-1}.
$$
Let $S_e = \{x\in S: |x| = q^e\}$, for any $e\in \ZZ_{\geq 0}$. Then, for all sufficiently large integers $d$, we have 
$$
\sum_{e<d} \#S_e \geq  \frac{q^d}{ q^A - 1}.
$$
Since $\sum_{e < d-A} \#S_e \leq  q^{d-A}$, by trivially bounding $\#S_e$, we obtain
$$
\sum_{d-A \leq  e < d} \#S_e\geq  \frac{q^d}{ q^A - 1} - q^{d-A} = \frac{q^{d-A}}{ q^A - 1}.
$$
Thus, by the pigeonhole principle,  there  exists $e = e(d) \in [d-A, d-1]$ such that 
\begin{equation}\label{eq:chai-tea}
\#S_e \geq  \frac{1}{A} \cdot \frac{q^{d-A}}{q^A - 1}.
\end{equation}
Given $k\in \OK$, we denote by 
 $r_S(k)$ the number of $(x,y,z)\in S^3$ for which $x^3+y^3+z^3=k$
 and $|x|= |y|=|z|= q^e$. Then  it follows from Cauchy--Schwarz that 
$$
\left(\sum_{\substack{
|k|\leq  q^{3e}}} r_S(k)\right)^2\leq 
U_S(e)
\sum_{\substack{|k|\leq  q^{3e}}} r_S(k)^2 \leq U_S(e)N(t^e).
$$
Part (5) implies that  $N(t^e)=O(q^{3e})$. 
  Moreover, on the left  hand side we have 
$$
\sum_{\substack{
|k|\leq q^{3e}}} r_S(k)\geq 
\#S_e^3\gg_A  q^{3d},
$$
by \eqref{eq:chai-tea}.
It now follows that  $U_S(d)\geq U_S(e)\gg_A q^{3d}$, which 
 establishes part (6).
\end{proof}

When  $F=x_1^3+\dots+x_6^3$, 
through  careful use of H\"{o}lder's inequality, 
it would be possible to replace the upper bound in (5) by 
$N_F(w,P) \ll_w \abs{P}^3$ for any $w\in S(K^n_\infty)$.
However, since doing so would make  
the paper needlessly longer, we have decided to omit the details here. 

\subsection{Possible reductions to Ratios}
\label{SEC:detailed-ratio-reductions}

To facilitate a deeper discussion of the Ratios Conjecture, we make the following definition.

\begin{definition}
\label{DEFN:q-restricted-form-of-Ratios}
Let $\mathcal{F}_q$ be a \emph{geometric family} of $L$-functions $L(s,f)$ over $\FF_q(t)$,
with a \emph{gauge function} $D(f)$ approximating the conductor,
in the sense of \cite{sarnak2016families}*{pp.~534--535}.
Let $\mcal{Q} \belongs \ZZ$.
Assume $\mathcal{F}_q$ varies naturally with $q\in \mcal{Q}$, forming a family $\mathcal{F} = (\mathcal{F}_q)_{q\in \mcal{Q}}$.
We say that the \emph{$q$-restricted Ratios Conjecture} holds if for any ratio average over $D(f)\le q^Z$,
one has an asymptotic with power saving $O_q(q^{-\delta Z})$ in the strip
\begin{equation}
\label{INEQ:q-restricted-strip}
\tfrac12+q^{-\delta} \le \Re(s) \le \tfrac12+\delta
\end{equation}
where $\delta>0$ depends only on $\mathcal{F}$, $\mcal{Q}$, and the ratio in question.
\end{definition}

Note that one can also study the Ratios Conjecture in other aspects; see work of Bui--Florea--Keating \cite{bui2021ratios}, and work of Florea \cite{florea2021negative}.
We shall prove in 
Theorem~\ref{THM:second-goal} that 
the $q$-restricted Ratios Conjecture suffices in Theorem \ref{THM:pos-density}, 
if $q$ is large enough. 
The homological stability framework of \cites{bergstrom2023hyperelliptic,MPPRW} might eventually resolve
this particular form of the Ratios Conjecture for sufficiently large values of $q$.
Evidence for this appears in work of Wang
\cite{ratio2024forthcoming},
in the  simpler setting of quadratic Dirichlet $L$-functions.

\begin{theorem}[\cite{ratio2024forthcoming}]
The $q$-restricted Ratios Conjecture holds for quadratic Dirichlet $L$-functions over $\FF_q(t)$,
for odd $q$ large enough in terms of the ratio in question.
\end{theorem}

The following result illustrates 
how a  homological stability proof might impact the work of this paper. 
The specific term $q^{-\delta}$ in \eqref{INEQ:q-restricted-strip} is unimportant; any term of the form $o_{q\to \infty}(1)$ would suffice.

\begin{theorem}
\label{THM:second-goal}
The $q$-restricted Ratios Conjecture for $L(s_1,\bm{c})^{-1} L(s_2,\bm{c})^{-1}$,
in the sense of Definition~\ref{DEFN:q-restricted-form-of-Ratios}, implies (1)--(6)  in 
Theorem \ref{THM:first-goal}
for all large enough $q$.
\end{theorem}

\begin{proof}
Assume the $q$-restricted Ratios Conjecture, for  $L(s_1,\bm{c})^{-1} L(s_2,\bm{c})^{-1}$.
Let $\beta \defeq 2 q^{-\delta}$.
If $q$ and $Z$ are sufficiently large in terms of $\delta$, then
$q^{-\delta} \le \beta + \frac1Z \le \delta$ and $6\beta < \delta$,
so Conjecture~\ref{CNJ:(R2o)} holds with constant $\beta = 2 q^{-\delta}$.
On the other hand, if $Z \ll 1$, then $L(s_1,\bm{c})^{-1} L(s_2,\bm{c})^{-1} \ll_q 1$ for $\norm{\bm{c}} \le q^Z$ by GRH, so Conjecture~\ref{CNJ:(R2o)} holds trivially.
This proves (1) in Theorem \ref{THM:first-goal},
and therefore (1)--(6) too,
for $q\gg 1$.
\end{proof}

\subsection{Future extensions}

First, as mentioned in the introduction,
we plan to prove a density $1$ relative of Theorem~\ref{THM:pos-density} in future work.
Second,
most of the proof of Theorem~\ref{THM:first-goal} works when $\cha(\FF_q)=2$,
but the proof of the following main ingredients would need to be modified:
Lemmas \ref{Le: IntEstimateI} and 
\ref{LEM:new-vanishing-for-small-dual-form},
\cite{hooley1994nonary}*{Lemma 60},
Lemma \ref{lem:VW9.1},
and 
Proposition~\ref{Prop:DualContribution}.
We leave open the challenge of extending Theorem~\ref{THM:first-goal} to $\cha(\FF_q)=2$,
but some general progress towards
Proposition \ref{Prop: LinearSpaceBias} (which is a  key ingredient 
in Proposition \ref{Prop:DualContribution})
and \cite{hooley1994nonary}*{Lemma 60}
can be found in \cite{wang2023dichotomous}*{Theorem 1.2}.

We close by discussing the task of rigorously verifying Conjecture \ref{CNJ:(R2o)} (R2).
Thanks to a uniform homological stability result for braid groups established in \cite{MPPRW},
the Moments Conjecture for quadratic Dirichlet $L$-functions for large fixed odd $q$
is now proven, even with a power saving \cite{bergstrom2023hyperelliptic}.
The next questions would be
(a) whether these techniques can establish the $q$-restricted Ratios Conjecture, say, for negative moments;
and (b) whether they work for more general geometric families in the sense of \cite{sarnak2016families}.

The note \cite{ratio2024forthcoming} affirmatively answers (a).
On the other hand, question (b) is much more technical, so we limit ourselves to a few brief remarks.
The $L$-function of a quadratic field $K(\sqrt{d}) \cong K[y]/(y^2-d)$ of discriminant $d\in \Omon$ can be interpreted in terms of
the corresponding hyperelliptic curve $y^2 = d(t)$ over $\FF_q$.
In higher dimensions the situation is more complicated,
but at least when $F^\ast(\bm{c})\in \OK$ is nearly square-free,
one may use \cite{ratio2024forthcoming}*{Proposition~A.1} to interpret the $L$-function of a smooth cubic threefold $$x_1^3+\dots+x_6^3 = c_1(t)x_1+\dots+c_6(t)x_6 = 0$$ over $K$ in terms of
the corresponding \emph{fibred fourfold} over $\FF_q$.

In \cite{bergstrom2023hyperelliptic}, braid groups arise as the \emph{fundamental groups} of moduli spaces
\begin{equation*}
\map{Conf}_e(\mathbb{A}^1)(\CC) =
\set{d\in \CC[t]\textnormal{ monic}:
\deg(d) < e,\;
\disc_t(d) \ne 0}
\end{equation*}
of hyperelliptic curves.
Moreover, the moduli spaces in question are \emph{aspherical}, allowing for a purely group-theoretic reformulation of the relevant homology groups arising from the Grothendieck--Lefschetz trace formula.
As a first step towards (b), one might try to compute the (likely very rich) fundamental groups of the moduli spaces
\begin{equation*}
X_e \defeq
\set{\cc\in \CC[t]^6:
\deg(c_i) < e,\;
\deg_t(F^\ast(\cc)) = (e-1) \deg{F^\ast},\;
\disc_t(F^\ast(\cc)) \ne 0},
\end{equation*}
to bound their Betti numbers,
and to determine whether $X_e$ is aspherical or not.

An interesting, yet simpler higher-dimensional task for (b)
would be to study $L$-function statistics in
quadratic or cubic twist families of elliptic curves over $\FF_q(t)$.
In this case, the base moduli spaces from \cite{bergstrom2023hyperelliptic} would most likely not need to be changed much, with the main changes occurring for the relevant \emph{local systems} and \emph{gluing maps}.

\end{document}